\documentclass{amsart}
\usepackage[T1]{fontenc}
\usepackage[left=1in, right=1in]{geometry}
\usepackage{graphicx} 
\usepackage{amsmath}
\usepackage{amsthm}
\usepackage{amssymb}
\usepackage{url}
\usepackage{dsfont}
\usepackage{float}
\usepackage{hyperref}
\usepackage[frak=esstix]{mathalpha}
\usepackage{mathtools}
\usepackage[english]{babel}
\DeclarePairedDelimiter{\ceil}{\lceil}{\rceil}
\theoremstyle{plain}
\newtheorem{theorem}{Theorem}[section]

\newtheorem{lemma}[theorem]{Lemma}

\newtheorem{proposition}[theorem]{Proposition}
\theoremstyle{definition}
\newtheorem{definition}[theorem]{Definition}
\newtheorem{remark}{Remark}[section]

\newtheorem*{theorem*}{Theorem}
\newcommand{\interior}[1]{%
  {\kern0pt#1}^{\mathrm{o}}%
}

\newcommand{\spec}{\operatorname{spec}}
\newcommand{\res}{\operatorname{res}}
\newcommand{\ess}{\operatorname{ess}}

\newcommand{\rg}{\operatorname{rg}}

\numberwithin{equation}
{section}

\title[Stable Blow-Up on a Sphere for a Quadratic-Derivative NLW]
{Stable blow-up on a sphere for a quadratic-derivative nonlinear wave equation}

\author[M. Del Pino, O. Gough, and M. Musso]%
{Manuel del Pino, Oliver Gough, and Monica Musso}
\begin{document}
\begin{abstract}
We study finite-time blow-up for the nonlinear wave equation
\begin{equation*}
v_{tt}-\Delta v=|\nabla_x v|^2
\end{equation*}
in dimensions \(n\ge2\), under radial symmetry. For every prescribed radius
$r_0>0$, we construct solutions which blow up in finite time $T>0$ on the sphere
$\{|x|=r_0\}$ with logarithmic Type-I rate. The leading singular dynamics
are governed by ``generalised self-similar'' profiles of the associated
one-dimensional equation, while the radial geometry generates a
curvature correction of size $\mathcal{O}((\frac{T}{r_0})^2)$. A key simplification in our approach is a logarithmic radial correction which removes the first-order radial drift
and reduces the geometry to a decaying inverse-square forcing.

We further prove asymptotic stability of the resulting family under radial
perturbations. A new feature compared with the one-dimensional theory is
that the stable blow-up family is not fully explicit. To overcome this, we develop
spectral and semigroup estimates on an extended light cone, together with
Lipschitz dependence on the modulation parameters for the spectral projections, the stable flow, and the non-explicit correction.
\end{abstract}
\maketitle
\section{Introduction}
We consider the Cauchy problem for the nonlinear wave equation (NLW) with quadratic spatial-derivative nonlinearity in arbitrary spatial dimensions $n \geq 2$,
\begin{equation}
    v_{tt} - \Delta v = |\nabla_x v|^2 \qquad (x,t)\in \mathbb{R}^n\times \mathbb{R}_{\geq 0}
    \label{Nd-equation}
\end{equation}
under radial symmetry, with blow-up occurring away from the origin at radius $r_0>0$. This equation is motivated by effective field theory models in cosmology, where scalar perturbations can lead to nonlinear wave equations with quadratic terms involving spatial derivatives; see, for example, \cite{eckmann2023instabilities, hassani2022new}. Beyond this physical motivation, \eqref{Nd-equation} serves as an archetypal model for a quadratic derivative semilinear wave equation. Classical work of John \cite{john1979blow}, together with Klainerman \cite{klainerman1986null} and Christodoulou's \cite{christodoulou1986global} independent work on the null condition, places \eqref{Nd-equation} for $n=3$ within the small-data global-existence/blow-up dichotomy for quadratic derivative nonlinear wave equations. In particular, the nonlinearity $|\nabla_x v|^2$
fails the classical null condition, so in three space dimensions global existence is not guaranteed even for small initial data. Our concern here, however, is not the general small-data problem; rather, we
give a precise description of a Type-I singularity by constructing
a family of stable radial solutions, smooth in the backward annular light
cone, which diverge logarithmically on any prescribed sphere
\(\{|x|=r_0\}\) as \(t\uparrow T\).

More specifically, within the broad literature on singularity formation for nonlinear wave equations, we focus on questions of blow-up profile, blow-up rate, and stability. Without attempting a comprehensive survey of this vast literature, we mention only a few representative works most relevant to the current analysis. For nonlinear wave equations with power nonlinearities, substantial progress has been made on these questions. In the semilinear power-law setting, a foundational work of Merle and Zaag \cite{merle2007existence} established universality of blow-up profiles in one spatial dimension, which was later extended to radial symmetry outside the origin in \cite{merle2011blow}. There has since been the work \cite{boughrara2026radial} providing an explicit construction which illustrates the behaviour established in \cite{merle2011blow}. The area was then placed within a robust spectral and semigroup framework by Donninger and collaborators, allowing for a systematic analysis of stability for self-similar blow-up in more general settings; for semilinear wave equations see \cite{csobo2024blowup,donninger2017stable,ostermann2024stable} and for related results on wave maps see \cite{donninger2026blowup,donninger2025stable,glogic2025existence}.

By contrast, much less is known for nonlinear wave equations with derivative nonlinearities, particularly in the scalar semilinear setting. Much of the existing literature on derivative nonlinear wave equations concerns lifespan estimates or qualitative properties of blow-up sets, rather than the construction of explicit blow-up profiles and their stability; for recent work on the regularity of the blow-up curve in one-dimensional derivative-type models, see \cite{bchatnia2026regularity}. 

A first result in the direction of the present work was obtained by Speck \cite{speck2020stable}, who proved stable ODE-type blow-up for a class of quasilinear wave equations with derivative nonlinearities, whose framework can be adapted to any spatial dimension. Amongst this class, the author considered 
\begin{equation}\label{eq:Specks_eq}
    u_{tt} - \frac{1}{1+u_t}\Delta u = (u_t)^2
\end{equation}
and proved stability of logarithmic ODE-type blow-up, namely solutions of order $-\log(T-t)$ driven by the Riccati mechanism $u_{tt}=(u_t)^2$.
As noted in Section 1 of \cite{speck2020stable}, for the associated semilinear equation to \eqref{eq:Specks_eq}, without the weight in front of the Laplacian, the generic mechanism remains open. Nevertheless, in a recent work of the second author \cite{gough2025stable}, the existence and stability of a family of non-trivial Type-I blow-up solutions was established for the one-dimensional equation \eqref{eq:time_deriv_1_d}, which we discuss below.  

Indeed, for the semilinear setting, comparable results have appeared only very recently. Of particular relevance is the work of Ghoul, Liu, and Masmoudi \cite{ghoul2025blow} on the one-dimensional analogue of \eqref{Nd-equation}
\begin{equation}\label{eq:1d_masmoudi}
u_{tt}-u_{xx} = (u_x)^2.
\end{equation}
They proved non-existence of smooth self-similar blow-up profiles, instead constructing smooth ``generalised self-similar'' singularities consisting of self-similar profiles corrected by logarithmic terms, together with a proof of their asymptotic stability. Due to the derivative nonlinearity, their stability analysis adapts the spectral framework used by Donninger and collaborators to a setting involving non-compact perturbations of the linearised operator. 

In subsequent work, the second author \cite{gough2025stable} extended this method to the one-dimensional time-derivative equation
\begin{equation}\label{eq:time_deriv_1_d}
    u_{tt}-u_{xx} = (u_t)^2
\end{equation}
by constructing analogous Type-I blow-up solutions bifurcating from the Riccati-ODE blow-up solutions associated with $u_{tt}=(u_t)^2$ and proving their asymptotic stability. Afterwards, Liu and Raees \cite{raees2025stable} obtained analogous results for two quadratic time-derivative models, including \eqref{eq:time_deriv_1_d} and the null form
\begin{equation}\label{eq:null_form_1d}
    u_{tt}-u_{xx} = (u_t)^2 - (u_x)^2
\end{equation}
which enjoys small-data global existence but also admits stable ODE-type blow-up for suitably large data. 

In all four works corresponding to equations \eqref{eq:Specks_eq}--\eqref{eq:null_form_1d}, the solution itself exhibits Type-I blow-up of order $-\log(T-t)$, associated with the ODE blow-up $u_{tt}=(u_t)^2$. Whether this behaviour is generic remains an open question. In fact, for the null form \eqref{eq:null_form_1d}, the change of variables $v(x,t)=e^{-u(x,t)}$ reduces \eqref{eq:null_form_1d} to the homogeneous linear wave equation, so the general solution can be written logarithmically as $u(x,t)=-\log(v_{\mathrm{lin}}(x,t))$; this may suggest a generic mechanism.

This should be contrasted with the ``milder'' singularity associated with shock formation, or geometric blow-up, present in many quasilinear hyperbolic models: rather than the solution itself blowing up, which instead remains bounded, certain derivatives become singular as a result of characteristic intersections. Consequently, the solution can often be continued in a weak sense. For an overview of shock formation, and of the rather different methods used there, we refer to \cite{holzegel2016small}. This distinction is relevant here because Speck’s equation \eqref{eq:Specks_eq} exhibits the same logarithmic blow-up as the semilinear models relevant to this work, despite lying in the quasilinear regime in which shock formation is more prevalent.  

The current work extends the above investigation to higher dimensions focusing on the spatial derivative nonlinearity. We study \eqref{Nd-equation} under radial symmetry, outside the origin, and construct stable blow-up solutions whose \textit{leading-order} behaviour is described by the same generalised self-similar solutions to \eqref{eq:1d_masmoudi} given in \cite{ghoul2025blow} but which now blow up on a sphere. Although we choose to focus on the nonlinearity $|\nabla_x v|^2$ because of its physical significance, the method  also extends, with minor modifications, to $(v_t)^2$ and $(v_t)^2 - |\nabla_x v|^2$.

The contribution of this paper is not only the construction of radial
blow-up solutions on a sphere, but also the stability of the
resulting non-explicit family. Before this, we first exhibit an explicit logarithmic radial correction which reduces the higher-dimensional geometry to a decaying inverse-square forcing, so
that the unknown correction is even smaller, of order $\mathcal{O}((\frac{T}{r_0})^2)$. The
stable object nevertheless contains this correction, which is constructed
implicitly by a Lyapunov--Perron argument. This introduces a difficulty
absent from the explicit stability theory: the modulation
analysis must be carried out around an implicitly constructed background.
Since changing the blow-up time changes the similarity frame, comparing
nearby frames requires control on an interval larger than the unit light
cone. This is why we develop the linearised spectral and semigroup analysis
on
\[
    \mathcal I_{\alpha_0}=(-R_{\alpha_0},R_{\alpha_0}), \qquad R_{\alpha_0}>1. 
\]
\subsection{Main results}
Under radial symmetry, equation \eqref{Nd-equation} reduces to 
\begin{equation}\label{eq:eq_in_rad_sym_main_results_section}
    v_{tt}-v_{rr}-\frac{n-1}{r}v_r = (v_r)^2.
\end{equation}
For blow-up on a sphere \(r=r_0>0\), the leading singular dynamics are
expected to be governed by the one-dimensional model
\begin{equation}\label{eq:1d_main_results_sec}
    v_{tt}-v_{rr}=(v_r)^2.
\end{equation}
In fact, there is a stronger observation. Given $r_0>0$, let $v(r,t) = u(r,t) - \frac{n-1}{2}\log\left(\frac{r}{r_0}\right)$. Equation \eqref{eq:eq_in_rad_sym_main_results_section} then reduces to
\begin{equation*}
    u_{tt} - u_{rr} = (u_r)^2 - \frac{(n-1)(n-3)}{4r^2}.
\end{equation*}

Thus, away from the origin $r=0$, the key insight is that the radial equation can be viewed as a perturbation to the one-dimensional equation \eqref{eq:1d_main_results_sec}. In \cite{ghoul2025blow}, it was shown for  the one-dimensional equation \eqref{eq:1d_main_results_sec} that the relevant singular objects are not smooth exact self-similar profiles, but a five-parameter $(\alpha,\beta,\kappa,x_0,T)$ family of generalised self-similar solutions with logarithmic growth, and with the stable branches corresponding to $\beta = 0$ and $\beta =\infty$.
The stability of, say, $\beta =0$, follows automatically from $\beta=\infty$ by reflective symmetry. In the radial problem \eqref{eq:eq_in_rad_sym_main_results_section}, there is no longer symmetry in $r$, so although we work with $\beta = \infty$, we expect  our argument to apply for $\beta =0$.

We show that the same blow-up mechanism persists in higher dimensions under radial symmetry, with the blow-up set now given by the sphere $\{|x|=r_0\}$, and further prove stability of these new solutions.

However, a more naive substitution of the one-dimensional profile into the radial
equation \eqref{eq:eq_in_rad_sym_main_results_section} reveals \textit{two} geometric scales in similarity variables and the relevant norm: a leading
radial-drift scale of size $\mathcal{O}((\frac{T}{r_0})e^{-\tau})$, and a smaller
scale of size $\mathcal{O}((\frac{T}{r_0})^2e^{-2\tau})$. The logarithmic correction  $-\frac{n-1}{2}\log(\frac{r}{r_0})$ in the ansatz actually absorbs the first scale explicitly, leaving
only the second scale in the equation for the unknown correction. It also reveals the special role of dimension
\(n=3\), where the forcing vanishes and any logarithmically
corrected one-dimensional solution provides a closed-form solution to the radial equation in $\mathbb{R}^{3+1}$.

Before we state the first theorem giving existence of a family of blow-up solutions, we denote the following quantities. Given $\alpha_0>0$, let
\begin{align*}
    R_{\alpha_0} \coloneqq \frac{1+\sqrt{1+\alpha_0}}{2}, \qquad 
    \delta_0 \coloneqq \min\left\{1,\frac{\alpha_0}{4}\right\}, \qquad 
    c_{\alpha_0} \coloneqq \frac{\alpha_0+\delta_0}{\sqrt{1+\alpha_0-\delta_0}-\frac{1+\sqrt{1+\alpha_0}}{2}}.
\end{align*}
For any $r_0>0$, $T>0$, and $R\geq 1$, the one-dimensional extended ($R>1$) backward light cone with vertex $(r_0,T)$ is given by 
\begin{equation*}
    \Gamma_{R}^{(1)}(r_0,T) \coloneqq \{(r,t)\in \mathbb R_{+}\times[0,T):\ |r-r_0|\leq R(T-t)\}.
\end{equation*}
In $\mathbb{R}^n$, $\Gamma_{R}^{(1)}(r_0,T)$ corresponds to the space-time region
\begin{equation*}
    \mathcal{A}_{R}(r_0,T)
:= \bigl\{(x,t) \in\mathbb{R}^n \times [0,T):\bigl||x|-r_0\bigr|\le R(T-t)\}.
\end{equation*}
This set, with $R=1$, now defines the domain of dependence for blow-up dynamics on $\{|x|=r_0\}$.
Finally, we denote each time slice of $\mathcal{A}_{R}(r_0,T)$ as the annuli/spherical shells
\begin{equation*}
    A_{R(T-t)}(r_0)
:= \bigl\{x\in\mathbb R^n:\bigl||x|-r_0\bigr|\le R(T-t)\}. 
\end{equation*}

\begin{theorem}[Existence of a family of radial blow-up solutions on $|x|=r_0$]\label{Existence_of_sol_Main_Theorem}
Fix $n\geq 2$, $r_0>0$ and $\omega_0 \in (0,1)$. For any   
\[
\alpha_0>0,\quad
k \ge \lceil c_{\alpha_0} \rceil + 1,\quad \text{and} \quad 0<\gamma<\omega_0
\]
so that $\omega = \omega_0 - \gamma >0$, there exists $0<\varepsilon<\frac{1}{2R_{\alpha_0}}$ with the following property. For all $T_0>0$ satisfying
\[
0<\frac{T_0}{r_0}<\varepsilon
\]
and $\kappa_0 \in \mathbb{R}$,
there exists a three-parameter family of smooth solutions to \eqref{Nd-equation}, arising from smooth initial data, well-defined in $\mathcal{A}_{R_{\alpha_0}}(r_0,T_0)$, given by 
\begin{equation*}
v_{\alpha_0,\kappa_0,T_0}^{(n)}(x,t) = -\alpha_0\log\left(1-\frac{t}{T_0}\right) -\alpha_0 \log\left(\sqrt{1+\alpha_0}+\frac{|x|-r_0}{T_0-t}\right) - \frac{(n-1)}{2}\log\left(\frac{|x|}{r_0}\right) + \zeta_{\alpha_0,T_0}^{(n)}(|x|,t) + \kappa_0 ,
\end{equation*}
where the correction $\zeta_{\alpha_0,T_0}(r,t)$ satisfies, for all $t\in[0,T_0)$,
\begin{align}
(T_0-t)^{-\frac12+j}\,
\|\partial_r^j \zeta_{\alpha_0,T_0}(\cdot,t)\|_{L^2(B_{R_{\alpha_0}(T_0-t)}(r_0))}
&\lesssim \left(\frac{T_0}{r_0}\right)^2\left(1-\frac{t}{T_0}\right)^{\omega},\label{eq:thm_1.1_estimate_1}
\qquad j=0,\dots,k+1,\\
(T_0-t)^{\frac12+j}\,
\|\partial_r^j \partial_t \zeta_{\alpha_0,T_0}(\cdot,t)\|_{L^2(B_{R_{\alpha_0}(T_0-t)}(r_0))}
&\lesssim \left(\frac{T_0}{r_0}\right)^2\left(1-\frac{t}{T_0}\right)^{\omega},
\qquad j=0,\dots,k.\label{eq:thm_1.1_estimate_2}
\end{align}
Finally, if $n=3$, then $\zeta^{(3)}_{\alpha_0,T_0}\equiv0$, provides a \textit{closed-form} family
\begin{equation*}
v^{(3)}_{\alpha_0,\kappa_0,T_0}(x,t)
=
-\alpha_0\log\left(1-\frac{t}{T_0}\right)
-\alpha_0\log\left(\sqrt{1+\alpha_0}
+\frac{|x|-r_0}{T_0-t}\right)
-\log\left(\frac{|x|}{r_0}\right)
+\kappa_0,
\end{equation*}
smooth in $\mathcal{A}_1(r_0,T_0)$ whenever $\frac{T_0}{r_0}<1$.
\end{theorem}
\begin{remark}[A prescribed radius and small parameter]
The radius \(r_0>0\) is prescribed by the desired blow-up set and is not treated
as a modulation parameter.  Once \(r_0\) is fixed, the relevant 
small parameter is \(\frac{T}{r_0}\), which measures the size of the backward light cone
relative to the radius of the blow-up sphere.
\end{remark}
\begin{remark}[The logarithmic radial correction]
The correction $-\frac{n-1}{2}\log(\frac{r}{r_0})$
is not part of the singular one-dimensional profile. It is an accurate and explicit ansatz which exactly removes the first-order radial drift in the
Laplacian. After this transformation, the radial geometry enters only through
an inverse-square forcing. In similarity variables and the relevant norm, this forcing is
$\mathcal{O}((\frac{T}{r_0})^2 e^{-2\tau})$. Thus, this transformation reveals a sharper perturbative structure of the radial problem and explains why the remaining correction
\(\zeta_{\alpha,T}\) is one order smaller in \(\frac{T}{r_0}\) than the logarithmic radial correction itself.
\end{remark}
\begin{remark}[The dimension \(n=3\)]
Furthermore, the coefficient of the inverse-square forcing vanishes when \(n=3\).
Consequently, in three spatial dimensions the logarithmic radial
transformation reduces the radial equation exactly to the one-dimensional
quadratic spatial-derivative model. Thus the leading one-dimensional
profile, together with this explicit correction, gives
a closed-form family of radial blow-up solutions. In Theorem~\ref{Existence_of_sol_Main_Theorem} one may therefore
take $\zeta_{\alpha,T}^{(3)}\equiv 0$.
\end{remark}
\begin{remark}(Equivalent norms).
    Although we formulated the estimates for $\zeta_{\alpha_0,T_0}$ in the radial coordinate $r$, we note that because $r\sim r_0$, an equivalent norm could be used for the decay estimates \eqref{eq:thm_1.1_estimate_1}-\eqref{eq:thm_1.1_estimate_2} taken over the annuli/shells $A_{R(T-t)}(r_0)$ in $\mathbb{R}^n$. 
\end{remark}
\begin{remark}(Smooth, global in space blow-up solutions).
    By the finite propagation speed for wave equations such as \eqref{Nd-equation}, we may construct \textit{global} in space solutions using a cut-off argument applied to the initial data of $v_{\alpha_0,\kappa_0,T_0}$.
\end{remark}
\begin{remark}(Almost explicit initial data).
    Although \(\zeta_{\alpha_0,T_0}\) is implicit, we construct it so that it arises from initial data lying in the three-dimensional symmetry space generated by the one-dimensional profile, associated to eigenvalues $0$ and $1$. Thus the Cauchy data are explicit up to three coefficients determined by the Lyapunov--Perron correction.
\end{remark}
Theorem~\ref{Existence_of_sol_Main_Theorem} transports the one-dimensional generalised self-similar
blow-up mechanism to the higher-dimensional radial case. The leading
singularity is inherited from the one-dimensional equation, but the radial
geometry produces a higher-dimensional correction. After the 
logarithmic radial transformation, the correction is driven by a decaying
inverse-square forcing and is of order $\mathcal{O}((\frac{T}{r_0})^2)$ in similarity
variables. To the best of our knowledge, this gives the first construction of a stable family of nontrivial blow-up solutions for
a quadratic derivative nonlinear wave equation 
in dimensions \(n\geq 2\). 

The next theorem shows that this family is stable under radial perturbations of the initial data. Specifically, starting from initial data close to one of these higher-dimensional radial blow-up solutions, we show that the perturbed evolution still blows up on $\{|x|=r_0\}$ and, after modulation of the parameters $(\alpha,\kappa,T)$, converges to a nearby member of the same family. 

\begin{theorem}[Asymptotic stability of the radial blow-up solutions]\label{thm:stability}
Fix $n\geq 2$, $r_0>0$ and $\omega_0 \in (0,1)$. For any 
\[
\alpha_0>0,\quad
k \ge \lceil c_{\alpha_0} \rceil + 1,\quad 0<\gamma<\omega_0, \quad \text{and} \quad \kappa_0 \in \mathbb{R},
\]
so that $\omega = \omega_0 - \gamma$, there exists $0<\varepsilon<\frac{1}{2R_{\alpha_0}}$, such that the family of solutions from Theorem \ref{Existence_of_sol_Main_Theorem} has the following property. Given
\begin{align*}
    0<\frac{T_0}{r_0}<\varepsilon
\end{align*}
there exists $\delta_2=\delta_2(T_0)>0$ and $M_2>1$ such that for any $(f,g)\in H_{\text{rad}}^{k+1}(\mathbb{R}^n)\times H_{\text{rad}}^k(\mathbb{R}^n)$ with $\|(f,g)\|_{H^{k+1}(\mathbb{R}^n)\times H^k(\mathbb{R}^n)}\leq \frac{\delta_2}{M_2^2}$, there exist parameters 
\begin{equation*}
        \alpha^{\star} \in [\alpha_0-\frac{\delta_2}{M_2},\alpha_0+\frac{\delta_2}{M_2}], \qquad \kappa^{\star}\in [\kappa_0-\frac{\delta_2}{M_2},\kappa_0+\frac{\delta_2}{M_2}],  \qquad T^{\star}\in [T_0-\frac{\delta_2}{M_2} T_0,T_0+\frac{\delta_2}{M_2} T_0] 
    \end{equation*}
    and a unique solution $v(x,t):\mathcal{A}_{1}(r_0,T^{\star})\to \mathbb{R}$ to \eqref{Nd-equation} with initial data 
    \begin{align*}
        v(x,0) = v_{\alpha_0,\kappa_0,T_0}(x,0) + f(x), \qquad \partial_t v(x,0) = \partial_t v_{\alpha_0,\kappa_0,T_0}(x,0) + g(x) \qquad (x\in A_{T^{\star}}(r_0))
    \end{align*}
    of the form 
    \begin{equation*}
        v(x,t) = v_{\alpha^{\star},\kappa^{\star},T^{\star}}(x,t) + \eta(|x|,t).
    \end{equation*}
Moreover, for all \(0\le t<T^\star\), we have the convergence rates
\[
(T^\star-t)^{-\frac12+j}
\|\partial_r^j\eta(\cdot,t)\|_{L^2(B_{T^\star-t}(r_0))}
\lesssim \delta_2 \left(1-\frac{t}{T^{\star}}\right)^\omega,
\qquad j=0,1,\dots,k+1,
\]
and
\[
(T^\star-t)^{\frac12+j}
\|\partial_r^j\partial_t\eta(\cdot,t)\|_{L^2(B_{T^\star-t}(r_0))}
\lesssim \delta_2 \left(1-\frac{t}{T^{\star}}\right)^\omega,
\qquad j=0,1,\dots,k. 
\]
As a corollary, the same result holds for the fully explicit family $v_{\alpha,\kappa,T}^{(3)}$.
\end{theorem}
\begin{remark}[Stability of a non-explicit blow-up family]
The stable object in Theorem~\ref{thm:stability} is not only the explicit leading-order
one-dimensional profile.  It contains the higher-dimensional correction
\(\zeta_{\alpha,T}\), which is constructed perturbatively and depends on the
modulation parameters.  Thus the theorem proves stability of a genuinely
higher-dimensional blow-up family.  As \(t\to T\), this family
asymptotically approaches the one-dimensional profile.
\end{remark}

\subsection{Structure of the paper}
Section~\ref{section_two} extends the  spectral and semigroup analysis given in \cite{ghoul2025blow} to an enlarged light cone $\mathcal{I}_{\alpha_0}$. The purpose is to derive exponential decay for the semigroup on the stable subspace. This machinery is crucial  for Sections \ref{section_three}--\ref{section_four}.

Section~\ref{section_three} uses this extended-light-cone analysis to construct the family of radial blow-up solutions to \eqref{Nd-equation}, thereby proving Theorem~\ref{Existence_of_sol_Main_Theorem}.

Section \ref{section_four} proves asymptotic stability. We first construct stable perturbations
around the non-explicit background family, then perform a finite-dimensional
modulation argument, and finally translate the result back to the original
radial variables to prove Theorem \ref{thm:stability}.
\subsection{Proof strategy and new difficulties}

The proof has three main ideas.
\begin{itemize}
    \item First, in Section~\ref{section_two}, we extend the spectral and semigroup 
analysis of the one-dimensional linearised operator to an $\alpha_0$-dependent, but fixed, interval
\[
    \mathcal{I}_{\alpha_0}=(-R_{\alpha_0},R_{\alpha_0}), \qquad R_{\alpha_0}>1.
\]
This extension is needed for the later modulation of the blow-up time. Indeed, 
when initial data written in the \(T_0\)-frame are compared with data in the 
\(T\)-frame, the spatial variable is rescaled by
\[
    y_{T_0}=\frac{T}{T_0}y_T.
\]
Thus, if \(T>T_0\), the comparison requires control beyond the unit light-cone. 
This extended analysis gives uniform spectral projections, stable semigroup 
bounds, and Lipschitz dependence on \(\alpha\). One new feature is proving constant rank of the projections in an extended and varying phase space. The constant-rank argument must therefore be carried out uniformly in the
parameter \(\alpha\); the special-parameter reduction available in the unit-cone
setting is no longer sufficient. We address this by a more general ODE analysis in Lemma \ref{Appendix_lemma_eigenfunctions}.
\item Second, in Section~\ref{section_three}, we construct the radial blow-up family.
A direct substitution of the one-dimensional profile into the radial equation
produces a decaying radial-drift term of size $\mathcal{O}(\frac{T}{r_0}e^{-\tau})$.
Rather than carrying this term through the fixed-point and modulation arguments,
we absorb it explicitly by the logarithmic radial correction
\[
    u=v+\frac{n-1}{2}\log\left(\frac r{r_0}\right),
\]
normalised so that the correction vanishes at the blow-up set. This conjugation
isolates the leading curvature contribution and reduces the remaining geometric
forcing to the inverse-square term
$
    \mathcal{O}((\frac{T}{r_0})^2e^{-2\tau}).
$
Consequently, the unknown correction \(\mathbf q_{\alpha,T}\) is one order smaller
in \(\frac{T}{r_0}\) than the explicit logarithmic radial correction. In dimension
\(n=3\), the inverse-square forcing vanishes, yielding a closed-form radial
blow-up family.
In similarity coordinates, the resulting equation is an asymptotically autonomous perturbation of the 
one-dimensional linearised flow. We then use the stable/unstable decomposition derived in Section \ref{section_two}
and a modified Lyapunov--Perron argument to construct a decaying correction 
\(\mathbf{q}_{\alpha,T}\). For the specific family given in the stability theorem we choose the 
free `datum' \(\mathbf{q}_0=\mathbf{0}\), so that \(\mathbf{q}_{\alpha,T}(0)\) lies in a three-dimensional 
symmetry space
\[
    \rg(\mathbf{P}_{\alpha})=\operatorname{span}\{\mathbf{g}_{0,\alpha},\mathbf{f}_{0,\alpha},\mathbf{f}_{1,\alpha}\}.
\]
\item In Section~\ref{section_four}, we prove stability of this non-explicit 
family. Given modulation parameters \((\alpha,\kappa,T)\), the logarithmic normalisation
again cancels the explicit radial drift in the perturbation equation. Thus the
principal part is the one-dimensional operator \(\mathbf L_\alpha\); the remaining
linear perturbation is the decaying term \(\mathbf H_{\alpha,T}(v)\), coming from
the non-explicit correction \(\mathbf q_{\alpha,T}\), together with the quadratic
nonlinearity. These terms are first handled by a modified Duhamel fixed-point 
argument on the stable subspace.

The remaining task is to choose \((\alpha,\kappa,T)\) so that the finite-dimensional
correction vanishes. The initial-data comparison gives leading modulation 
terms together with small errors. One new error is \(\mathbf{h}_{\alpha,T}\), which measures the 
difference between the implicit corrections in the \(T_0\) and \(T\) frames. 
Although \(\mathbf{q}_{\alpha,T}(0)\in \rg(\mathbf{P}_{\alpha})\), its non-explicit nature means that this contribution is best viewed as a small perturbation to an already invertible system. The smallness parameter is \(\frac{T}{r_0}\), which controls the curvature forcing,
the implicit correction \(\mathbf q_{\alpha,T}\), and the frame-comparison error
\(\mathbf h_{\alpha,T}\). At the same time, \(T\) is one of the modulation
parameters. We therefore first choose a final time ceiling \(\overline T\) with
\(\frac{\overline T}{r_0}\) sufficiently small, then fix \(T_0\in(0,\overline T)\), and
only afterwards modulate \(T\) in a two-sided interval around \(T_0\) contained
in \((0,\overline T)\). A Brouwer fixed-point argument then selects 
\((\alpha^\star,\kappa^\star,T^\star)\), forcing the finite-dimensional correction 
to vanish and turning the modified solution into a genuine solution with the 
prescribed initial data.
\end{itemize}

\subsection{Notation}
In this paper we adopt the following notation.
For $r_0\in\mathbb R$ and $R>0$, we denote the open interval centred at $r_0$ by
\[
B_R(r_0):=(r_0-R,\,r_0+R).
\]
Its closure is $\overline{B}_R(r_0)=[r_0-R,\,r_0+R]$. For $r_0\in\mathbb{R}$, $T>0$ and $R\geq 1$ the (extended) past light cone is
\[
\Gamma_{R}(r_0,T):=\{(r,t)\in \mathbb R\times[0,T):\ |r-r_0|\le R(T-t)\}.
\]

For $a,b\in\mathbb R$, we write $a\lesssim b$ if there exists a constant $C>0$ (independent of the variables under consideration) such that $a\le C\,b$. If $C$ may depend on a parameter $p$, we write $a\lesssim_p b$.

If $\Omega\subset\mathbb R$ is a domain, then $C^\infty(\overline{\Omega})$ denotes smooth functions on $\Omega$ with derivatives continuous up to $\partial\Omega$. If $\Omega$ is bounded and $k\in\mathbb N\cup\{0\}$, we use the convention 
\[
\|f\|_{H^k(\Omega)}^2:=\sum_{i=0}^k \|\partial_y^i f\|_{L^2(\Omega)}^2,
\qquad
\|f\|_{\dot H^k(\Omega)}:=\|\partial_y^k f\|_{L^2(\Omega)},
\quad f\in C^\infty(\overline{\Omega}).
\]
The Sobolev space $H^k(\Omega)$ is the completion of $C^\infty(\overline{\Omega})$ with respect to $\|\cdot\|_{H^k(\Omega)}$. For $k\ge0$, we use the product space
\[
\mathcal{H}^k(-1,1):=H^{k+1}(-1,1)\times H^k(-1,1)
\] and for convenience, we use the following equivalent notations 
\[\mathcal{H}^k(-R_{\alpha_0},R_{\alpha_0})\coloneqq \mathcal H_{\alpha_0}^k := H^{k+1}(\mathcal{I}_{\alpha_0})\times H^{k}(\mathcal{I}_{\alpha_0}).
\]
We use boldface for tuples of functions; e.g.
\[
\mathbf f\equiv(f_1,f_2)^{\!\top},\qquad
\mathbf{q}(\tau) = \mathbf q(\tau,\cdot)\equiv\big(q_1(\tau,\cdot),\,q_2(\tau,\cdot)\big)^{\!\top}.
\]
Linear operators acting on such tuples are also written in boldface, e.g. $\mathbf{L}$. 
Moreover $\mathbf{q}^{(i)}$ will denote the i-th component of the tuple $\mathbf{q}$.  
If $\mathbf{L}$ is a closed linear operator on a Banach space $X$, we denote its domain by $\mathcal{D}(\mathbf{L})$, its spectrum by $\spec(\mathbf{L})$, and its point and essential spectrum by $\spec_{p}(\mathbf{L})$ and $\spec_{\text{ess}}(\mathbf{L})$, respectively. The resolvent set is $\res(\mathbf{L}):=\mathbb C\setminus \spec(\mathbf{L})$.
The space of bounded operators on $X$ is denoted by $\mathcal L(X)$. 

When parameters are fixed and the context is clear, for parameter-dependent functions or operators, e.g. $\widetilde{\mathbf{q}}_{\alpha,T}$ or $\mathbf{P}_{\alpha}^{(\alpha_0)}$ we may omit the parameters.

For $c\in \mathbb{R}$ we denote by $\lceil c \rceil$ the smallest integer greater than or equal to $c$. Conversely, $\lfloor c \rfloor$ will denote the largest integer less than or equal to $c$. 
\section{Spectral analysis on an extended light cone}\label{section_two}
This section concerns the one-dimensional problem
\begin{equation}\label{eq:1_d_in_r_section_2}
    u_{tt} - u_{rr} = (u_r)^2
\end{equation}
in the radial coordinate $(r,t)$. Given a space-time blow-up point $(r_0,T)$ with $r_0>0$, we define the standard similarity coordinates 
\begin{align*}
    \tau & \coloneqq \tau_{T} \coloneqq \log(T)-\log(T-t) \\
    y &\coloneqq y_{T} \coloneqq \frac{r-r_0}{T-t}.
\end{align*}
The change of variables $(r,t)\to (y,\tau) $ is a smooth diffeomorphism from the interior of the extended ($R>1$) light cone $\Gamma_{R}(r_0,T)$ onto $(-R,R)\times [0,\infty)$  whose inverse is given by 
\begin{align*}
    r &= r_0 + yTe^{-\tau} \\
    t &= T-Te^{-\tau}.
\end{align*}
In these coordinates, the equation \eqref{eq:1_d_in_r_section_2} reads
\begin{equation}\label{eq:sec_2_1d_eq_in_sim_coordinates}
    (y^2-1)U_{yy}+2yU_{y\tau} + 2y U_y + U_{\tau} + U_{\tau\tau} = (U_y)^2.
\end{equation}
Given $\alpha>0,T>0,r_0>0$, and $\kappa \in \mathbb{R}$ we begin with the family of solutions to \eqref{eq:1_d_in_r_section_2}
\begin{equation*}
    u_{\alpha,\kappa,T,r_0}(r,t)=-\alpha \log\left(1-\frac{t}{T}\right)
       - \alpha \log\left(\sqrt{1+\alpha}+\frac{r-r_0}{T-t}\right)
       + \kappa,
\end{equation*}
which, in similarity coordinates centred at $(r_0,T)$, read
\begin{equation}\label{eq_sol_family_sim_coordinates}
    U_{\alpha,\kappa}(\tau,y) \coloneqq \alpha \tau + \widetilde{U}_{\alpha,\kappa}(y) = \alpha \tau - \alpha \log(\sqrt{1+\alpha}+y)+\kappa.
\end{equation}
Here, the function $\widetilde{U}_{\alpha,\kappa}$ is the self-similar part of the profile. 
Letting $\mathbf{q}(\tau,y)=(q_1(\tau,y),q_2(\tau,y))^{\top} = (q_1, \partial_\tau q_1 + y\partial_y q_1)^{\top}$, the associated linearised operator around the family \eqref{eq_sol_family_sim_coordinates} is 
\begin{equation}
\widetilde{\mathbf L}_\alpha \mathbf{q}
:=
\begin{pmatrix}
-y\partial_y q_1 + q_2\\[2mm]
\partial_{yy}q_1 - \dfrac{2\alpha}{\sqrt{1+\alpha}+y}\,\partial_y q_1 - q_2 - y\partial_y q_2
\end{pmatrix}.
\end{equation}
By the finite propagation speed of \eqref{Nd-equation}, and more generally semilinear wave equations with power nonlinearities in the solution and derivatives, it is sufficient to analyse the dynamics in the light cone $y\in [-1,1]$. Thus, $\alpha>0$ is required to ensure such solutions arise from smooth initial data, remaining smooth and well-defined in the light cone up until the blow-up point. In particular, the condition $\alpha>0$ ensures the nondegeneracy of the logarithm term in $\Gamma_{1}(r_0,T)$. However, once an $\alpha_0>0$ is assigned, a crucial observation is that these solutions are actually defined on the larger ($\alpha_0$-dependent) interval $y\in (-\sqrt{1+\alpha_0},\infty)$. We choose to work on a symmetrical interval so that a uniform constant $R_{\alpha_0}>1$ may describe an extended light cone:
\begin{equation*}
    \Gamma_{R_{\alpha_0}}(r_0,T) = \{(r,t)\in \mathbb{R}\times [0,T) : |r-r_0|\leq R_{\alpha_0}(T-t)\} = \{|y|\leq R_{\alpha_0}\}.
\end{equation*}
Thus, we fix $\alpha_0>0$ and for concreteness choose the `halfway point' between $1$ and $\sqrt{1+\alpha_0}$ and the corresponding extended interval denoted by
\[
R_{\alpha_0}:=\frac{1+\sqrt{1+\alpha_0}}{2},\qquad \mathcal I_{\alpha_0}:=(-R_{\alpha_0},R_{\alpha_0}).
\]
For $\alpha$ close to $\alpha_0$, the profile $U_{\alpha,\kappa}(\tau,y)=\alpha\tau-\alpha\log(\sqrt{1+\alpha}+y)+\kappa$
is still smooth on $\overline{\mathcal I_{\alpha_0}}$ provided such $\alpha$ satisfy $\sqrt{1+\alpha}>R_{\alpha_0}$.
Thus, for concreteness we henceforth fix the following notation for $\delta_0$ as 
\[
\delta_0:=\min\left\{1,\frac{\alpha_0}{4}\right\},
\qquad \alpha\in [\alpha_0-\delta_0,\alpha_0+\delta_0].
\]
This guarantees that for $\alpha \in [\alpha_0-\delta_0,\alpha_0+\delta_0]$ we have $\alpha\geq \frac{3}{4}\alpha_0>0$ and, by the basic inequality
\begin{equation*}
    \sqrt{1+\frac{3\alpha_0}{4}} > \frac{1+\sqrt{1+\alpha_0}}{2}, \qquad \text{for all} \quad \alpha_0>0
\end{equation*}
then indeed $\sqrt{1+\alpha}>R_{\alpha_0}$ for $\alpha \in [\alpha_0-\delta_0,\alpha_0+\delta_0]$. Hence $\sqrt{1+\alpha}+y \geq k_{\alpha_0}>0$ uniformly for all $y\in \overline{\mathcal{I}_{\alpha_0}}$ and all $\alpha \in [\alpha_0-\delta_0,\alpha_0+\delta_0]$, where $k_{\alpha_0}$ is some fixed constant depending on $\alpha_0$.

In conclusion, given any $\alpha_0>0$ the linearised operator $\mathbf L_\alpha$ will be shown to be well-defined as a densely defined operator
$\mathbf L_\alpha^{(\alpha_0)}:D(\mathbf L_\alpha)\subset \mathcal H^k(\mathcal I_{\alpha_0})\to \mathcal H^k(\mathcal I_{\alpha_0})$
for all $\alpha\in [\alpha_0-\delta_0,\alpha_0+\delta_0]$.
In this section, we carry out a full spectral and semigroup analysis for this operator and extend the analogous results shown in \cite{ghoul2025blow}. This extension will provide us with the machinery to construct solutions to the higher-dimensional equation on a \textit{larger} interval than the light cone, which is crucial in proving stability of such solutions. 
\begin{remark}
    Since the linearised operator is not self-adjoint, we carry out the whole of this section with complex-valued function spaces. However, since the operator $\mathbf{L}_{\alpha}$ has real components, it is clear that if the initial data is real-valued, then so is the solution. 
\end{remark}
Finally, throughout the section, we will keep $\alpha_0>0$ fixed, but this is of course arbitrary and loses no generality. Therefore we will often omit the $\alpha_0$ in $\mathbf{L}_\alpha^{(\alpha_0)}$.

Within this section, there are several new features. Specifically, there are now more boundary contributions in the high-order energy identities. Secondly, compared to the unit light cone we modify the trace term with a weight $1/R_{\alpha_0}$. Thirdly, for the free resolvent construction which we solve explicitly, we have to cross the singularities, which requires a more delicate ODE analysis. We also carry out an ODE analysis to prove non-existence of further generalised eigenfunctions in Lemma \ref{Appendix_lemma_eigenfunctions} which generalises the result of \cite{ghoul2025blow} to any $\alpha>0$ (rather than $\alpha=3$), which crucially requires us to consider $\sqrt{1+\alpha}\notin \mathbb{N}\setminus \{1\}$. Since we had to restrict to $\alpha \in [\alpha_0-\delta_0,\alpha_0+\delta_0]$ earlier in the work, we do not have the freedom to choose such a convenient $\alpha_0>0$.
\subsection{Mode stability}
The first step is to extend the mode stability statement to the larger domain $y \in \mathcal{I}_{\alpha_0}$ which amounts to studying possible \textit{unstable eigenvalues} of the linearised operator ($\mathfrak{Re}\lambda \geq 0$) -- specifically ruling out such eigenvalues except for those induced by symmetries of the equation. To do so, we present the eigenequation for $\alpha \in [\alpha_0-\delta_0,\alpha_0+\delta_0]$
\begin{equation}
    (\lambda^2 + \lambda)\varphi + \left((2\lambda+2)y+\frac{2\alpha}{\sqrt{1+\alpha}+y}\right)\partial_y \varphi + (y^2-1)\partial_{yy}\varphi = 0, \quad y\in [-R_{\alpha_0},R_{\alpha_0}].
    \label{Eigenequation}
\end{equation}
Indeed, if $\mathbf{q}=(q_1,q_2)^\top$ solves $\mathbf{L}_{\alpha}\mathbf{q}=\lambda \mathbf{q}$, then $q_1$ solves \eqref{Eigenequation}. Conversely if $\varphi$ solves \eqref{Eigenequation}, then $\mathbf{\Psi} \coloneqq (\varphi,\lambda \varphi+y\partial_y \varphi)$ solves $\mathbf{L}_{\alpha}\mathbf{\Psi}=\lambda \mathbf{\Psi}$.

Note that \eqref{Eigenequation} has four `regular singular points': $1,-1,-\sqrt{1+\alpha}$ and $\infty$ where the coefficients degenerate. However, as mentioned, by our choice of $\mathcal{I}_{\alpha_0}$ and $\delta_0>0$, then for all $\alpha \in [\alpha_0-\delta_0,\alpha_0+\delta_0]$, $y\in \mathcal{I}_{\alpha_0}$, we have $y>-\sqrt{1+\alpha}$. This means that we avoid the singular point $y=-\sqrt{1+\alpha}$.

Now, to apply Frobenius theory, we apply a change of variables  $\varphi(2z-1)=\phi(z)$ so that \eqref{Eigenequation} reads as the standard-form Heun ODE
\begin{equation}\label{eigen_heun_standard_form}
    \phi'' + \left[\frac{\lambda-\sqrt{1+\alpha}}{z}+\frac{\lambda+\sqrt{1+\alpha}}{z-1}+\frac{2}{z+\frac{\sqrt{1+\alpha}-1}{2}} \right]\phi' +\frac{\lambda^2+\lambda}{z(z-1)}\phi = 0, \quad z\in \left[\frac{-R_{\alpha_0}+1}{2},\frac{R_{\alpha_0}+1}{2} \right].
\end{equation}
Note that this change of variables moves the singular points $-1 \mapsto 0$, $1\mapsto  1$, $\infty \mapsto \infty$, $-\sqrt{1+\alpha}\mapsto -\frac{\sqrt{1+\alpha}-1}{2}$. 
\begin{definition}
    (Mode stability) For $\alpha \in [\alpha_0-\delta_0,\alpha_0+\delta_0]$, the solutions $U_{\alpha,\kappa}$ are mode stable if the existence of a nontrivial smooth function $\varphi \in C^\infty([-R_{\alpha_0},R_{\alpha_0}])$ solving \eqref{Eigenequation} implies that $\lambda \in \{0,1\}$ or $\mathfrak{Re}\lambda<0$.  
\end{definition}
\begin{remark}(Explicit spectral gap)
    We will actually prove a stronger result, where in the analogous statement above, $\mathfrak{Re}\lambda<0$ is replaced by $\mathfrak{Re}\lambda<-1$.
\end{remark}
We now show that mode stability on the larger interval follows easily from the one-dimensional analysis. 
\begin{proposition}[Mode stability]\label{mode_stability}
For $\alpha \in [\alpha_0-\delta_0,\alpha_0+\delta_0]$, $U_{\alpha,\kappa}$ is mode stable. Namely, there are no non-trivial smooth solutions $\varphi \in C^\infty([-R_{\alpha_0},R_{\alpha_0}])$ to \eqref{Eigenequation} for $\mathfrak{Re}\lambda>-1$, except $\lambda\in \{0,1\}$. 
\end{proposition}
\begin{proof}
Consider $\mathfrak{Re}\lambda>-1$. Suppose we have a non-trivial $\varphi\in C^\infty([-R_{\alpha_0},R_{\alpha_0}])$ solving \eqref{Eigenequation} on
$[-R_{\alpha_0},R_{\alpha_0}]$, then its restriction $\varphi|_{[-1,1]}\in C^\infty[-1,1]$ solves \eqref{Eigenequation} on $[-1,1]$. It remains to justify that the solution is also non-trivial on $[-1,1]$ since the mode stability in \cite[Prop.~3.11]{ghoul2025blow} proves that necessarily $\lambda\in \{0,1\}$.\\
Indeed, suppose for contradiction that $\varphi|_{[-1,1]} \equiv 0$, but $\varphi|_{(1,R_{\alpha_0}]}\not\equiv 0$. By the Frobenius analysis of \eqref{eigen_heun_standard_form} in \cite[Prop.~A.1]{ghoul2025blow}, any smooth solution of \eqref{Eigenequation} is analytic at $1$, thus $\varphi|_{[1,1+\varepsilon)}\equiv0$ for some $\epsilon>0$. However, since the ODE is regular in $(1,R_{\alpha_0}]$, the solution must be $0$ in the rest of the interval by uniqueness, which is a contradiction. The argument for $[-R_{\alpha_0},-1)$ is symmetrical.
\end{proof}
\subsection{Semigroup generation}
We now establish a functional framework to go beyond the light cone and prove generation of a $C_0$-semigroup, whose standard theory can be found in \cite{engel2000one}. First we have the \textit{tilde}-linearised operator, which is defined on a \textit{dense} subset of the eventual Sobolev space.

For $\mathbf{q}=(q_1,q_2)^\top$, we consider the linearised evolution in similarity time for $\alpha \in [\alpha_0-\delta_0,\alpha_0+\delta_0]$,
\[
\partial_{\tau} \mathbf{q} = \widetilde{\mathbf L}_\alpha \mathbf{q},
\]
where
\begin{equation}
\widetilde{\mathbf L}_\alpha \mathbf{q}
:=
\begin{pmatrix}
-y\partial_y q_1 + q_2\\[2mm]
\partial_{yy}q_1 - \dfrac{2\alpha}{\sqrt{1+\alpha}+y}\,\partial_y q_1 - q_2 - y\partial_y q_2
\end{pmatrix}.
\label{eq:Ltilde_alpha_Ia}
\end{equation}
We decompose
\begin{equation}
\widetilde{\mathbf L}_\alpha = \widetilde{\mathbf L} + \mathbf L_{\alpha,1},
\label{eq:decomp_Ltilde}
\end{equation}
where the (modified) free wave operator is
\begin{equation}
\widetilde{\mathbf L}\mathbf{q} :=
\begin{pmatrix}
-y\partial_y q_1 + q_2 - q_1(-R_{\alpha_0})\\[2mm]
\partial_{yy}q_1 - q_2 - y\partial_y q_2
\end{pmatrix},
\label{eq:Ltilde_free_Ia}
\end{equation}
and the remaining term is the ``potential''
\begin{equation}
\mathbf L_{\alpha,1}\mathbf{q}:=
\begin{pmatrix}
q_1(-R_{\alpha_0})\\[2mm]
-\dfrac{2\alpha}{\sqrt{1+\alpha}+y}\,\partial_y q_1
\end{pmatrix}.
\label{eq:Lalpha1_Ia}
\end{equation}
\begin{remark}
    Here we `stabilised' the free wave operator by adding a rank-one correction at $y=-R_{\alpha_0}$. As we will see, this choice is tailored to the energy $\langle\cdot,\cdot\rangle_0$ defined below, which yields dissipativity. 
\end{remark}
Now, for $k\ge 0$, we define 
\[
\mathcal{H}^k(-R_{\alpha_0},R_{\alpha_0})\coloneqq \mathcal H_{\alpha_0}^k := H^{k+1}(\mathcal{I}_{\alpha_0})\times H^{k}(\mathcal{I}_{\alpha_0}).
\]
We introduce the following sesquilinear form. For $\mathbf{q}\in C^1[\overline{\mathcal{I}_{\alpha_0}}]\times C[\overline{\mathcal{I}_{\alpha_0}}]$, set
\begin{equation}
\langle \mathbf{q},\tilde{\mathbf{q}}\rangle_{0}
:=\int_{-R_{\alpha_0}}^{R_{\alpha_0}} \partial_y q_1\,\overline{\partial_y \tilde q_1}\,dy
+\int_{-R_{\alpha_0}}^{R_{\alpha_0}} q_2\,\overline{\tilde q_2}\,dy
+ \frac{1}{R_{\alpha_0}}q_1(-R_{\alpha_0})\,\overline{\tilde q_1(-R_{\alpha_0})}.
\label{eq:ip0_alpha}
\end{equation}
For $k\ge 1$, and $\mathbf{q}\in C^{k+1}[\overline{\mathcal{I}_{\alpha_0}}]\times C^{k}[\overline{\mathcal{I}_{\alpha_0}}]$, define
\begin{equation}
\langle \mathbf{q},\tilde{\mathbf{q}}\rangle_{k}
:=\int_{-R_{\alpha_0}}^{R_{\alpha_0}} \partial_y^{k+1} q_1\,\overline{\partial_y^{k+1}\tilde q_1}\,dy
+\int_{-R_{\alpha_0}}^{R_{\alpha_0}} \partial_y^{k} q_2\,\overline{\partial_y^{k}\tilde q_2}\,dy
+\langle \mathbf{q},\tilde{\mathbf{q}}\rangle_{0}.
\label{eq:ipk_alpha}
\end{equation}
We write the corresponding norm as
\[
\|\mathbf{q}\|_{k}^2 := \langle \mathbf{q},\mathbf{q}\rangle_{k}.
\]
We first observe that with the above choice of quadratic form, we have equivalence to the Sobolev norm.
\begin{lemma}[Equivalence of norms]
\label{lem:norm_equivalence_alpha}
For every $k\ge 0$ there exists a constant $C_{k,\alpha_0}\ge 1$ such that for all
$\mathbf{q}\in C^{k+1}[-R_{\alpha_0},R_{\alpha_0}]\times C^{k}[-R_{\alpha_0},R_{\alpha_0}]$,
\[
C_{k,\alpha_0}^{-1}\big(\|q_1\|_{H^{k+1}(\mathcal{I}_{\alpha_0})}+\|q_2\|_{H^{k}(\mathcal{I}_{\alpha_0})}\big)
\le \|\mathbf{q}\|_{k}
\le C_{k,\alpha_0}\big(\|q_1\|_{H^{k+1}(\mathcal{I}_{\alpha_0})}+\|q_2\|_{H^{k}(\mathcal{I}_{\alpha_0})}\big).
\]
Consequently, the completion of $C^{k+1}[-R_{\alpha_0},R_{\alpha_0}]\times C^{k}[-R_{\alpha_0},R_{\alpha_0}]$
under $\|\cdot\|_{k}$ is $\mathcal H_{\alpha_0}^k$.
\end{lemma}

For $\alpha\in [\alpha_0-\delta_0,\alpha_0+\delta_0]$ we define the dense domain
\[
D(\widetilde{\mathbf L}_\alpha)=D(\widetilde{\mathbf L})
:= C^\infty[-R_{\alpha_0},R_{\alpha_0}]\times C^\infty[-R_{\alpha_0},R_{\alpha_0}]
\subset \mathcal H_{\alpha_0}^k,
\]
and view
\[
\widetilde{\mathbf L} : D(\widetilde{\mathbf L})\subset \mathcal H_{\alpha_0}^k\to \mathcal H_{\alpha_0}^k,
\qquad
\widetilde{\mathbf L}_\alpha : D(\widetilde{\mathbf L}_\alpha)\subset \mathcal H_{\alpha_0}^k\to \mathcal H_{\alpha_0}^k
\]
as densely defined linear operators.
The subsequent steps are:
(i) prove dissipativity of $\widetilde{\mathbf L}$ in $\langle\cdot,\cdot\rangle_{k}$,
(ii) prove that the free resolvent $\lambda-\widetilde{\mathbf L}$ has dense range for some $\lambda>-\frac{1}{2}$,
and then apply the Lumer--Phillips theorem on $\left(\mathcal{H}_{\alpha_0}^k,\|\cdot\|_k 
\right)$ to obtain semigroup generation by the closure of $\widetilde{\mathbf L}$.
\begin{lemma}[Dissipativity of the free operator]\label{Dissipativity_free_op}
For every $k\ge 0$ and every $\mathbf q=(q_1,q_2)^\top\in C^{k+1}[-R_{\alpha_0},R_{\alpha_0}]\times C^{k}[-R_{\alpha_0},R_{\alpha_0}]$,
\begin{equation}
\mathfrak{Re}\langle \mathbf{\widetilde{L}}\mathbf q,\mathbf q\rangle_{k}
\le -\frac12 \|\mathbf q\|_{k}^2.
\end{equation}
\end{lemma}
\begin{proof}
    By definition, for $k=0$ 
    \begin{align*}
        \mathfrak{Re}\langle \mathbf{\widetilde{L}}\mathbf q,\mathbf q\rangle_{0} &= \mathfrak{Re}\int_{-R_{\alpha_0}}^{R_{\alpha_0}}\partial_y(-y\partial_y q_1+q_2)\overline{\partial_y q_1}\,dy + \mathfrak{Re}\int_{-R_{\alpha_0}}^{R_{\alpha_0}}(\partial_{yy}q_1 - q_2-y\partial_y q_2)\overline{q}_2\,dy \\
        &+\frac{1}{R_{\alpha_0}}\mathfrak{Re}((-y\partial_y q_1+q_2-q_1)\overline{q}_1)|_{-R_{\alpha_0}} \\
        &= -\frac{1}{2}\left(\|\partial_y q_1\|_{L_{\alpha_0}^2}^2 + \|q_2\|_{L_{\alpha_0}^2}^2\right) + \left[-\frac{y}{2}|\partial_y q_1|^2\right]_{-R_{\alpha_0}}^{R_{\alpha_0}}+\mathfrak{Re}\left[\partial_y q_1 \overline{q}_2 \right]_{-R_{\alpha_0}}^{R_{\alpha_0}} + \left[-\frac{y}{2}|q_2|^2\right]_{-R_{\alpha_0}}^{R_{\alpha_0}} \\
        &+\frac{1}{R_{\alpha_0}}\mathfrak{Re}(-y\partial_y q_1+q_2-q_1)\overline{q}_1)\rvert_{-R_{\alpha_0}} \\
        &= -\frac{1}{2}\left(\|\partial_y q_1\|_{L_{\alpha_0}^2}^2 + \|q_2\|_{L_{\alpha_0}^2}^2+ \frac{1}{R_{\alpha_0}}|q_1(-R_{\alpha_0})|^2\right) -\frac{R_{\alpha_0}}{2}\left|\partial_y q_1(R_{\alpha_0})-\frac{q_2(R_{\alpha_0})}{R_{\alpha_0}}\right|^2 \\
        &- \frac{R_{\alpha_0}^2-1}{2R_{\alpha_0}}\left|q_2(R_{\alpha_0})\right|^2 
        -\frac{R_{\alpha_0}}{2}\left|\partial_y q_1(-R_{\alpha_0})+\frac{q_2(-R_{\alpha_0})}{R_{\alpha_0}} -\frac{q_1(-R_{\alpha_0})}{R_{\alpha_0}}\right|^2 - \frac{R_{\alpha_0}^2-1}{2R_{\alpha_0}}\left|q_2(-R_{\alpha_0})\right|^2 \\
        &\leq -\frac{1}{2}\|\mathbf{q}\|_0^2.
    \end{align*}
    Now for $k\geq 1$, 
    \begin{align*}
        \mathfrak{Re}\langle \mathbf{\widetilde{L}}\mathbf q,\mathbf q\rangle_{k}  &= -\left(k+\frac{1}{2}\right)\left(\|\partial_y^{k+1}q_1\|_{L_{\alpha_0}^2}^2+\|\partial_y^k q_2\|_{L_{\alpha_0}^2}^2\right) - \frac{R_{\alpha_0}}{2}\left|\partial_y^{k+1}q_1(R_{\alpha_0})-\frac{\partial_y^k q_2(R_{\alpha_0})}{R_{\alpha_0}}\right|^2 + \mathfrak{Re}\langle  \mathbf{\widetilde{L}}\mathbf q,\mathbf q\rangle_{0} \\
        & -\frac{R_{\alpha_0}}{2}\left|\partial_y^{k+1}q_1(-R_{\alpha_0})+\frac{\partial_y^k q_2(-R_{\alpha_0})}{R_{\alpha_0}}\right|^2 - \frac{R_{\alpha_0}^2-1}{2R_{\alpha_0}}\left(|\partial_y^k q_2(R_{\alpha_0})|^2 + |\partial_y^k q_2(-R_{\alpha_0})|^2\right) \\
        &\leq -\frac{1}{2}\left(\|\partial_y^{k+1}q_1\|_{L_{\alpha_0}^2}^2+\|\partial_y^k q_2\|_{L_{\alpha_0}^2}^2\right) -\frac{1}{2}\|\mathbf{q}\|_0^2 = -\frac{1}{2}\|\mathbf{q}\|_k^2
    \end{align*}
\end{proof}
We now construct the resolvent $(\lambda-\widetilde{\mathbf L})^{-1}$ for some $\lambda>-\frac{1}{2}$. We choose $\lambda =0$ for simplicity. 
\begin{lemma}[Dense range of $-\widetilde{\mathbf L}$ on $\mathcal{I}_{\alpha_0}$]
\label{lem:dense_range_L_alpha}
Let $k\ge 0$. For every $\mathbf f\in \mathcal{H}_{\alpha_0}^k$
and $\varepsilon>0$ there exists $\mathbf q\in (C^\infty[-R_{\alpha_0},R_{\alpha_0}])^2$ such that
\[
\|-\widetilde{\mathbf L}\mathbf q-\mathbf f\|_{k}<\varepsilon.
\]
In fact, if $\mathbf f\in (C^\infty[-R_{\alpha_0},R_{\alpha_0}])^2$, then there exists a unique
$\mathbf q\in (C^\infty[-R_{\alpha_0},R_{\alpha_0}])^2$ solving $-\widetilde{\mathbf L}\mathbf q=\mathbf f$ on $\mathcal{I}_{\alpha_0}$.
\end{lemma}
\begin{proof}
Since $(C^\infty[-R_{\alpha_0},R_{\alpha_0}])^2$ is dense in $\mathcal H_{\alpha_0}^k$, it suffices to solve
$-\widetilde{\mathbf L}\mathbf q=\mathbf f$ for smooth $\mathbf f=(f_1,f_2)^\top$.

By definition $-\widetilde{\mathbf L}\mathbf q=\mathbf f$ reads
\begin{equation}
\label{eq:system_free_resolvent}
\begin{cases}
y\partial_y q_1 - q_2 + q_1(-R_{\alpha_0})= f_1,\\
q_2 + y\partial_y q_2 - \partial_{yy}q_1 = f_2.
\end{cases}
\end{equation}
Substituting $q_2$ from the top equation of \eqref{eq:system_free_resolvent} yields the scalar ODE
\begin{equation}
\label{eq:scalar_div_form}
\partial_y\!\Big((y^2-1)\partial_y q_1\Big)=F(y)-q_1(-R_{\alpha_0}),
\qquad
F(y):=f_2(y)+y\partial_y f_1(y)+f_1(y).
\end{equation}
Integrating from $-R_{\alpha_0}$ to $y$
\begin{equation*}
    \partial_y q_1(y) = \frac{1}{y^2-1}\left((R_{\alpha_0}^2-1)\partial_y q_1(-R_{\alpha_0}) + \int_{-R_{\alpha_0}}^{y}F(z)-q_1(-R_{\alpha_0})\,dz\right)\coloneqq \frac{N(y)}{y^2-1},
\end{equation*}
we note that $N\in C^\infty[-R_{\alpha_0},R_{\alpha_0}]$. To avoid a singularity at $y=\pm1$ we need $N(\pm1)=0$. Thus, subtracting $N(1)$ from $N(-1)$ yields the first condition
\begin{align}
\label{eq:choice_anchor1}
q_1(-R_{\alpha_0})&=\frac12\int_{-1}^1 F(z)\,dz \iff \int_{-1}^{1}F(z)-q_1(-R_{\alpha_0})dz =0.
\end{align}
Secondly, $N(-1)=0$, say, implies the condition
\begin{align}
\label{eq:choice_anchor2}
    (R_{\alpha_0}^2-1)\partial_y q_1(-R_{\alpha_0}) = -\int_{-R_{\alpha_0}}^{-1}F(z)-q_1(-R_{\alpha_0})dz. 
\end{align}
Using both \eqref{eq:choice_anchor1}--\eqref{eq:choice_anchor2} allows us to simplify $N(y)$ to 
\begin{equation*}
    N(y) = \int_{-1}^{y}F(z)-q_1(-R_{\alpha_0})dz.
\end{equation*}
It now remains to show that $\partial_y q_1$ given by 
\begin{equation}\label{derivative_q_1_candidate}
    \partial_y q_1(y) = \frac{1}{y^2-1}N(y) 
\end{equation}
is smooth on $[-R_{\alpha_0},R_{\alpha_0}]$ and it suffices to show smoothness only at $y=\pm 1$ since $N(y)$ and $1/(y^2-1)$ are regular elsewhere. For $y=-1$, since $F\in C^\infty[-R_{\alpha_0},R_{\alpha_0}]$, we can Taylor expand $N(y)$ up to arbitrary $n\in \mathbb{N}$ around $y=-1$ 
\begin{equation*}
\partial_y q_1(y) = \frac{1}{y-1}\left(N'(-1) + \dots + \frac{1}{n!}(y+1)^{n-1} N^{(n)}(-1)+ \frac{1}{n!(y+1)}\int_{-1}^{y}(y-t)^n N^{(n+1)}(t)\,dt \right).
\end{equation*}
Note that the remainder is smooth at $y=-1$. Thus $\partial_y q_1 \in C^n$ at $y=-1$ for arbitrary $n\in \mathbb{N}$. The argument for $y=1$ is identical. Therefore $q_1\in C^\infty[-R_{\alpha_0},R_{\alpha_0}]$, which implies that $q_2(y) = y\partial_y q_1 + q_1(-R_{\alpha_0}) - f_1 \in C^\infty[-R_{\alpha_0},R_{\alpha_0}]$.
\end{proof}
\begin{proposition}[Semigroup generation by $\mathbf{L}$]\label{semigroup_generation_free_operator}
    For $k\geq 0$, the operator $\widetilde{\mathbf{L}}:\mathcal{D}(\widetilde{\mathbf{L}})\subset \mathcal{H}^k(-R_{\alpha_0},R_{\alpha_0})\to \mathcal{H}^k(-R_{\alpha_0},R_{\alpha_0})$ is densely defined, closable and its closure $\mathbf{L}$ generates a strongly continuous semigroup $\mathbf{S}:[0,\infty)\to \mathcal{L}(\mathcal{H}^k(\mathcal{I}_{\alpha_0}))$ with growth bound 
\begin{equation}\label{growth_bound_free_op}
        \|\mathbf{S}(\tau)\|_{\mathcal{L}(\mathcal{H}_{\alpha_0}^k)}\leq C_{k,\alpha_0}e^{-\frac{1}{2}\tau}
    \end{equation}
    for all $\tau \geq 0$.
\end{proposition}
\begin{proof}
    We verify the conditions of the Lumer Phillips Theorem \cite[Theorem 3.15]{engel2000one} on the Banach space $\left(\mathcal{H}_{\alpha_0}^k,\|\cdot\|_k 
\right)$, with $\|\cdot\|_k$ defined as in  \eqref{eq:ip0_alpha}--\eqref{eq:ipk_alpha}. Indeed, $\mathcal{D}(\widetilde{\mathbf{L}}) = (C^{\infty}[-R_{\alpha_0},R_{\alpha_0}])^2$ is dense in $\mathcal{H}^k(-R_{\alpha_0},R_{\alpha_0})$, so $\widetilde{\mathbf{L}}$ is densely defined. By Lemma \ref{Dissipativity_free_op}, the operator $\widetilde{\mathbf{L}}+\frac{1}{2}$ is dissipative, and by Lemma \ref{lem:dense_range_L_alpha}, $\rg(\lambda-\widetilde{\mathbf{L}})$ is dense in $\mathcal{H}_{\alpha_0}^k$ for $\lambda=0$ (and hence all $\lambda>-\frac{1}{2}$, by dissipativity). Thus the closure $\overline{\widetilde{\mathbf{L}}+\frac{1}{2}}\coloneqq \mathbf{L}+\frac{1}{2}$ generates a contraction semigroup $\mathbf{T}$ by the Lumer Phillips theorem, and the rescaled semigroup $\mathbf{S}(\tau)\coloneqq e^{-\frac{\tau}{2}}\mathbf{T}(\tau)$ defines the semigroup for $\mathbf{L}$ satisfying the $\mathcal{H}^k$--operator bound \eqref{growth_bound_free_op} by the equivalence of norms in Lemma \ref{lem:norm_equivalence_alpha}. 
\end{proof}
\begin{proposition}[Semigroup generation by $\mathbf{L}_{\alpha}^{(\alpha_0)}$]\label{Semigroup_generation_lin_operator}
    For $\alpha \in [\alpha_0-\delta_0,\alpha_0+\delta_0]$, $k\geq 0$, the operator $\mathbf{L}_{\alpha,1}$ is bounded on $\mathcal{H}^k(-R_{\alpha_0},R_{\alpha_0})$. As a corollary, $\mathbf{L}_{\alpha}\coloneqq \mathbf{L}+\mathbf{L}_{\alpha,1}$ generates a strongly continuous semigroup $\mathbf{S}_{\alpha}:[0,\infty)\to \mathcal{L}(\mathcal{H}_{\alpha_0}^k)$ satisfying 
    \begin{equation}
\label{growth_bound_linearised_opp}
\|\mathbf{S}_{\alpha}(\tau)\|_{\mathcal{L}(\mathcal{H}_{\alpha_0}^k)}\leq C_{k,\alpha_0}e^{(-\frac{1}{2}+\|\mathbf{L}_{\alpha,1}\|_{\mathcal{L}(\mathcal{H}_{\alpha_0}^k)})\tau}.
    \end{equation}
    for all $\tau\geq 0$. Moreover, the map $\alpha \mapsto \mathbf{L}_{\alpha}^{(\alpha_0)}$ is Lipschitz on $[\alpha_0-\delta_0,\alpha_0+\delta_0]$; for $\alpha_1, \alpha_2 \in [\alpha_0-\delta_0,\alpha_0+\delta_0]$
    \begin{equation}\label{eq:lipschitz_L_alpha}
        \|\mathbf{L}_{\alpha_1}-\mathbf{L}_{\alpha_2}\|_{\mathcal{L}(\mathcal{H}_{\alpha_0}^k)}\leq K_{k,\alpha_0}|\alpha_1-\alpha_2|.
    \end{equation}
\end{proposition}
\begin{proof}
    Recall that for $\delta_0 \coloneqq \min\{1,\frac{\alpha_0}{4} \}$ and $\alpha \in [\alpha_0-\delta_0,\alpha_0+\delta_0]$, the denominator of the potential, $\partial_y U_{\alpha,\kappa}$, is uniformly non--zero
    \begin{equation*}
        \sqrt{1+\alpha}+y \geq \sqrt{1+\frac{3\alpha_0}{4}} - R_{\alpha_0}>0.
    \end{equation*}
    Thus for 
    $\mathbf{q}\in \mathcal{H}^k(\mathcal{I}_{\alpha_0})$
    \begin{align*}
\|\mathbf{L}_{\alpha,1}\mathbf{q}\|_{\mathcal{H}_{\alpha_0}^k} &= \|q_1(-R_{\alpha_0})\|_{H^{k+1}(\mathcal{I}_{\alpha_0})}+\left\|-\frac{2\alpha}{\sqrt{1+\alpha}+y}\partial_y q_1\right\|_{H^k(\mathcal{I}_{\alpha_0})} \\
    &\leq C_{k,\alpha_0}\left( \|q_1\|_{L^\infty(\mathcal{I}_{\alpha_0})}+
    \left\| \frac{2\alpha}{\sqrt{1+\alpha}+y} \right\|_{W^{k,\infty}(\mathcal{I}_{\alpha_0})} 
    \|\partial_y q_1\|_{H^k(\mathcal{I}_{\alpha_0})}\right)\\
    &\leq C_{k,\alpha_0}\|q_1\|_{H^{k+1}(\mathcal{I}_{\alpha_0})} \leq C_{k,\alpha_0}\|\mathbf{q}\|_{\mathcal{H}^k(\mathcal{I}_{\alpha_0})}.
    \end{align*}
    Thus, by the bounded perturbation theorem \cite[Theorem 1.3]{engel2000one}, $\mathbf{L}_{\alpha}$ also generates a strongly continuous semigroup, with growth bound \eqref{growth_bound_linearised_opp}.

    Now, to prove the Lipschitz estimate \eqref{eq:lipschitz_L_alpha}, by definition of the operator, we use the product estimate
    \begin{align*}
        \|(\mathbf{L}_{\alpha_1}-\mathbf{L}_{\alpha_2})\mathbf{q}\|_{\mathcal{H}_{\alpha_0}^k} = 2\|\partial_y U_{\alpha_1,\kappa}\partial_y q_1 - \partial_y U_{\alpha_2,\kappa}\partial_y q_1\|_{H_{\alpha_0}^k} \lesssim_k \|\partial_y U_{\alpha_1,\kappa}-\partial_y U_{\alpha_2,\kappa}\|_{W^{k,\infty}_{\alpha_0}}\|\mathbf{q}\|_{\mathcal{H}_{\alpha_0}^k}.
    \end{align*}
    Using the fundamental theorem of calculus in the parameter $\alpha>0$, we have for each $y\in[-R_{\alpha_0},R_{\alpha_0}]$
    \begin{equation*}
        \partial_y U_{\alpha_2,\kappa}-\partial_y U_{\alpha_1,\kappa} = \int_{\alpha_1}^{\alpha_2}\partial_a \partial_y U_{a,\kappa}(y)\,da = (\alpha_2-\alpha_1)\int_{0}^{1}\partial_a \partial_y U_{a(s),\kappa}(y)\,ds, \quad a(s) \coloneqq \alpha_1+(\alpha_2-\alpha_1)s.
    \end{equation*}
    Since $a(s)\in[\alpha_0-\delta_0,\alpha_0+\delta_0]$ and $y\in[-R_{\alpha_0},R_{\alpha_0}]$, we have
$\sqrt{1+a(s)}+y\geq k(\alpha_0)>0$; hence $\partial_a\partial_y^{m+1}U_{a,\kappa}$
is continuous and uniformly bounded on this set. Thus we 
    differentiate $m\leq k$ times in $y$ under the integral, commuting derivatives in $a,y$, which gives 
    \begin{equation*}
        \partial_y^m[\partial_y U_{\alpha_2,\kappa}(y)-\partial_y U_{\alpha_1,\kappa}(y)] = (\alpha_2-\alpha_1)\int_{0}^{1}\partial_a \partial_y^{m+1}U_{a(s),\kappa}(y)\,ds.
    \end{equation*}
    Then for all integers $0\leq m\leq k$
    \begin{equation*}
        \|\partial_y^m(\partial_y U_{\alpha_2,\kappa}-\partial_y U_{\alpha_1,\kappa})\|_{L^{\infty}(\mathcal{I}_{\alpha_0})} \leq |\alpha_2-\alpha_1|\sup_{a\in [\alpha_0-\delta_0,\alpha_0+\delta_0]}\|\partial_a \partial_y^{m+1}U_{a,\kappa}(\cdot)\|_{L^\infty(\mathcal{I}_{\alpha_0})}.
    \end{equation*}
    Taking the maximum over all $m\leq k$ yields the $W^{k,\infty}$ norm and the Lipschitz constant $K_{k,\alpha_0}>0$. 
\end{proof}
\subsection{A compact decomposition}
Since the perturbation $\mathbf{L}_{\alpha,1}$ is only a bounded operator on $\mathcal{H}^k(-R_{\alpha_0},R_{\alpha_0})$ and not compact, the goal is to decompose $\mathbf{L}_{\alpha}$ into the sum of a maximally dissipative operator and a finite-rank (compact) operator. This will allow us to identify the unstable and stable parts of the spectrum. Note, however, that $\mathbf{L}_{\alpha}$ does not always exhibit dissipation for every choice of $\alpha_0>0$. Indeed, for all $\alpha\in [\alpha_0-\delta_0,\alpha_0+\delta_0]$
\begin{align*}
    \mathfrak{Re}\langle \mathbf{L}_{\alpha}\mathbf{q},\mathbf{q}\rangle_{\dot{H}^1(\mathcal{I}_{\alpha_0})\times L^2(\mathcal{I}_{\alpha_0})} = -\frac{1}{2}\left(\|\partial_y q_1 \|_{L_{\alpha_0}^2}^2+\|q_2\|_{L_{\alpha_0}^2}^2\right)
    -\frac{R_{\alpha_0}^2-1}{2R_{\alpha_0}}\left(|q_2(R_{\alpha_0})|^2+|q_2(-R_{\alpha_0})|^2\right) \\
    - \frac{R_{\alpha_0}}{2}\left(\left|\partial_y q_1(R_{\alpha_0})-\frac{q_2(R_{\alpha_0})}{R_{\alpha_0}}\right|^2 
    +\left|\partial_y q_1(-R_{\alpha_0})+\frac{q_2(-R_{\alpha_0})}{R_{\alpha_0}} \right|^2\right)  + \mathfrak{Re}\int_{-R_{\alpha_0}}^{R_{\alpha_0}}-\frac{2\alpha}{\sqrt{1+\alpha}+y}\partial_y q_1\overline{q_2}dy \\
    \leq -\frac{1}{2}\left(\|\partial_y q_1 \|_{L_{\alpha_0}^2}^2+\|q_2\|_{L_{\alpha_0}^2}^2\right) + c_{\alpha_0}\left(\|\partial_y q_1\|_{L_{\alpha_0}^2}^2 + \|q_2\|_{L_{\alpha_0}^2}^2\right)
\end{align*}
where we denote as $c_{\alpha_0}$ the uniform (in $\alpha \in [\alpha_0-\delta_0,\alpha_0+\delta_0]$) bound
\begin{equation}\label{c_alpha_0}
     \sup_{\alpha}\left\|\frac{\alpha}{\sqrt{1+\alpha}+y}\right\|_{L_{\alpha_0}^{\infty}} \leq \frac{(\alpha_0+\delta_0)}{\sqrt{1+\alpha_0-\delta_0}-\frac{1+\sqrt{1+\alpha_0}}{2}} = \frac{(\alpha_0+\min\{1,\frac{\alpha_0}{4}\})}{\sqrt{1+\alpha_0-\min\{1,\frac{\alpha_0}{4}\}}-\frac{1+\sqrt{1+\alpha_0}}{2}} \coloneqq c_{\alpha_0}.
\end{equation}
We demonstrate the following important property of this regularity threshold in relation to the one-dimensional analysis. In order to import certain results from the one-dimensional analysis, we need that for $\alpha \in [\alpha_0-\delta_0,\alpha_0+\delta_0]$, there holds $\lceil c_{\alpha_0} \rceil + 1 \geq \lceil \sqrt{1+\alpha} \rceil +2$. Indeed, we will see that our regularity assumption is much stronger. Denoting the denominator as $D$ in the definition of $c_{\alpha_0}$, we recall from the discussion at the beginning of Section \ref{section_two} the positivity 
\begin{align*}
    D\coloneqq \sqrt{1+\alpha_0 - \delta_0} - \frac{1+\sqrt{1+\alpha_0}}{2} \geq \sqrt{1+\frac{3}{4}\alpha_0}-\frac{1+\sqrt{1+\alpha_0}}{2}>0.
\end{align*}
Moreover 
\begin{equation*}
    D<\sqrt{1+\alpha_0}-\frac{1+\sqrt{1+\alpha_0}}{2}<\frac{\sqrt{1+\alpha_0}-1}{2}
\end{equation*}
and thus 
\begin{equation*}
    c_{\alpha_0} = \frac{\alpha_0+\delta_0}{D}>\frac{2\alpha_0}{\sqrt{1+\alpha_0}-1}=2(1+\sqrt{1+\alpha_0}).
\end{equation*}
\begin{remark}
    The constant $c_{\alpha_0}$ will play a crucial role in determining the required regularity of initial data. Indeed, as $\alpha_0$ increases, we require more regularity. For example, in \cite{ghoul2025blow} it was shown that any $H^k(-1,1)$ solution with $k\geq \lceil \sqrt{1+\alpha }\rceil+2$ to the eigenequation \eqref{Eigenequation} is actually smooth. In this section we will also see that dissipativity depends on the regularity exceeding this constant.
\end{remark}
\begin{remark}
    We also remark that making the choice of $\frac{\alpha_0}{4}$ in the definition of $\delta_0>0$ prevents the regularity from becoming absurdly large for $\alpha_0>0$ small. In this paper, we do not concern ourselves with optimising the regularity of initial data, which is a trade-off between the interval $\mathcal{I}_{\alpha_0}$ and the choice of $\delta_0$. 
\end{remark}
Despite the lack of dissipation in $\dot{H}^1 \times L^2$, we will see that by considering $\mathbf{L}_{\alpha}^{(\alpha_0)}$ in Sobolev spaces with regularity $k\geq \lceil c_{\alpha_0} \rceil +1$, we obtain more dissipation. Therefore, to proceed, we state a technical lemma for commutation between $\mathbf{L}_{\alpha}$ and derivatives. 
\begin{lemma}[Commuting with derivatives]\label{Commutation_derivatives}
    For $\alpha\in [\alpha_0-\delta_0,\alpha_0+\delta_0]$ and $k\geq 0$, we have
    \begin{equation*}
        \partial_y^k \mathbf{L}_{\alpha} = \mathbf{L}_{\alpha,k}\partial_y^k + \mathbf{L}_{\alpha,k}'
    \end{equation*}
    where \begin{equation*}
        \mathbf{L}_{\alpha,k}\coloneqq \begin{pmatrix}
            -k -y\partial_y & 1 \\
            \partial_{yy} -\frac{2\alpha}{\sqrt{1+\alpha}+y}\partial_y & -1-k-y\partial_y
        \end{pmatrix}
    \end{equation*}
    and $\mathbf{L}_{\alpha,k}'$ satisfies the pointwise bound 
    \begin{equation*}
        |\mathbf{L}_{\alpha,k}'\mathbf{q}| \lesssim \begin{pmatrix}
            0 \\
            \sum_{j=0}^{k-1}|\partial_y^{k-j+1}U_{\alpha,\kappa}||\partial_y^{j+1} q_1| 
        \end{pmatrix}
        \lesssim_{\alpha_0,k} \begin{pmatrix}
            0 \\ 
            \sum_{j=0}^{k-1}|\partial_y^{j+1}q_1|
        \end{pmatrix}
    \end{equation*}
\end{lemma}
We now introduce the following equivalent Sobolev inner product. For $k\geq 0$, set 
\begin{equation*}
    \langle \langle \Psi,\Psi \rangle \rangle_k \coloneqq (\partial_y^k\Psi,\partial_y^k \Psi)_{L_{\alpha_0}^2}+(\Psi,\Psi)_{L_{\alpha_0}^2}
\end{equation*}
whereas we denote $(\cdot,\cdot)_k$ as the standard $H^k$ inner product with all intermediate derivatives. For tuples $\mathbf{q} = (q_1,q_2)^\top$ we will denote
\begin{equation*}
    \langle \langle \mathbf{q},\widetilde{\mathbf{q}}\rangle \rangle_k \coloneqq \langle \langle q_1, \widetilde{q}_1 \rangle \rangle_{k+1} + \langle \langle q_2,\widetilde{q}_2 \rangle \rangle_k.
\end{equation*}
We now use the following subcoercivity result about the space $H^k$, which has been used for the NLS and the supercritical power-nonlinear wave equation, for which a proof can be found in \cite{merle2022blow,kim2024self}.
\begin{lemma}[Subcoercivity]\label{subcoercivity}
    Given $m\geq 1$, there exists $\epsilon_n>0$ with $\lim_{n\to \infty}\epsilon_n =0$, $\{\Pi_i\}_{1\leq i\leq n}\subset H^{m}(-R_{\alpha_0},R_{\alpha_0})$, and $c_n >0$ such that for all $n\geq 0$ and $\Psi\in H^{m}(-R_{\alpha_0},R_{\alpha_0})$,
    \begin{equation*}
        \epsilon_n \langle \langle \Psi,\Psi\rangle \rangle_m \geq \|\Psi\|_{H_{\alpha_0}^{m-1}}^2 -c_n \sum_{i=1}^{n}|(\Psi,\Pi_i)_{L_{\alpha_0}^2}|^2.
    \end{equation*}
\end{lemma}
\begin{proposition}[Maximal dissipativity]\label{maximal_dissipativity_proposition}
    For all $\alpha \in [\alpha_0-\delta_0,\alpha_0+\delta_0]$, $k\geq \lceil c_{\alpha_0} \rceil + 1$, $\varepsilon>0$ sufficiently small, there exist functions $(\mathbf{\Pi}_{\alpha_0,i})_{1\leq i\leq N}\subset \mathcal{H}^k(-R_{\alpha_0},R_{\alpha_0})$ such that for the finite-rank operator
    \begin{equation*}
        \widehat{\mathbf{P}}_{\alpha_0} \coloneqq \sum_{i=1}^{N}\langle \langle \cdot,\mathbf{\Pi}_{\alpha_0,i}\rangle \rangle_k\mathbf{\Pi}_{\alpha_0,i},
    \end{equation*}
    the modified operator 
    \begin{equation*}
        \widehat{\mathbf{L}}_{\alpha}^{(\alpha_0)} \coloneqq \mathbf{L}_{\alpha}^{(\alpha_0)}-\widehat{\mathbf{P}}_{\alpha_0}
    \end{equation*}
    is dissipative: 
    \begin{equation}\label{maximal_dissipativity_estimate}
        \forall \mathbf{q}\in \mathcal{D}(\mathbf{L}_{\alpha}^{(\alpha_0)}),\quad \mathfrak{Re}\langle \langle -\widehat{\mathbf{L}}_{\alpha}^{(\alpha_0)}\mathbf{q},\mathbf{q}\rangle\rangle_k \geq \left(k+\frac{1}{2}-c_{\alpha_0}-\varepsilon \right)\langle \langle \mathbf{q},\mathbf{q}\rangle \rangle_k 
        \geq \left(\frac{3}{2}-\varepsilon\right)\langle \langle \mathbf{q},\mathbf{q}\rangle \rangle_k
    \end{equation}
    and is maximal: 
    \begin{equation*}
        \lambda - \widehat{\mathbf{L}}_{\alpha}^{(\alpha_0)} \, \text{is surjective for some} \,\, \lambda>0.
    \end{equation*}
\end{proposition}
\begin{proof}
    For $\mathbf{q}\in (C^\infty[-R_{\alpha_0},R_{\alpha_0}])^2$ and $\alpha \in [\alpha_0-\delta_0,\alpha_0+\delta_0]$, by definition of the commutation in Lemma \ref{Commutation_derivatives}
    \begin{align*}
        \langle \langle -\mathbf{L}_{\alpha}\mathbf{q},\mathbf{q}\rangle \rangle_k &= (-\mathbf{L}_{\alpha}\mathbf{q},\mathbf{q})_{\dot{H}_{\alpha_0}^{k+1}\times \dot{H}_{\alpha_0}^k}+(-\mathbf{L}_{\alpha}\mathbf{q},\mathbf{q})_{L_{\alpha_0}^2\times L_{\alpha_0}^2} \\
        &=(-(\partial_y^{k+1}\mathbf{L}_{\alpha}\mathbf{q})_1,\partial_y^{k+1}\mathbf{q})_{L_{\alpha_0}^2} + (-(\partial_y^k\mathbf{L}_{\alpha}\mathbf{q})_2,\partial_y^k \mathbf{q})_{L_{\alpha_0}^2} + (-\mathbf{L}_{\alpha}\mathbf{q},\mathbf{q})_{L_{\alpha_0}^2\times L_{\alpha_0}^2} \\
        &= (-(\mathbf{L}_{\alpha,k+1}\partial_y^{k+1}\mathbf{q})_1,\partial_y^{k+1}q_1)_{L_{\alpha_0}^2} + (-(\mathbf{L}_{\alpha,k}\partial_y^k \mathbf{q})_2,\partial_y^k q_2)_{L_{\alpha_0}^2} \\
        &+(-(\mathbf{L}_{\alpha,k+1}'\mathbf{q})_1,\partial_y^{k+1}q_1)_{L_{\alpha_0}^2} + (-(\mathbf{L}_{\alpha,k}'\mathbf{q})_2,\partial_y^k q_2)_{L_{\alpha_0}^2} + (-\mathbf{L}_{\alpha}\mathbf{q},\mathbf{q})_{L_{\alpha_0}^2\times L_{\alpha_0}^2}.
    \end{align*}
    For the first two terms of the final equality, using integration by parts we bound
    \begin{align*}(-(\mathbf{L}_{\alpha,k+1}\partial_y^{k+1}\mathbf{q})_1,\partial_y^{k+1}q_1)_{L_{\alpha_0}^2} + (-(\mathbf{L}_{\alpha,k}\partial_y^k \mathbf{q})_2,\partial_y^k q_2)_{L_{\alpha_0}^2} = \left(k+\frac{1}{2}\right)(\|\partial_y^{k+1}q_1\|_{L_{\alpha_0}^2}^2+\|\partial_y^{k}q_2\|_{L_{\alpha_0}^2}^2) \\
    -2(\partial_y U_{\alpha,\kappa}\partial_y^{k+1}q_1,\partial_y^k q_2)_{L_{\alpha_0}^2} + \frac{R_{\alpha_0}-1}{4}\left[|\partial_y^{k+1}q_1(R_{\alpha_0})+\partial_y^k q_2(R_{\alpha_0})|^2 + |\partial_y^{k+1}q_1(-R_{\alpha_0})-\partial_y^k q_2(-R_{\alpha_0})|^2\right] \\
    +\frac{R_{\alpha_0}+1}{4}\left[|\partial_y^{k+1}q_1(R_{\alpha_0})-\partial_y^k q_2(R_{\alpha_0})|^2+|\partial_y^{k+1}q_1(-R_{\alpha_0})+\partial_y^k q_2(-R_{\alpha_0})|^2 \right] 
    \end{align*}
    \begin{equation*}
        \geq \left( k+\frac{1}{2} - c_{\alpha_0}\right)(\|\partial_y^{k+1}q_1\|_{L_{\alpha_0}^2}^2+\|\partial_y^{k}q_2\|_{L_{\alpha_0}^2}^2).
    \end{equation*}
    where $c_{\alpha_0}$  is as given in \eqref{c_alpha_0} and appears from bounding 
    \begin{equation*}
        2(\partial_y U_{\alpha,\kappa}\partial_y^{k+1}q_1,\partial_y^k q_2)_{L_{\alpha_0}^2} \leq \sup_{\alpha}\|\partial_y  U_{\alpha,\kappa}\|_{L_{\alpha_0}^\infty}\left(\|\partial_y^{k+1}q_1\|_{L_{\alpha_0}^2}^2 + \|\partial_y^k q_2\|_{L_{\alpha_0}^2}^2 \right). 
    \end{equation*}For the remaining three terms, by Young's inequality and Lemma \ref{Commutation_derivatives}, for $k\geq 1$,
    \begin{align*}
        |(-(\mathbf{L}_{\alpha,k+1}'\mathbf{q})_1,\partial_y^{k+1}q_1)_{L_{\alpha_0}^2}| + |(-(\mathbf{L}_{\alpha,k}'\mathbf{q})_2,\partial_y^k q_2)_{L_{\alpha_0}^2}| + |(-\mathbf{L}_{\alpha}\mathbf{q},\mathbf{q})_{L_{\alpha_0}^2\times L_{\alpha_0}^2}| \\
        \leq \frac{\varepsilon}{2}\|\partial_y^k q_2\|_{L_{\alpha_0}^2}^2 + C(\varepsilon,\alpha_0,k)\left(\|q_1\|_{H_{\alpha_0}^k}^2+\|q_2\|_{H_{\alpha_0}^{k-1}}^2\right).
    \end{align*}
    Collecting terms, 
    \begin{align*}
        \mathfrak{Re}\langle \langle -\mathbf{L}_{\alpha}\mathbf{q},\mathbf{q}\rangle \rangle_k \geq \left(k+\frac{1}{2}-c_{\alpha_0}-\frac{\varepsilon}{2}\right)(\|\partial_y^{k+1}q_1\|_{L_{\alpha_0}^2}^2+\|\partial_y^{k}q_2\|_{L_{\alpha_0}^2}^2) - C(\varepsilon,\alpha_0,k)\left(\|q_1\|_{H_{\alpha_0}^k}^2+\|q_2\|_{H_{\alpha_0}^{k-1}}^2\right).
    \end{align*}
    Now we apply Lemma \ref{subcoercivity} to the latter terms for $q_1$ and $q_2$ respectively and derive 
    \begin{align*}
        \mathfrak{Re}\langle \langle -\mathbf{L}_{\alpha}\mathbf{q},\mathbf{q}\rangle \rangle_k \geq \left(k+\frac{1}{2}-c_{\alpha_0}-\frac{\varepsilon}{2}\right)\|\mathbf{q}\|_{\mathcal{H}_{\alpha_0}^k}^2- \frac{\varepsilon}{2}\|\mathbf{q}\|_{\mathcal{H}_{\alpha_0}^k}^2 \\
        - \widetilde{C}(\varepsilon,\alpha_0,k)\left(\sum_{i=1}^{N_1}(q_1,\Pi_{i}^{(1)})_{L_{\alpha_0}^2}^2 + \sum_{i=1}^{N_2}(q_2,\Pi_{i}^{(2)})_{L_{\alpha_0}^2}^2 \right).
    \end{align*}
    By the Riesz representation theorem, we obtain the existence of $(\mathbf{\Pi}_{i,\alpha_0})_{1\leq i\leq N}\subset \mathcal{H}_{\alpha_0}^k$ such that 
    \begin{align*}
        \langle \langle \mathbf{q},\mathbf{\Pi}_{i,\alpha_0}\rangle \rangle_k &= \sqrt{\widetilde{C}(\varepsilon,\alpha_0,k)}(q_1,\Pi_i^{(1)})_{L_{\alpha_0}^2}, \quad 1\leq i \leq N_1, \\
        \langle \langle \mathbf{q},\mathbf{\Pi}_{i,\alpha_0} \rangle \rangle_k &= \sqrt{\widetilde{C}(\varepsilon,\alpha_0,k)}(q_2,\Pi_i^{(2)})_{L_{\alpha_0}^2}, \quad N_1+1 \leq i \leq N_1+N_2 \coloneqq N.
    \end{align*}
    By defining $\widehat{\mathbf{P}}_{\alpha_0}\coloneqq \sum_{i=1}^{N}\langle \langle \cdot,\mathbf{\Pi}_{i,\alpha_0}\rangle \rangle_k \mathbf{\Pi}_{i,\alpha_0}$ and $\widehat{\mathbf{L}}_{\alpha} = \mathbf{L}_{\alpha} - \widehat{\mathbf{P}}_{\alpha_0}$ we obtain the dissipativity \eqref{maximal_dissipativity_estimate} for $\mathbf{q}\in (C^\infty[\mathcal{I}_{\alpha_0}])^2$.

    It remains to show the same estimate for $\mathbf{q}\in \mathcal{D}(\mathbf{L}_{\alpha})$. Indeed, let $\mathbf{q}\in \mathcal{D}(\mathbf{L}_{\alpha})$, then $\mathbf{L}_{\alpha}\mathbf{q} \in \mathcal{H}_{\alpha_0}^k$. Since the perturbation $\mathbf{L}_{\alpha,1}$ is bounded on $\mathcal{H}_{\alpha_0}^k$, then we also have that $\mathbf{L}\mathbf{q}\in \mathcal{H}_{\alpha_0}^k$. Now using the density of $(C^\infty[\mathcal{I}_{\alpha_0}])^2$ in $\mathcal{H}_{\alpha_0}^k$, we take $\{\mathbf{f}_n\}_{n\in\mathbb{N}}\subset (C^\infty[\mathcal{I}_{\alpha_0}])^2$ with $\|\mathbf{f}_n-(-\mathbf{L}\mathbf{q})\|_{\mathcal{H}_{\alpha_0}^k}\to 0$. However, for each $n\in \mathbb{N}$, Lemma \ref{lem:dense_range_L_alpha} guarantees a unique $\mathbf{q}_n \in (C^\infty[\mathcal{I}_{\alpha_0}])^2$ with $-\mathbf{L}\mathbf{q}_n = \mathbf{f}_n$. 
    
    We claim $\|\mathbf{q}_n -\mathbf{q}\|_{\mathcal{H}_{\alpha_0}^k}\to 0$. Firstly by the dissipativity of the free wave operator in Lemma \ref{Dissipativity_free_op} 
    \begin{equation*}
        \|\mathbf{q}_n-\mathbf{q}_m\|_{\mathcal{H}_{\alpha_0}^k} \lesssim \|\mathbf{L}\mathbf{q}_n - \mathbf{L}\mathbf{q}_m\|_{\mathcal{H}_{\alpha_0}^k} = \|\mathbf{f}_n -\mathbf{f}_m\|_{\mathcal{H}_{\alpha_0}^k} \to 0,
    \end{equation*}
    we see that $\mathbf{q}_n$ is a Cauchy sequence in $\mathcal{H}_{\alpha_0}^k$; we denote its limit as $\mathbf{q}' \in \mathcal{H}_{\alpha_0}^k$. By the closedness of $\mathbf{L}$, we have $-\mathbf{L}\mathbf{q}=-\mathbf{L}\mathbf{q}'$ and by Proposition \ref{semigroup_generation_free_operator}, since $0\in \res(\mathbf{L})$, we may invert $-\mathbf{L}$ to conclude that $\mathbf{q}=\mathbf{q}'$.

    As for the dissipative estimate, by writing $\widehat{\mathbf{L}}_{\alpha}=\mathbf{L}+\mathbf{L}_{\alpha,1}-\widehat{\mathbf{P}}_{\alpha_0}$, noting the boundedness of $\widehat{\mathbf{P}}_{\alpha_0},\mathbf{L}_{\alpha,1}$ in $\mathcal{H}_{\alpha_0}^k$, then $\|\widehat{\mathbf{L}}_{\alpha}\mathbf{q}_n-\widehat{\mathbf{L}}_{\alpha}\mathbf{q}\|_{\mathcal{H}_{\alpha_0}^k}\to 0$, and thus the estimate \eqref{maximal_dissipativity_estimate} holds for $\mathbf{q}\in \mathcal{D}(\mathbf{L}_{\alpha})$. 

    Finally, to show $\widehat{\mathbf{L}}_{\alpha}$ is maximal, i.e. $\lambda-\widehat{\mathbf{L}}_{\alpha}$ is surjective for all $\lambda>0$, we actually only need to show surjectivity for one $\lambda>0$ by the dissipativity proven above. Indeed, by Proposition \ref{Semigroup_generation_lin_operator} we have a finite growth bound $\omega = -\frac{1}{2}+\|\mathbf{L}_{\alpha,1}\|_{\mathcal{L}(\mathcal{H}_{\alpha_0}^k)}$ for all $\alpha \in [\alpha_0-\delta_0,\alpha_0+\delta_0]$. Thus, by the Hille-Yosida Theorem \cite[Theorem 3.5]{engel2000one} we can take $\mathfrak{Re}\lambda>0$ large enough so that $\lambda\in \res(\mathbf{L}_{\alpha})$. We write
    \begin{equation}\label{Neumann_product}
        \lambda -\widehat{\mathbf{L}}_{\alpha} = (\lambda - \mathbf{L}_{\alpha})\left(\mathbf{I}+(\lambda-\mathbf{L}_{\alpha})^{-1}\widehat{\mathbf{P}}_{\alpha_0}\right)
    \end{equation}
    and then note by the boundedness of $\widehat{\mathbf{P}}_{\alpha_0}$ and for $\mathfrak{Re}\lambda>0$ sufficiently large, then 
    \begin{equation*}
        \|(\lambda-\mathbf{L}_{\alpha})^{-1}\widehat{\mathbf{P}}_{\alpha_0}\|_{\mathcal{L}(\mathcal{H}_{\alpha_0}^k)}<1
    \end{equation*}
    so $\left(\mathbf{I}+(\lambda-\mathbf{L}_{\alpha})^{-1}\widehat{\mathbf{P}}_{\alpha_0}\right)^{-1}$ is well defined by a Neumann series. Therefore the product \eqref{Neumann_product} is also invertible, and in particular surjective. 
\end{proof}
\subsection{Spectral analysis of the generator}
We can now give a sufficient description of the spectrum of $\mathbf{L}_{\alpha}^{(\alpha_0)}:\mathcal{D}(\mathbf{L}_{\alpha}^{(\alpha_0)})\subset \mathcal{H}^k(-R_{\alpha_0},R_{\alpha_0})\to \mathcal{H}^k(-R_{\alpha_0},R_{\alpha_0})$.
\begin{proposition}\label{spectral_inclusion_prop}
    For $\alpha \in [\alpha_0-\delta_0,\alpha_0+\delta_0]$, $k\geq \lceil c_{\alpha_0} \rceil + 1$ we have 
    \begin{equation}\label{spectral_inclusion}
        \spec(\mathbf{L}_{\alpha}^{(\alpha_0)}) \subset \{z\in \mathbb{C}:\mathfrak{Re}z \leq -1 \}\cup \{0,1\}
    \end{equation}
    where $\{0,1\}\subset \spec_p(\mathbf{L}_{\alpha}^{(\alpha_0)})$ are the eigenvalues arising from symmetries. The geometric eigenspaces of eigenvalues $0$ and $1$ are spanned, respectively, by $\mathbf{f}_{0,\alpha}$, $\mathbf{f}_{1,\alpha}$ and are given by 
    \begin{align*}
        \mathbf{f}_{0,\alpha}(y) = \begin{pmatrix}
            1 \\
            0
        \end{pmatrix}
        \quad \text{and} \quad 
        \mathbf{f}_{1,\alpha}(y) = \begin{pmatrix}
            \frac{\alpha\sqrt{1+\alpha}}{\sqrt{1+\alpha}+y} \\
            \frac{\alpha(1+\alpha)}{(\sqrt{1+\alpha}+y)^2} 
        \end{pmatrix}
    \end{align*}
    Moreover, there is a generalised eigenfunction $\mathbf{g}_{0,\alpha}$ of the eigenvalue $0$ given by 
    \begin{align*}
        \mathbf{g}_{0,\alpha}(y) = \begin{pmatrix}
            -\log(\sqrt{1+\alpha}+y)-\frac{\alpha}{2\sqrt{1+\alpha}}\frac{1}{\sqrt{1+\alpha}+y} \\
            1-\frac{y}{\sqrt{1+\alpha}+y}+\frac{\alpha}{2\sqrt{1+\alpha}}\frac{y}{(\sqrt{1+\alpha}+y)^2}
        \end{pmatrix}
    \end{align*}
    satisfying $\mathbf{L}_{\alpha}\mathbf{g}_{0,\alpha} = \mathbf{f}_{0,\alpha}$, and $\mathbf{L}_{\alpha}^2 \mathbf{g}_{0,\alpha} = \mathbf{0}$. 
    \end{proposition}
    \begin{proof}
        By Proposition \ref{maximal_dissipativity_proposition}, for $\alpha\in [\alpha_0-\delta_0,\alpha_0+\delta_0]$, $k\geq \lceil c_{\alpha_0} \rceil + 1$ and $\varepsilon \in (0,\frac{1}{2}]$, $\mathbf{I}+\widehat{\mathbf{L}}_{\alpha}:\mathcal{H}^k(-R_{\alpha_0},R_{\alpha_0}) \to \mathcal{H}^k(-R_{\alpha_0},R_{\alpha_0})$ is maximally dissipative. Thus 
        \begin{equation}\label{Spectral_inclusion_L_hat}
            \spec(\widehat{\mathbf{L}}_{\alpha}) \subset \{z\in \mathbb{C}:\mathfrak{Re}z\leq -1\}. 
        \end{equation}
        Since $\mathbf{L}_{\alpha} = \widehat{\mathbf{L}}_{\alpha}+\widehat{\mathbf{P}}_{\alpha_0}$, with $\widehat{\mathbf{P}}_{\alpha_0}$ compact, $\spec_{\ess}(\mathbf{L}_{\alpha})=\spec_{\ess}(\widehat{\mathbf{L}}_{\alpha})\subset \spec(\widehat{\mathbf{L}}_{\alpha})$. We recall from the mode stability in Proposition \ref{mode_stability} that all eigenvalues of $\mathbf{L}_{\alpha}$ in $\{z\in \mathbb{C}:\mathfrak{Re}z >-1 \}$ lie in $\{0,1\}$, which yields \eqref{spectral_inclusion}. 

        A direct computation shows that $\mathbf{f}_{0,\alpha}$, $\mathbf{f}_{1,\alpha}$ satisfy $\mathbf{L}_{\alpha}\mathbf{f}_{0,\alpha} =\mathbf{0}$, and $\mathbf{L}_{\alpha}\mathbf{f}_{1,\alpha} = \mathbf{f}_{1,\alpha}$, as well as $\mathbf{L}_{\alpha}\mathbf{g}_{0,\alpha} = \mathbf{f}_{0,\alpha}$. From the Frobenius analysis in \cite[Prop. A.1]{ghoul2025blow}, for $k\geq \lceil \sqrt{1+\alpha}\rceil +2$ any $H^{k+1}(-1,1)\times H^k(-1,1)$--eigenfunction of $\mathbf{L}_{\alpha}$ is necessarily $C^\infty[-1,1]$ and in fact analytic at $y=\pm 1$. Moreover, the set of locally smooth solutions around $y=1$ is one--dimensional. Since the ODE is regular elsewhere in $[-R_{\alpha_0},R_{\alpha_0}]$, and noting that $c_{\alpha_0}>\sqrt{1+\alpha_0}+2$ for all $\alpha_0>0$, then the geometric eigenspaces of eigenvalues $0$ and $1$ of the operator $\mathbf{L}_{\alpha}:\mathcal{H}_{\alpha_0}^k \to \mathcal{H}_{\alpha_0}^k$ are one-dimensional, and thus spanned by $\mathbf{f}_{0,\alpha}$, $\mathbf{f}_{1,\alpha}$.  
    \end{proof}
Now, for $\alpha\in [\alpha_0-\delta_0,\alpha_0+\delta_0]$, and $k\geq \lceil c_{\alpha_0} \rceil + 1$, we define the spectral projections $\mathbf{P}_{0,\alpha}^{(\alpha_0)}:\mathcal{H}_{\alpha_0}^k\to \mathcal{H}_{\alpha_0}^k$ and $\mathbf{P}_{1,\alpha}^{(\alpha_0)}:\mathcal{H}_{\alpha_0}^k\to \mathcal{H}_{\alpha_0}^k$ associated to eigenvalues $0$ and $1$ respectively:
    \begin{align*}
        \mathbf{P}_{0,\alpha}^{(\alpha_0)} &\coloneqq \frac{1}{2\pi i}\int_{\gamma_0}(z\mathbf{I}-\mathbf{L}_{\alpha}^{(\alpha_0)})^{-1}dz \qquad \gamma_0:[0,1]\to \mathbb{C}, \quad s\mapsto \frac{1}{4}e^{2\pi i s} \\
        \mathbf{P}_{1,\alpha}^{(\alpha_0)} &\coloneqq \frac{1}{2\pi i}\int_{\gamma_1}(z\mathbf{I}-\mathbf{L}_{\alpha}^{(\alpha_0)})^{-1}dz \qquad \gamma_1:[0,1]\to \mathbb{C}, \quad s\mapsto 1+\frac{1}{2}e^{2\pi i s}.
    \end{align*}
    We remark that $\mathbf{P}_{i,\alpha}^{(\alpha_0)}$ are not orthogonal projections because of the non-self-adjoint nature of $\mathbf{L}_{\alpha}^{(\alpha_0)}$.
\begin{proposition}\label{prop:projection_rank}
    For all $\alpha \in [\alpha_0-\delta_0,\alpha_0+\delta_0]$, $k\geq \lceil c_{\alpha_0}\rceil +1$, the projections $\mathbf{P}_{0,\alpha}^{(\alpha_0)}$, $\mathbf{P}_{1,\alpha}^{(\alpha_0)}$ have ranks 2 and 1, respectively. 
\end{proposition}
\begin{proof}
    We first note that for any fixed $\alpha \in [\alpha_0-\delta_0,\alpha_0+\delta_0]$,  $\mathbf{P}_{\lambda,\alpha}$ ($\lambda \in \{0,1\}$) have finite rank since isolated spectrum outside the essential spectrum consists of eigenvalues of finite algebraic multiplicity, hence the Riesz projection has finite rank. 
    
    Secondly, since $\mathbf{L}_{\alpha}$ depends smoothly on the parameter $\alpha$, $\alpha \mapsto \mathbf{P}_{\lambda,\alpha}$ is a continuous map and so the rank of $\mathbf{P}_{\lambda,\alpha}$ is the same for all $\alpha \in [\alpha_0-\delta_0,\alpha_0+\delta_0]$ by \cite[Lemma 4.10]{kato2013perturbation}. Thus we prove the claim for a single fixed $\alpha=\alpha_0$ in the interval. (We write $\alpha$ instead of $\alpha_0$ for convenience). We study $\mathbf{P}_{0,\alpha}$ first. 
    
    Note that $\rg(\mathbf{P}_{\lambda,\alpha})\subset \mathcal{D}(\mathbf{L}_{\alpha})$. Indeed, let $\mathbf{v}\in \rg(\mathbf{P}_{\lambda,\alpha})$, then by the density of $\mathcal{D}(\mathbf{L}_{\alpha})$ in $\mathcal{H}^k$, we take $(\mathbf{u}_n)_{n\in \mathbb{N}}\subset \mathcal{D}(\mathbf{L}_{\alpha})$ with $\mathbf{u}_n \to \mathbf{v}$. Using the fact that $\mathbf{P}_{\lambda,\alpha}\mathcal{D}(\mathbf{L}_{\alpha})\subset \mathcal{D}(\mathbf{L}_{\alpha})$ we define $\mathbf{v}_n \coloneqq \mathbf{P}_{\lambda,\alpha}\mathbf{u}_n \in \rg(\mathbf{P}_{\lambda,\alpha})\cap \mathcal{D}(\mathbf{L}_{\alpha})$. Then $\mathbf{P}_{\lambda,\alpha}\mathbf{u}_n \to \mathbf{P}_{\lambda,\alpha}\mathbf{v} = \mathbf{v}$ by the continuity of $\mathbf{P}_{\lambda,\alpha}$. Since the restricted operator $\mathbf{L}_{\alpha}\rvert_{\mathcal{D}(\mathbf{L}_{\alpha})\cap \rg\mathbf{P}_{\lambda,\alpha}}$ is continuous, then $\mathbf{L}_{\alpha}\mathbf{v}_n \to \mathbf{w}$ for some $\mathbf{w}\in \rg(\mathbf{P}_{\lambda,\alpha})$, and then by the closedness of $\mathbf{L}_{\alpha}$, we have that $\mathbf{v}\in \mathcal{D}(\mathbf{L}_{\alpha})$.

    Thus we can define $\mathbf{A}\coloneqq \mathbf{L}_{\alpha}\rvert_{\rg(\mathbf{P}_{0,\alpha})}$, then $\mathbf{A}$ is bounded operator on the finite-dimensional subspace $\rg(\mathbf{P}_{0,\alpha})$ with $\spec(\mathbf{A})=\{0\}$. This implies that $\mathbf{A}$ is nilpotent, i.e. there exists a minimal $n\in \mathbb{N}$ such that $\mathbf{A}^n = \mathbf{0}$. We claim $n=2$.

    Suppose for contradiction that $n=1$. Then $\mathbf{A}=\mathbf{0}$ on $\rg(\mathbf{P}_{0,\alpha})$ which implies every element of $\rg(\mathbf{P}_{0,\alpha})$ is an eigenfunction and therefore a multiple of $\mathbf{f}_{0,\alpha}$. However $\mathbf{g}_{0,\alpha}\in \rg(\mathbf{P}_{0,\alpha})$, which is linearly independent of $\mathbf{f}_{0,\alpha}$. Thus $n\geq 2$. 
    
    Now, suppose for contradiction that $n>2$. This means that there exists some $\mathbf{u}\in \rg(\mathbf{P}_{0,\alpha})$ such that $\mathbf{A}^2 \mathbf{u}$ is a non-trivial element of 
    $\ker(\mathbf{A})\subset \ker(\mathbf{L}_{\alpha})= \text{span} (\mathbf{f}_{0,\alpha})$, i.e. $\mathbf{L}_{\alpha}^2\mathbf{u}=c_1 \mathbf{f}_{0,\alpha}$ for some $c_1\neq 0$. Without loss of generality, take $c_1=1$. Recall that $\mathbf{g}_{0,\alpha}$ satisfies $\mathbf{L}_{\alpha}\mathbf{g}_{0,\alpha}=\mathbf{f}_{0,\alpha}$, this implies that $\mathbf{L}_{\alpha}(\mathbf{L}_{\alpha}\mathbf{u}-\mathbf{g}_{0,\alpha}) = \mathbf{0}$, i.e $\mathbf{L}_{\alpha}\mathbf{u}-\mathbf{g}_{0,\alpha} = c_2 \mathbf{f}_{0,\alpha} = c_2 \mathbf{L}_{\alpha}\mathbf{g}_{0,\alpha}$. Rearranging, we have $\mathbf{L}_{\alpha}(\mathbf{u}-c_2\mathbf{g}_{0,\alpha})=\mathbf{g}_{0,\alpha}$. That is, there exists some $\mathbf{v}\in \mathcal{D}(\mathbf{L}_{\alpha})$ such that $\mathbf{L}_{\alpha}\mathbf{v} = \mathbf{g}_{0,\alpha}$. However, there are no such solutions for $k\geq \lceil c_{\alpha} \rceil +1$ by Lemma \ref{Appendix_lemma_eigenfunctions}. By a similar ODE analysis, it can be verified that $\mathbf{P}_{1,\alpha}^{(\alpha_0)}$ has rank 1. 
\end{proof}
\subsection{A uniform resolvent estimate}
We now prove a uniform bound for the resolvent operator in $|\alpha-\alpha_0|\leq \delta_0$. This will be the final ingredient in proving the semigroup bounds on the stable/unstable spectra after invoking the Gearhardt--Prüss--Greiner theorem. The next proposition confines possible unstable eigenvalues to a compact set. Thus, for fixed $0<\omega_0<1$ we define the compact rectangle 
\begin{equation}
    R_{m,n} \coloneqq \{z\in \mathbb{C}:\mathfrak{Re}z\in [-\omega_0,m], \mathfrak{Im}z \in [-n,n]\}
\end{equation}
and the `complement' of $R_{m,n}$ to the right of $z=-\omega_0$
\begin{equation}
    R_{m,n}' \coloneqq \{z\in \mathbb{C}:\mathfrak{Re}z\geq -\omega_0\}\setminus R_{m,n}.
\end{equation}
\begin{proposition}\label{prop:resolvent_estimate}
Let $k\geq \lceil c_{\alpha_0}\rceil +1$, $0<\omega_0<1$. There exists $m> 1,n>0$ such that $R_{m,n}'\subset \res(\mathbf{L}_{\alpha}^{(\alpha_0)})$ and 
\begin{equation}\label{resolvent_estimate}
    \sup_{|\alpha-\alpha_0|\leq \delta_0}\sup_{\lambda \in R_{m,n}'}\|(\lambda-\mathbf{L}_{\alpha}^{(\alpha_0)})^{-1}\|_{\mathcal{L}(\mathcal{H}^k(\mathcal{I}_{\alpha_0}))} \leq C
\end{equation}
where $C>0$. 
\end{proposition}
\begin{proof}
    Firstly, by Proposition \ref{spectral_inclusion_prop}, for $\alpha \in [\alpha_0-\delta_0,\alpha_0+\delta_0]$, $k\geq \lceil c_{\alpha_0}\rceil +1$, we have $\spec(\mathbf{L}_{\alpha}^{(\alpha_0)})\subset \{\mathfrak{Re}z\leq -1\}\cup \{0,1\}$ which implies that for any $m>1, n>0$, we have $R_{m,n}'\subset \{\mathfrak{Re}z>-1 \}\setminus \{0,1\}\subset \res(\mathbf{L}_{\alpha})$. 
    It therefore suffices to prove the existence of $m> 1, n>0$ such that for all $\lambda \in R_{m,n}'$, $\alpha \in [\alpha_0-\delta_0,\alpha_0+\delta_0]$, and $\mathbf{q}\in \mathcal{D}(\mathbf{L}_{\alpha})$ we have 
\begin{equation}
    \left|\left( (\lambda - \mathbf{L}_{\alpha})\mathbf{q},\mathbf{q} \right)_{\dot{H}_{\alpha_0}^{k+1}\times \dot{H}_{\alpha_0}^k} \right| + \left| \left\langle (\lambda-\mathbf{L}_{\alpha})\mathbf{q},\mathbf{q} \right\rangle_0 \right| \gtrsim \|\mathbf{q}\|_{\mathcal{H}_{\alpha_0}^k}^2.
\label{desired_estimate}
\end{equation}
    Indeed, if \eqref{desired_estimate} holds, then by Cauchy Schwarz, we have 
    \begin{equation}
        \|\mathbf{q}\|_{\mathcal{H}_{\alpha_0}^k} \lesssim \|(\lambda-\mathbf{L}_{\alpha})\mathbf{q}\|_{\mathcal{H}_{\alpha_0}^k}.
        \label{estimate_implies_resolvent_bound}
    \end{equation}
    This implies the resolvent estimate \eqref{resolvent_estimate}, since for $\lambda \in \res(\mathbf{L}_{\alpha})$, we have the standard fact
    \begin{equation*}
        \inf_{\{\mathbf{q}\neq \mathbf{0}\}}\frac{\|(\lambda-\mathbf{L}_{\alpha})\mathbf{q}\|_{\mathcal{H}_{\alpha_0}^k}}{\|\mathbf{q}\|_{\mathcal{H}_{\alpha_0}^k}} = \frac{1}{\|(\lambda-\mathbf{L}_{\alpha})^{-1}\|_{\mathcal{H}_{\alpha_0}^k}}.
    \end{equation*}
    We proceed by proving \eqref{desired_estimate} for $\mathbf{q}\in (C^\infty[-R_{\alpha_0},R_{\alpha_0}])^2$, since by density and the same argument as in the proof of Proposition \ref{maximal_dissipativity_proposition}, the estimate also holds for $\mathbf{q}\in \mathcal{D}(\mathbf{L}_{\alpha})$. Now we divide the set $R_{m,n}'$ into two parts, defined by
    \begin{align*}
        S_{m,n}^{(1)}&\coloneq \{z\in \mathbb{C}:\mathfrak{Re}z\geq -\omega_0,|\mathfrak{Im}z|\geq n\} \\
        S_{m,n}^{(2)}&\coloneq \{z\in \mathbb{C}:\mathfrak{Re}z\geq m,\mathfrak{Im}z\in[-n,n]\}
    \end{align*}
    and prove the estimate on each set separately. In fact, the majority of the proof is devoted to the harder task of  estimating \eqref{desired_estimate} for $\lambda\in S_{m,n}^{(1)}$. By definition we first have 
    \begin{align*}
        &\left((\lambda - \mathbf{L}_{\alpha})\mathbf{q},\mathbf{q} \right)_{\dot{H}_{\alpha_0}^{k+1}\times \dot{H}_{\alpha_0}^k} = \int_{-R_{\alpha_0}}^{R_{\alpha_0}}\partial_y^{k+1}(\lambda q_1+y\partial_y q_1 -q_2)\overline{\partial_y^{k+1}q_1}dy \\
        &+ \int_{-R_{\alpha_0}}^{R_{\alpha_0}}\partial_y^k\left((\lambda+1)q_2-\partial_{yy}q_1 + \frac{2\alpha}{\sqrt{1+\alpha}+y}\partial_y q_1 + y\partial_y q_2\right)\overline{\partial_y^k q_2}\,dy .
    \end{align*}
    Taking real parts after integrating by parts, we obtain
\begin{align*}
&\mathfrak{Re}\left((\lambda - \mathbf{L}_{\alpha})\mathbf{q},\mathbf{q} \right)_{\dot{H}_{\alpha_0}^{k+1}\times \dot{H}_{\alpha_0}^k} = \left(\mathfrak{Re}\lambda + k +\frac{1}{2}\right)
\left(\|\partial_y^{k+1}q_1\|_{L_{\alpha_0}^2}^2
+\|\partial_y^{k}q_2\|_{L_{\alpha_0}^2}^2\right)\\
&+ \mathfrak{Re} \sum_{i=0}^{k}\int_{-R_{\alpha_0}}^{R_{\alpha_0}}
C_i^k \partial_y^{k-i+1}q_1
\partial_y^{i}\left(\frac{2\alpha}{\sqrt{1+\alpha}+y}\right)
\overline{\partial_y^k q_2}\,dy +\frac{R_{\alpha_0}}{2}\left|\partial_y^{k+1}q_1(R_{\alpha_0})
-\frac{1}{R_{\alpha_0}}\partial_y^k q_2(R_{\alpha_0})\right|^2 \\
&+\frac{R_{\alpha_0}}{2}\left|\partial_y^{k+1}q_1(-R_{\alpha_0})
+\frac{1}{R_{\alpha_0}}\partial_y^k q_2(-R_{\alpha_0})\right|^2 + \frac{R_{\alpha_0}^2-1}{2R_{\alpha_0}}
\left(|\partial_y^k q_2(R_{\alpha_0})|^2+|\partial_y^k q_2(-R_{\alpha_0})|^2\right).
\end{align*}
    To bound, we  recall the `regularity constant' $c_{\alpha_0}$ defined in \eqref{c_alpha_0} 
    \begin{equation*}
        \sup_{|\alpha-\alpha_0|\leq\delta_0}\left|\int_{-R_{\alpha_0}}^{R_{\alpha_0}}\frac{2\alpha}{\sqrt{1+\alpha}+y}\partial_y^{k+1}q_1 \overline{\partial_y^k q_2}\,dy\right| \leq c_{\alpha_0}\left(\|\partial_y^{k+1}q_1\|_{L_{\alpha_0}^2}^2 + \|\partial_y^{k}q_2\|_{L_{\alpha_0}^2}^2\right).
    \end{equation*}
    Next, using the fact that
    \begin{equation*}
        \left| \partial_y^{i}\left(\frac{2\alpha}{\sqrt{1+\alpha}+y}\right) \right| \lesssim_{i}\frac{\alpha}{(\sqrt{1+\alpha}+y)^{i+1}} \leq \frac{\alpha_0 +\delta_0}{(\sqrt{1+\alpha_0-\delta_0}-R_{\alpha_0})^{i+1}} = M(\alpha_0)
    \end{equation*}
    which is well defined and independent of $\alpha$ by our choice of $R_{\alpha_0}, \delta_0(\alpha_0)$, we apply the Gagliardo--Nirenberg inequality 
    \begin{equation*}
        \left|\sum_{i=1}^{k}\int_{-R_{\alpha_0}}^{R_{\alpha_0}}C_i^k \partial_y^{k-i+1}q_1\,\partial_y^i\left(\frac{2\alpha}{\sqrt{1+\alpha}+y}\right)\overline{\partial_y^k q_2}\,dy\right| \leq M_{\alpha_0,\varepsilon'}\|\partial_y q_1\|_{L_{\alpha_0}^2}^2 + \varepsilon'\left(\|\partial_y^{k+1}q_1\|_{L_{\alpha_0}^2}^2 + \|\partial_y^{k}q_2\|_{L_{\alpha_0}^2}^2\right).
    \end{equation*}
    Putting everything together and using the real part on the previous page as a lower bound, for $k\geq \lceil c_{\alpha_0}\rceil +1$ 
    \begin{equation}\label{eq:real_estimate}
    \begin{split}
        \left|\left((\lambda-\mathbf{L}_{\alpha})\mathbf{q},\mathbf{q} \right)_{\dot{H}_{\alpha_0}^{k+1}\times \dot{H}_{\alpha_0}^k} \right| &\geq \left(\mathfrak{Re}\lambda + k +\frac{1}{2}-c_{\alpha_0}-\varepsilon'\right)\left(\|\partial_y^{k+1}q_1\|_{L_{\alpha_0}^2}^2 + \|\partial_y^{k}q_2\|_{L_{\alpha_0}^2}^2\right) - M_{\alpha_0,\varepsilon'}\|\partial_y q_1\|_{L_{\alpha_0}^2}^2 \\
        &\geq \left(\mathfrak{Re}\lambda + \frac{3}{2}-\varepsilon' \right)\left(\|\partial_y^{k+1}q_1\|_{L_{\alpha_0}^2}^2 + \|\partial_y^{k}q_2\|_{L_{\alpha_0}^2}^2\right) - M_{\alpha_0,\varepsilon'}\|\partial_y q_1\|_{L_{\alpha_0}^2}^2.
    \end{split}
    \end{equation}
    Now we move on to the $k=0$ term of \eqref{desired_estimate}. By definition of the inner product, integration by parts, and separating real and imaginary parts, we write 
    \begin{align*}
        &\left\langle (\lambda-\mathbf{L}_{\alpha})\mathbf{q},\mathbf{q} \right\rangle_0 = \int_{-R_{\alpha_0}}^{R_{\alpha_0}}\partial_y\left(\lambda q_1 +y\partial_y q_1 -q_2\right)\overline{\partial_y}q_1\,dy + \frac{1}{R_{\alpha_0}}(\lambda q_1 +y\partial_y q_1 -q_2)\overline{q_1}\big\rvert_{y=-R_{\alpha_0}} \\
        &+\int_{-R_{\alpha_0}}^{R_{\alpha_0}}\left((\lambda+1)q_2 +y\partial_y q_2 + \frac{2\alpha}{\sqrt{1+\alpha}+y}\partial_y q_1-\partial_{yy}q_1 \right)\overline{q_2}\,dy  \\
        &= \left(\mathfrak{Re}\lambda+\frac{1}{2}\right)\left(\|\partial_y q_1\|_{L_{\alpha_0}^2}^2+\| q_2\|_{L_{\alpha_0}^2}^2\right) + \frac{\mathfrak{Re}\lambda}{R_{\alpha_0}}|q_1(-R_{\alpha_0})|^2 + \frac{R_{\alpha_0}^2 -1}{2R_{\alpha_0}}\left(|q_2(R_{\alpha_0})|^2 + |q_2(-R_{\alpha_0})|^2\right) \\
        &+ \frac{R_{\alpha_0}}{2}\left(\left|\partial_y q_1(R_{\alpha_0}) -\frac{1}{R_{\alpha_0}}q_2(R_{\alpha_0}) \right|^2+\left|\partial_y q_1(-R_{\alpha_0})+\frac{1}{R_{\alpha_0}}q_2(-R_{\alpha_0}) \right|^2 \right) \\
        &+ \int_{-R_{\alpha_0}}^{R_{\alpha_0}}\frac{2\alpha}{\sqrt{1+\alpha}+y}\partial_y q_1 \overline{q_2}\,dy + \frac{1}{R_{\alpha_0}}(y\partial_y q_1 -q_2)\overline{q_1}\big\rvert_{y=-R_{\alpha_0}} \\
    &+i\mathfrak{Im}\lambda\left(\|\partial_y q_1\|_{L_{\alpha_0}^2}^2
+\|q_2\|_{L_{\alpha_0}^2}^2
+\frac{1}{R_{\alpha_0}}|q_1(-R_{\alpha_0})|^2\right) 
+ \frac12 \int_{-R_{\alpha_0}}^{R_{\alpha_0}}
y\Bigl(\partial_{yy}q_1\,\overline{\partial_y q_1}
-\overline{\partial_{yy}q_1}\,\partial_y q_1\Bigr)\,dy \\
&+ \frac12 \int_{-R_{\alpha_0}}^{R_{\alpha_0}}
y\Bigl(\partial_y q_2\,\overline{q_2}
-\overline{\partial_y q_2}\,q_2\Bigr)\,dy 
+ \int_{-R_{\alpha_0}}^{R_{\alpha_0}}
\Bigl(\partial_y q_1\,\overline{\partial_y q_2}
-\overline{\partial_y q_1}\,\partial_y q_2\Bigr)\,dy \\
&- \frac12\Bigl(\partial_y q_1(R_{\alpha_0})\overline{q_2(R_{\alpha_0})}
-\overline{\partial_y q_1(R_{\alpha_0})}\,q_2(R_{\alpha_0})\Bigr)
+ \frac12\Bigl(\partial_y q_1(-R_{\alpha_0})\overline{q_2(-R_{\alpha_0})}
-\overline{\partial_y q_1(-R_{\alpha_0})}\,q_2(-R_{\alpha_0})\Bigr).
    \end{align*}
    We now bound the terms
    \begin{align*}
        \left|\int_{-R_{\alpha_0}}^{R_{\alpha_0}}\frac{2\alpha}{\sqrt{1+\alpha}+y}\partial_y q_1 \overline{q_2}\, dy\right| \leq c_{\alpha_0}\left(\|\partial_y q_1 \|_{L_{\alpha_0}^2}^2 + \|q_2\|_{L_{\alpha_0}^2}^2 \right) 
    \end{align*}
    and 
    \begin{align*}
        \left|\frac{1}{R_{\alpha_0}}(y\partial_y q_1 - q_2)\overline{q_1}|_{y=-R_{\alpha_0}} \right| &\lesssim_{\alpha_0}\left(\|\partial_y q_1\|_{L_{\alpha_0}^\infty}+\|q_2\|_{L_{\alpha_0}^\infty} \right)\|q_1\|_{L_{\alpha_0}^\infty} 
        \lesssim_{\alpha_0}\left(\|\partial_y q_1\|_{H_{\alpha_0}^1}+\|q_2\|_{H_{\alpha_0}^1} \right)\|q_1\|_{H_{\alpha_0}^1} \\
        &\lesssim_{\alpha_0} \varepsilon'\left(\|\partial_y^{k+1} q_1\|_{L_{\alpha_0}^2}^2+\|\partial_y^k q_2\|_{L_{\alpha_0}^2}^2 \right) + M_0\langle \mathbf{q},\mathbf{q}\rangle_0.
    \end{align*}
    For the other imaginary terms, 
    \begin{align*}
        \left|\frac12 \int_{-R_{\alpha_0}}^{R_{\alpha_0}}
y\Bigl(\partial_y q_2\,\overline{q_2}
-\overline{\partial_y q_2}\,q_2\Bigr)\,dy\right| \lesssim_{\alpha_0}\varepsilon'\|\partial_y^{k}q_2\|_{L_{\alpha_0}^2}^2 + M_1\|q_2\|_{L_{\alpha_0}^2}^2,\\
\left|\int_{-R_{\alpha_0}}^{R_{\alpha_0}}
\Bigl(\partial_y q_1\,\overline{\partial_y q_2}
-\overline{\partial_y q_1}\,\partial_y q_2\Bigr)\,dy \right|\lesssim_{\alpha_0} \varepsilon' \|\partial_y^k q_2\|_{L_{\alpha_0}^2}^2 + M_2 \langle \mathbf{q},\mathbf{q}\rangle_0, \\
\left|\frac12 \int_{-R_{\alpha_0}}^{R_{\alpha_0}}
y\Bigl(\partial_{yy}q_1\,\overline{\partial_y q_1}
-\overline{\partial_{yy}q_1}\,\partial_y q_1\Bigr)\,dy \right|\lesssim_{\alpha_0} \varepsilon'\|\partial_y^{k+1}q_1\|_{L_{\alpha_0}^2}^2 + M_3\langle \mathbf{q},\mathbf{q}\rangle_0,\\
\left|\frac12\left[\partial_y q_1\overline{q_2}
-\overline{\partial_y q_1}\,q_2\right]_{-R_{\alpha_0}}^{R_{\alpha_0}} \right| \lesssim_{\alpha_0} \varepsilon'\left(\|\partial_y^{k+1}q_1\|_{L_{\alpha_0}^2}^2+\|\partial_y^k q_2\|_{L_{\alpha_0}^2}^2\right) + M_4\langle \mathbf{q},\mathbf{q}\rangle_0
    \end{align*}
    of which all can be bounded by $\varepsilon'(\|\partial_y^{k+1}q_1\|_{L_{\alpha_0}^2}^2+\|\partial_y^k q_2\|_{L_{\alpha_0}^2}^2) + M_i\langle \mathbf{q},\mathbf{q}\rangle_0 $.
    Using these estimates and bounding below by the imaginary terms instead, we have
    \begin{align}\label{eq:imaginary_estimate}
        \left|\left\langle (\lambda-\mathbf{L}_{\alpha})\mathbf{q},\mathbf{q} \right\rangle_0 \right| \geq \left(|\mathfrak{Im}\lambda| - c_{\alpha_0} - \sum_{i=0}^{4}M_i\right)\langle \mathbf{q},\mathbf{q}\rangle_0 - 4\varepsilon'\left(\|\partial_y^{k+1}q_1\|_{L_{\alpha_0}^2}^2+\|\partial_y^k q_2\|_{L_{\alpha_0}^2}^2\right).
    \end{align}
    Combining \eqref{eq:real_estimate}--\eqref{eq:imaginary_estimate} yields
    \begin{align*}
         &\left|\left( (\lambda - \mathbf{L}_{\alpha})\mathbf{q},\mathbf{q} \right)_{\dot{H}_{\alpha_0}^{k+1}\times \dot{H}_{\alpha_0}^k} \right| + \left| \left\langle (\lambda-\mathbf{L}_{\alpha})\mathbf{q},\mathbf{q} \right\rangle_0 \right| \\
         &\geq \left(|\mathfrak{Im}\lambda| - c_{\alpha_0} - \sum_{i=0}^{4}M_i - M_{\alpha_0,\varepsilon'}\right)\langle \mathbf{q},\mathbf{q}\rangle_0 
         +\left(\mathfrak{Re}\lambda + \frac{3}{2}-5\varepsilon' \right)\left(\|\partial_y^{k+1}q_1\|_{L_{\alpha_0}^2}^2 + \|\partial_y^{k}q_2\|_{L_{\alpha_0}^2}^2\right) \\
         &\geq \frac{1}{2}\|\mathbf{q}\|_k^2
    \end{align*}
    for sufficiently small $\varepsilon'$  and large enough $n$ and $\mathfrak{Re}\lambda \geq -\omega_0$. This implies that the estimate \eqref{desired_estimate} holds for $\lambda \in S_{m,n}^{(1)}$. Finally, for $\lambda \in S_{m,n}^{(2)}$, by Proposition \ref{Semigroup_generation_lin_operator}, and the Hille--Yosida theorem we have the resolvent bound
    \begin{align}\label{eq:lambda_in_s_2}
        \|(\lambda-\mathbf{L}_{\alpha}^{(\alpha_0)})^{-1}\|_{\mathcal{L}(\mathcal{H}_{\alpha_0}^k)} \leq \frac{C_{k,\alpha_0}}{\mathfrak{Re}\lambda-(-\frac{1}{2}+\|\mathbf{L}_{\alpha,1}\|_{\mathcal{L}(\mathcal{H}_{\alpha_0}^k)})}.
    \end{align}
    Recall that $\|\mathbf{L}_{\alpha,1}\|_{\mathcal{L}(\mathcal{H}_{\alpha_0}^k)}\lesssim c_{\alpha_0}<\infty$ for $\alpha \in [\alpha_0-\delta_0,\alpha_0+\delta_0]$, thus we may take $m$ sufficiently large so that \eqref{eq:lambda_in_s_2} has a uniform bound for all $\lambda \in S_{m,n}^{(2)}$. 
\end{proof}
\subsection{Stable and unstable dynamics}
Now we can finally give a sufficient description of the linearised flow, which will provide us with the machinery needed for Section \ref{section_three}--\ref{section_four}.
\begin{proposition}\label{prop:linearised_evolution}
    Let $k\geq \lceil c_{\alpha_0} \rceil +1$, $0<\omega_0<1$. Then for any $\alpha \in [\alpha_0-\delta_0,\alpha_0+\delta_0]$ and $\gamma \in (0,\omega_0)$, with $\omega \coloneqq \omega_0 - \gamma$, the following properties hold.
    
    \medskip
\noindent(i) Commutation properties.
For every $\tau \ge 0$,
\begin{align}\label{eq:commutation}
    \mathbf{S}_{\alpha}(\tau)\mathbf{P}_{1,\alpha}^{(\alpha_0)} 
      = \mathbf{P}_{1,\alpha}^{(\alpha_0)}\mathbf{S}_{\alpha}(\tau),
    \qquad
    \mathbf{S}_{\alpha}(\tau)\mathbf{P}_{0,\alpha}^{(\alpha_0)}
      = \mathbf{P}_{0,\alpha}^{(\alpha_0)}\mathbf{S}_{\alpha}(\tau).
\end{align}

\medskip
\noindent(ii) Dynamics on invariant subspaces. 
For every $\tau \ge 0$,
\begin{align}
    \mathbf{S}_{\alpha}(\tau)\mathbf{P}_{1,\alpha}^{(\alpha_0)} 
        &= e^{\tau}\,\mathbf{P}_{1,\alpha}^{(\alpha_0)}, \\
    \mathbf{S}_{\alpha}(\tau)\mathbf{P}_{0,\alpha}^{(\alpha_0)} 
        &= \mathbf{P}_{0,\alpha}^{(\alpha_0)} + \tau\,\mathbf{L}_{\alpha}\mathbf{P}_{0,\alpha}\label{eq:0_eigenspace_dynamic}, \\
    \bigl\|\mathbf{S}_{\alpha}(\tau)\widetilde{\mathbf{P}}_{\alpha}^{(\alpha_0)}\mathbf{q}\bigr\|_{\mathcal{H}^k(\mathcal{I}_{\alpha_0})}
        &\le M e^{-\omega_0 \tau}\,
           \bigl\|\widetilde{\mathbf{P}}_{\alpha}^{(\alpha_0)}\mathbf{q}\bigr\|_{\mathcal{H}^k(\mathcal{I}_{\alpha_0})} \label{eq:stable_semigroup_decay}
\end{align}
where
\[
    \widetilde{\mathbf{P}}_{\alpha}^{(\alpha_0)}
      \coloneqq \mathbf{I}-\mathbf{P}_{\alpha}^{(\alpha_0)}, \qquad \mathbf{P}_{\alpha}^{(\alpha_0)} \coloneqq \mathbf{P}_{0,\alpha}^{(\alpha_0)} + \mathbf{P}_{1,\alpha}^{(\alpha_0)},
    \qquad
    M = M(\omega_0,\alpha_0).
\]

\medskip
\noindent(iii) Ranges of the spectral projections.
\[
    \operatorname{rg}\,\mathbf{P}_{0,\alpha}^{(\alpha_0)} = 
        \operatorname{span}\{\mathbf{f}_{0,\alpha},\,\mathbf{g}_{0,\alpha}\},
    \qquad
    \operatorname{rg}\,\mathbf{P}_{1,\alpha}^{(\alpha_0)} = 
        \operatorname{span}\{\mathbf{f}_{1,\alpha}\}.
\]

\medskip
\noindent(iv) Lipschitz dependence on parameters. For $\alpha_1,\alpha_2 \in [\alpha_0-\delta_0,\alpha_0+\delta_0]$,
\begin{align}
    \|\mathbf{P}_{0,\alpha_1}^{(\alpha_0)}-\mathbf{P}_{0,\alpha_2}^{(\alpha_0)}\|_{\mathcal{L}(\mathcal{H}_{\alpha_0}^k)} &\lesssim_{\alpha_0} |\alpha_1-\alpha_2|\label{eq:lipschitz_dep_riesz_projec}, \\
    \|\mathbf{P}_{1,\alpha_1}^{(\alpha_0)}-\mathbf{P}_{1,\alpha_2}^{(\alpha_0)}\|_{\mathcal{L}(\mathcal{H}_{\alpha_0}^k)} &\lesssim_{\alpha_0} |\alpha_1-\alpha_2|
    \end{align}
    and for all $\tau \geq 0$
    \begin{equation}\label{eq:lipschitz_dep_stable_semigroup}
         \|\mathbf{S}_{\alpha_1}(\tau)\widetilde{\mathbf{P}}_{\alpha_1}^{(\alpha_0)}-\mathbf{S}_{\alpha_2}(\tau)\widetilde{\mathbf{P}}_{\alpha_2}^{(\alpha_0)}\|_{\mathcal{L}(\mathcal{H}_{\alpha_0}^k)} \lesssim_{\alpha_0,\omega_0,\gamma,k} e^{-\omega\tau}|\alpha_1-\alpha_2|.
    \end{equation}
\end{proposition}
\begin{proof}
Immediately from the definition of a $C_0$-semigroup, it is standard that it commutes with its generator and the resolvent operator, hence also the Riesz projections, which justifies \eqref{eq:commutation}. 

We next prove $(iii)$. Proposition \ref{spectral_inclusion_prop} implies that $\operatorname{span}\{\mathbf{f}_{0,\alpha},\,\mathbf{g}_{0,\alpha}\}\subseteq \rg(\mathbf{P}_{0,\alpha}^{(\alpha_0)})$, and Proposition \ref{prop:projection_rank} gives us equality of the sets. The same argument holds for $\rg(\mathbf{P}_{1,\alpha}^{(\alpha_0)})$. 

For $(ii)$, we prove only the polynomial growth \eqref{eq:0_eigenspace_dynamic} as the argument for $\mathbf{P}_{1,\alpha}^{(\alpha_0)}$ is identical and even simpler since there is no generalised mode. Indeed, for $|\alpha-\alpha_0|\leq \delta_0$, $\mathbf{q}\in \mathcal{H}^k(\mathcal{I}_{\alpha_0})$, by definition of the semigroup and the commutation in \textit{(i)}, we have the equations
\begin{align*}
    &\partial_{\tau}\mathbf{S}_{\alpha}(\tau)\mathbf{P}_{0,\alpha}\mathbf{q}=\mathbf{L}_{\alpha}\mathbf{S}_{\alpha}(\tau)\mathbf{P}_{0,\alpha}\mathbf{q}=\mathbf{S}_{\alpha}(\tau)\mathbf{L}_{\alpha}\mathbf{P}_{0,\alpha}\mathbf{q} \\
    &\partial_{\tau}\mathbf{S}_{\alpha}(\tau)\mathbf{L}_{\alpha}\mathbf{P}_{0,\alpha}\mathbf{q} = \mathbf{S}_{\alpha}(\tau)\mathbf{L}_{\alpha}^2\mathbf{P}_{0,\alpha}\mathbf{q} = 0
\end{align*}
which can be solved to obtain $\mathbf{S}_{\alpha}(\tau)\mathbf{P}_{0,\alpha} = \tau \mathbf{L}_{\alpha}\mathbf{P}_{0,\alpha} + \mathbf{P}_{0,\alpha}$. For the decay estimate \eqref{eq:stable_semigroup_decay} of $(ii)$, we use a corollary of the Gearhardt--Prüss--Greiner theorem \cite[Theorem 1.11]{engel2000one}. First, define $\mathbf{A}_{\alpha} \coloneqq \mathbf{L}_{\alpha}\rvert_{\rg\,\widetilde{\mathbf{P}}_{\alpha}}$. By Proposition \ref{spectral_inclusion_prop}, $\spec(\mathbf{A}_{\alpha}) \subset \{\mathfrak{Re}z \leq -1\}$. Moreover, note that the compact set $R_{m,n}$ defined in Proposition \ref{prop:resolvent_estimate} satisfies $R_{m,n}\subset \res(\mathbf{A}_{\alpha})$ for any $m,n>0$ whence the resolvent is uniformly bounded for $\alpha \in [\alpha_0-\delta_0,\alpha_0+\delta_0]$. Then, for $\lambda \in R_{m,n}'$, we have the uniform bound \eqref{resolvent_estimate}. Together, both bounds imply 
\begin{equation*}
    \|(\lambda - \mathbf{A}_{\alpha})^{-1}\|_{\mathcal{L}({\mathcal{H}_{\alpha_0}^k})}\leq C
\end{equation*}
for all $\lambda \in \{\mathfrak{Re}z\geq -\omega_0\}$ and $|\alpha-\alpha_0|\leq \delta_0$. Thus, the Gearhardt--Prüss--Greiner theorem yields the uniform semigroup bound in $(ii)$.

We now prove Lipschitz dependence on parameter $\alpha$ for the Riesz projection, say $\mathbf{P}_{0,\alpha}$ given in \eqref{eq:lipschitz_dep_riesz_projec}. 
Firstly, by Proposition \ref{spectral_inclusion_prop}, $\gamma_0 \subset \res(\mathbf{L}_{\alpha})$. Taking $\alpha_1,\alpha_2 \in [\alpha_0-\delta_0,\alpha_0+\delta_0]$, we have the second resolvent identity 
\begin{equation}\label{eq:second_resolvent_identity}
    (z-\mathbf{L}_{\alpha_1})^{-1}-(z-\mathbf{L}_{\alpha_2})^{-1} = (z-\mathbf{L}_{\alpha_1})^{-1}(\mathbf{L}_{\alpha_1}-\mathbf{L}_{\alpha_2})(z-\mathbf{L}_{\alpha_2})^{-1}.
\end{equation}
Then 
\begin{equation}\label{eq:difference_of_resolvents}
    \mathbf{P}_{\alpha_1}-\mathbf{P}_{\alpha_2} = \frac{1}{2\pi i}\int_{\gamma_0}(z-\mathbf{L}_{\alpha_1})^{-1}(\mathbf{L}_{\alpha_1}-\mathbf{L}_{\alpha_2})(z-\mathbf{L}_{\alpha_2})^{-1}\,dz.
\end{equation}
Note that the quantity $M_{\Gamma} \coloneqq \sup_{\{|\alpha-\alpha_0|\leq\delta_0\}}\sup_{z\in \gamma_0}\|(z-\mathbf{L}_{\alpha})^{-1}\|_{\mathcal{L}(\mathcal{H}^k)}<\infty$ by continuity of the resolvent and compactness of $\{|\alpha-\alpha_0|\leq\delta_0\}$ and $\gamma_0$. Therefore, using the Lipschitz dependence on $\alpha$ of $\mathbf{L}_{\alpha}$ from Proposition \ref{Semigroup_generation_lin_operator}, we may bound \eqref{eq:difference_of_resolvents} as 
\begin{equation*}
    \|\mathbf{P}_{\alpha_1}-\mathbf{P}_{\alpha_2}\|_{\mathcal{L}(\mathcal{H}^k)} \leq \frac{|\gamma_0|M_{\Gamma}^2}{2
\pi}\|\mathbf{L}_{\alpha_1}-\mathbf{L}_{\alpha_2}\|_{\mathcal{L}(\mathcal{H}^k)} \lesssim_{\alpha_0,k}|\alpha_1-\alpha_2|.
\end{equation*}
Next, we write out the proof for the third Lipschitz estimate, \eqref{eq:lipschitz_dep_stable_semigroup}, from \textit{(iv)}. For this, we fix $\mathbf{q}\in \mathcal{D}(\mathbf{L}_{\alpha}^{(\alpha_0)})$ and for $\alpha_1,\alpha_2 \in [\alpha_0-\delta_0,\alpha_0+\delta_0]$, we first define the function 
    \begin{equation*}
        \mathbf{D}_{\alpha_1,\alpha_2}(\tau) \coloneqq \frac{\mathbf{S}_{\alpha_1}(\tau)\widetilde{\mathbf{P}}_{\alpha_1}\mathbf{q}-\widetilde{\mathbf{P}}_{\alpha_1}\mathbf{S}_{\alpha_2}(\tau)\widetilde{\mathbf{P}}_{\alpha_2}\mathbf{q}}{\alpha_1-\alpha_2}.
    \end{equation*}
    Note that by definition of the semigroup and the commutation identities of \textit{(i)}, this function solves the initial value problem
    \begin{align*}
    \begin{cases}
        \partial_{\tau}\mathbf{D}_{\alpha_1,\alpha_2}(\tau) &= \mathbf{L}_{\alpha_1}\mathbf{D}_{\alpha_1,\alpha_2} + \widetilde{\mathbf{P}}_{\alpha_1}\frac{\mathbf{L}_{\alpha_1}-\mathbf{L}_{\alpha_2}}{\alpha_1-\alpha_2}\mathbf{S}_{\alpha_2}(\tau)\widetilde{\mathbf{P}}_{\alpha_2}\mathbf{q}\\
        \mathbf{D}_{\alpha_1,\alpha_2}(0) &= \frac{\widetilde{\mathbf{P}}_{\alpha_1}\mathbf{q}-\widetilde{\mathbf{P}}_{\alpha_1}\widetilde{\mathbf{P}}_{\alpha_2}\mathbf{q}}{\alpha_1-\alpha_2} = \widetilde{\mathbf{P}}_{\alpha_1}\frac{\mathbf{P}_{\alpha_2}-\mathbf{P}_{\alpha_1}}{\alpha_1-\alpha_2}. 
    \end{cases}
    \end{align*}
    Thus, by Duhamel's principle we have 
    \begin{align*}
        \mathbf{D}_{\alpha_1,\alpha_2}(\tau) = \mathbf{S}_{\alpha_1}(\tau)\widetilde{\mathbf{P}}_{\alpha_1}\frac{\mathbf{P}_{\alpha_2}-\mathbf{P}_{\alpha_1}}{\alpha_1-\alpha_2}\mathbf{q} + \int_{0}^{\tau}\mathbf{S}_{\alpha_1}(\tau-\tau')\widetilde{\mathbf{P}}_{\alpha_1}\frac{\mathbf{L}_{\alpha_1}-\mathbf{L}_{\alpha_2}}{\alpha_1-\alpha_2}\mathbf{S}_{\alpha_2}(\tau')\widetilde{\mathbf{P}}_{\alpha_2}\mathbf{q}\,d\tau'.
    \end{align*}
    We now estimate this quantity, using the established decay estimate from $\textit{(ii)}$ and the Lipschitz-dependence on parameters of $\mathbf{P}_{\lambda,\alpha}$ and $\mathbf{L}_{\alpha}$: 
    \begin{equation*}
        \|\mathbf{D}_{\alpha_1,\alpha_2}(\tau)\|_{\mathcal{H}^k} \lesssim_{\alpha_0,\omega_0,k} e^{-\omega_0\tau}\|\mathbf{q}\|_{\mathcal{H}^k}(1+\tau) \lesssim_{\gamma}e^{-(\omega_0-\gamma)\tau}\|\mathbf{q}\|_{\mathcal{H}^k}
    \end{equation*}
    where the implicit gamma-dependent constant $C(\gamma)\to \infty$ as $\gamma \to 0$. Therefore, by adding and subtracting, the true quantity we want to estimate can be written as 
    \begin{align*}
        \frac{(\mathbf{S}_{\alpha_1}(\tau)\widetilde{\mathbf{P}}_{\alpha_1}-\mathbf{S}_{\alpha_2}(\tau)\widetilde{\mathbf{P}}_{\alpha_2})\mathbf{q}}{\alpha_1-\alpha_2} = \mathbf{D}_{\alpha_1,\alpha_2}(\tau) + \frac{(\widetilde{\mathbf{P}}_{\alpha_1}\mathbf{S}_{\alpha_2}(\tau)\widetilde{\mathbf{P}}_{\alpha_2}-\mathbf{S}_{\alpha_2}(\tau)\widetilde{\mathbf{P}}_{\alpha_2})\mathbf{q}}{\alpha_1-\alpha_2}.
    \end{align*}
     The remaining term can be written as 
     \begin{equation*}
     \frac{(\widetilde{\mathbf{P}}_{\alpha_1}\mathbf{S}_{\alpha_2}(\tau)\widetilde{\mathbf{P}}_{\alpha_2}-\mathbf{S}_{\alpha_2}(\tau)\widetilde{\mathbf{P}}_{\alpha_2})\mathbf{q}}{\alpha_1-\alpha_2} = \frac{(\mathbf{P}_{\alpha_2}\mathbf{S}_{\alpha_2}(\tau)\widetilde{\mathbf{P}}_{\alpha_2}-\mathbf{P}_{\alpha_1}\mathbf{S}_{\alpha_2}(\tau)\widetilde{\mathbf{P}}_{\alpha_2})\mathbf{q}}{\alpha_1-\alpha_2} = \frac{(\mathbf{P}_{\alpha_2}-\mathbf{P}_{\alpha_1})\mathbf{S}_{\alpha_2}(\tau)\widetilde{\mathbf{P}}_{\alpha_2}\mathbf{q}}{\alpha_1-\alpha_2}.
     \end{equation*}
     Thus, by using the Lipschitz dependence on parameters of $\mathbf{P}_{\lambda,\alpha}$, it can be bounded as
     \begin{equation*}
         \left\|\frac{(\widetilde{\mathbf{P}}_{\alpha_1}\mathbf{S}_{\alpha_2}(\tau)\widetilde{\mathbf{P}}_{\alpha_2}-\mathbf{S}_{\alpha_2}(\tau)\widetilde{\mathbf{P}}_{\alpha_2})\mathbf{q}}{\alpha_1-\alpha_2}\right\|_{\mathcal{H}^k} = \left\|\frac{(\mathbf{P}_{\alpha_2}-\mathbf{P}_{\alpha_1})\mathbf{S}_{\alpha_2}(\tau)\widetilde{\mathbf{P}}_{\alpha_2}\mathbf{q}}{\alpha_1-\alpha_2}\right\|_{\mathcal{H}^k}
 \lesssim_{\alpha_0,\omega_0}e^{-\omega_0\tau}\|\mathbf{q}\|_{\mathcal{H}^k}.     \end{equation*}
 Finally, 
 \begin{equation*}
     \left\|\frac{(\mathbf{S}_{\alpha_1}(\tau)\widetilde{\mathbf{P}}_{\alpha_1}-\mathbf{S}_{\alpha_2}(\tau)\widetilde{\mathbf{P}}_{\alpha_2})\mathbf{q}}{\alpha_1-\alpha_2}\right\|_{\mathcal{H}^k}\lesssim_{\alpha_0,\omega_0,\gamma}(e^{-(\omega_0-\gamma)\tau}+e^{-\omega_0\tau})\|\mathbf{q}\|_{\mathcal{H}^k}\lesssim e^{-(\omega_0-\gamma)\tau}\|\mathbf{q}\|_{\mathcal{H}^k}.
 \end{equation*}
\end{proof}
\textit{To avoid ambiguity, we now  fix $\omega_0 \in (0,1)$ and then $\gamma \in (0,\omega_0)$. This fixes $\omega \coloneqq \omega_0-\gamma \in (0,\omega_0)$ for the rest of this work. Furthermore, to ensure we have both the Lipschitz estimate \eqref{eq:lipschitz_dep_stable_semigroup}  and the semigroup decay estimate, we will weaken the decay from $\omega_0$ to $\omega$}.  
\section{Construction of radially symmetric solutions}\label{section_three}
We now restrict attention to the construction of radially symmetric
blow-up solutions on the sphere \(\{|x|=r_0\}\). The
equation \eqref{Nd-equation} in radial variables, say for $v=v(r,t)$ becomes
\begin{equation}\label{radial_eq_v}
     v_{tt}- v_{rr}-\frac{n-1}{r} v_r=( v_r)^2.
\end{equation}
After introducing the following transformation
\[
    u(r,t)=v(r,t)+\frac{n-1}{2}\log\left(\frac{r}{r_0}\right),
\]
a direct calculation shows that \(v\) solves \eqref{radial_eq_v} if
and only if \(u\) solves
\begin{equation}\label{eq:transformed_in_u}
    u_{tt}-u_{rr}
    =
    (u_r)^2
    -
    \frac{(n-1)(n-3)}{4r^2}.
\end{equation}
This new transformation shows how the radial problem is converted into a one-dimensional equation with
an inverse-square forcing term.
Away from $r=0$, the leading-order singular dynamics are governed by the one-dimensional
equation
\begin{equation}\label{eq:1d_in_r}
    u_{tt}-u_{rr}=(u_r)^2,
\end{equation}
and the stable branch of singular solutions
\begin{equation}\label{eq:family_masmoudi_solutions}
    u_{\alpha,\kappa,T,r_0}(r,t)
    =
    -\alpha\log\left(1-\frac{t}{T}\right)
    -\alpha\log\left(
        \sqrt{1+\alpha}
        +
        \frac{r-r_0}{T-t}
    \right) + \kappa 
\end{equation}
whose linearised dynamics we studied in Section \ref{section_two}. We therefore seek
solutions of the transformed equation \eqref{eq:transformed_in_u} in the form
\[
    u(r,t)
    =
    u_{\alpha,\kappa,T,r_0}(r,t)
    +
    \zeta(r,t)
\]
which means an ansatz to the original equation \eqref{radial_eq_v} of
\[
    v(r,t)
    =
    u_{\alpha,\kappa,T,r_0}(r,t)
    -
    \frac{n-1}{2}\log\left(\frac r{r_0}\right)
    +
    \zeta(r,t).
\]
Substituting this into \eqref{radial_eq_v} gives the equation for $\zeta$:
\begin{equation}\label{eq:for_zeta}
    \zeta_{tt} - \zeta_{rr} + \frac{2\alpha}{\sqrt{1+\alpha}(T-t)+r-r_0}\zeta_r = -\frac{(n-1)(n-3)}{4r^2} + (\zeta_r)^2. 
\end{equation}
We now formulate this in similarity coordinates and a first order vector form i.e. 
\begin{align*}
    \tau &\coloneqq \tau_{T_0} \coloneqq  \log(T_0)-\log(T_0-t), \\
    y &\coloneqq y_{T_0} \coloneqq \frac{r-r_0}{T_0-t}
\end{align*}
and let $\zeta(r,t) = \eta(\tau,y)$, and then define $\mathbf{q} = (q_1,q_2)^{\top} =(\eta,\eta_\tau + y\eta_y)$. Using 
\begin{align*}
\partial_t \zeta 
  &= (T_0-t)^{-1}\!\left(\partial_\tau\eta + y\,\partial_y\eta\right), \\[4pt]
\partial_{tt} \zeta 
  &= (T_0-t)^{-2}\!\left(
        y^{2}\partial_{yy}\eta 
      + 2y\,\partial_{y\tau}\eta 
      + 2y\,\partial_y\eta 
      + \partial_\tau\eta 
      + \partial_{\tau\tau}\eta
     \right), \\[4pt]
\partial_r \zeta 
  &= (T_0-t)^{-1}\partial_y\eta, \\[4pt]
\partial_{rr} \zeta 
  &= (T_0-t)^{-2}\partial_{yy}\eta,
\end{align*}
and multiplying both sides by $(T_0-t)^2$, we have the scalar form of \eqref{eq:for_zeta}:
\begin{equation*}
    (y^2-1)\eta_{yy} + 2y\eta_{y\tau} + 2y \eta_y + \eta_{\tau} + \eta_{\tau \tau} + \frac{2\alpha_0}{\sqrt{1+\alpha_0}+y}\eta_y = -\frac{(n-1)(n-3)}{4}\left(\frac{T_0 e^{-\tau}}{r_0+yT_0e^{-\tau}}\right)^2 + (\eta_y)^2.
\end{equation*}
Indeed, we see that as $\tau \to \infty$, the equation is asymptotically autonomous with the principal operator given by the linearised operator $\mathbf{L}_{\alpha_0}^{(\alpha_0)}$ studied in Section \ref{section_two} around the singular solutions to the one-dimensional equation. The equation for $\mathbf{q}$ is then
\begin{equation}\label{eq:for_q_sim_vector_form}
    \partial_{\tau}\mathbf{q} = \mathbf{L}_{\alpha_0}\mathbf{q} + \mathbf{\mathcal{E}}_{n,T_0}(\tau,y) + \mathbf{N}(\mathbf{q})
\end{equation}
with \begin{equation}
    \mathbf{\mathcal{E}}_{n,T_0}(\tau,y) = \begin{pmatrix}
        0 \\
        \nu_n \sigma_{T_0}(\tau,y)
    \end{pmatrix}, \qquad \nu_n \coloneqq -\frac{(n-1)(n-3)}{4}, \qquad \sigma_{T_0}(\tau,y) = \frac{\left(\frac{T_0}{r_0}\right)^2e^{-2\tau}}{(1+\frac{T_0}{r_0}e^{-\tau}y)^2},
\end{equation}
the linearised operator:
\begin{equation}
    \mathbf{L}_{\alpha_0} \coloneqq \begin{pmatrix}
        -y\partial_y & 1 \\
        \partial_{yy} + 2(\partial_y U_{\alpha_0,\kappa_0})\partial_y & -1 -y\partial_y 
    \end{pmatrix}
\end{equation}
and, $\mathbf{N}(\mathbf{q})$, the nonlinearity 
\begin{equation*}
    \begin{pmatrix}
        0 \\
        (\partial_y q_1)^2
    \end{pmatrix}.
\end{equation*}
\subsection{A special solution for \texorpdfstring{$n=3$}{n=3}}
Before we look at the general dimensional case $n\geq 2$, including $n=3$, we note that there is a canonical closed-form solution for $n=3$. Indeed, since $\nu_3 = 0$, then any solution $u$ solving 
\begin{equation*}
    u_{tt} - u_{rr} = (u_r)^2
\end{equation*}
provides an exact solution to 
\begin{equation}\label{eq:radial_n=3}
    v_{tt} - v_{rr} - \frac{2}{r}v_r = (v_r)^2
\end{equation}
by the simple transformation 
\begin{equation*}
    v(r,t) = u(r,t) - \log\left(\frac{r}{r_0}\right).
\end{equation*}
Thus, in particular, we have the family of closed-form, radial blow-up solutions to \eqref{eq:radial_n=3} for three spatial dimensions
\begin{equation}\label{special_n=3_family}
    v^{(3)}_{\alpha,\kappa,T}(r,t) = -\alpha\log\left(1-\frac{t}{T}\right)-\alpha\log\left(\sqrt{1+\alpha}+\frac{r-r_0}{T-t}\right)-\log\left(\frac{r}{r_0}\right) + \kappa,
\end{equation}
which, in similarity coordinates, are given by
\begin{equation}\label{special_n=3_family:sim_coordinates}
    v^{(3)}_{\alpha,\kappa,T}(r,t)\coloneqq V_{\alpha,\kappa,T}^{(3)}(\tau,y) = \alpha\tau -\alpha\log(\sqrt{1+\alpha}+y) -\log\left(1+\frac{T}{r_0}ye^{-\tau}\right)+\kappa.
\end{equation}
Given $r_0>0$, the extra term $\log(\frac{r}{r_0})$ is smooth in the $(r,t)$-variable backward light cone $\Gamma^{(1)}(r_0,T)$ associated to \eqref{eq:radial_n=3} when $\frac{T}{r_0}<1$. Equivalently, it is well defined in shrinking annuli in $(x,t) \in \mathbb{R}^{n+1}$ which do not intersect the origin. We will elaborate on this geometry more precisely in the next subsection. We choose the normalising constant $-\log(r_0)$ so that the correction vanishes as $r\to r_0$. I.e., in similarity coordinates this new term decays as we approach the blow-up set; $\log(\frac{r}{r_0}) \to 0$ as $\tau \to \infty$. 

Moreover, one computes that in similarity coordinates and vectorial form, the linearised operator associated to \eqref{eq:radial_n=3} around the family \eqref{special_n=3_family} is precisely $\mathbf{L}_{\alpha_0}$. Thus, we will show how asymptotic stability of $v^{(3)}_{\alpha,\kappa,T}$ follows as a corollary.

The solution \eqref{special_n=3_family}-\eqref{special_n=3_family:sim_coordinates} will be seen as a special case for the trivial solution $\zeta \equiv 0$ and $\mathbf{q}=\mathbf{0}$. However, as is our main task, we show stability of solutions whose correction is non-explicit, i.e. those which contain a non-zero, non-explicit $\mathbf{q}$. 

\subsection{Nonlinear correction in all dimensions}
To solve \eqref{eq:for_q_sim_vector_form} we will use Banach's fixed point theorem, for which we need several estimates. 
\begin{lemma}\label{lem:estimates_v2}
        Let $\alpha_0>0$ and $k\geq 0$. Then for any $0<T< \frac{r_0}{2R_{\alpha_0}}$, $\tau \geq 0$, $y\in [-R_{\alpha_0},R_{\alpha_0}]$ we have 
\begin{equation*}
    \|\sigma_{T}(\tau,\cdot)\|_{H^k(-R_{\alpha_0},R_{\alpha_0})} \lesssim_{k} \left(\frac{T}{r_0}\right)^2 e^{-2\tau} 
\end{equation*}
and so
\begin{equation*}
    \|\mathbf{\mathcal{E}}_{n,T}(\tau,\cdot)\|_{\mathcal{H}^k(-R_{\alpha_0},R_{\alpha_0})}\lesssim_{n,k}\left(\frac{T}{r_0}\right)^2e^{-2\tau}. 
\end{equation*}
We also have the following parameter-Lipschitz estimate: for $0<T_1,T_2<\frac{r_0}{2R_{\alpha_0}}$
\begin{equation}\label{eq:lipschitz_param_E}
    \|\mathcal{E}_{n,T_1}(\tau,\cdot)-\mathcal{E}_{n,T_2}(\tau,\cdot)\|_{\mathcal{H}_{\alpha_0}^k}\lesssim_{\alpha_0,k} e^{-2\tau}\left(\frac{T_1+T_2}{r_0} \right)\left|\frac{T_1}{r_0}-\frac{T_2}{r_0}\right|
\end{equation}
Finally, for $k\geq 2$, $\mathbf{q},\mathbf{q}'\in \mathcal{H}^k(\mathcal{I}_{\alpha_0})$
\begin{align*}
                \|\mathbf{N}(\mathbf{q})\|_{\mathcal{H}^k} &\lesssim_k \|\mathbf{q}\|_{\mathcal{H}^k}^2 \\
            \|\mathbf{N}(\mathbf{q}) - \mathbf{N}(\mathbf{q}')\|_{\mathcal{H}^k} &\lesssim_k (\|\mathbf{q}\|_{\mathcal{H}^k}+\|\mathbf{q}'\|_{\mathcal{H}^k})\|\mathbf{q}-\mathbf{q}'\|_{\mathcal{H}^k}.
            \end{align*}
\end{lemma}
\begin{proof}
    Fix $k\geq 0$, $\alpha_0>0$ and let $R_{\alpha_0}>1$ be given from Section \ref{section_two}. For $0\leq j \leq k$, computing the $j-th$ derivative of $\sigma_{T}$ in $y$, we have
    \begin{equation*}
        \partial_y^j \sigma_{T}(\tau,\cdot) = (-1)^j(j+1)!\frac{\left(\frac{T}{r_0}e^{-\tau}\right)^{j+2}}{(1+\frac{T}{r_0}e^{-\tau}y)^{j+2}}.
    \end{equation*}
    Thus, for concreteness, to keep the dynamics uniformly bounded away from $r=0$ we consider only $0<T<\frac{r_0}{2R_{\alpha_0}}<\frac{r_0}{2}$. (Indeed, one could choose any $\theta \in (0,1)$ so that $\frac{T}{r_0}<\frac{\theta}{R_{\alpha_0}}$, but we let $\theta =\frac{1}{2}$ for simplicity). Therefore 
    \begin{equation*}
        \|\partial_y^j \sigma_{T}(\tau,\cdot)\|_{L^\infty(\mathcal{I}_{\alpha_0})} \lesssim_j \frac{\left(\frac{T}{r_0}e^{-\tau}\right)^{j+2}}{(1+\frac{T}{r_0}e^{-\tau}y)^{j+2}} \leq 2\left(\frac{T}{r_0}e^{-\tau}\right)^{j+2}.
    \end{equation*}
    Since we work on a compact interval $y\in \mathcal{I}_{\alpha_0}$, we use the simple bound
    \begin{equation*}
        \|\sigma_{T}(\tau,\cdot)\|_{H^k(\mathcal{I}_{\alpha_0})} \lesssim_{\alpha_0}\|\sigma_{T}(\tau,\cdot)\|_{W^{k,\infty}(\mathcal{I}_{\alpha_0})} \lesssim_{\alpha_0,k}\left(\frac{T}{r_0}e^{-\tau}\right)^{2}.
    \end{equation*}
    Then, by definition of the product norm $\mathcal{H}^k$ we have
    \begin{equation*}
        \|\mathbf{\mathcal{E}}_{n,T}(\tau,\cdot)\|_{\mathcal{H}^k(\mathcal{I}_{\alpha_0})} \lesssim_n \|\sigma_{T}(\tau,\cdot)\|_{H^k}\lesssim_{n,\alpha_0,k}\left(\frac{T}{r_0}e^{-\tau}\right)^2. 
    \end{equation*}
    The estimates for the nonlinearity are the same as in the one-dimensional analysis.
    To prove \eqref{eq:lipschitz_param_E}, we first let $0<T_1,T_2<\frac{r_0}{2R_{\alpha_0}}$, then for $0\leq j \leq k$
    \begin{equation*}
        \partial_y^j \sigma_{T} = (-1)^j(j+1)!e^{-(j+2)\tau}\left(\frac{T}{r_0+Te^{-\tau}y}\right)^{j+2}
    \end{equation*}
    so we have 
    \begin{align*}
        |\partial_y^j \sigma_{T_1}(\tau,y)-\partial_y^j\sigma_{T_2}(\tau,y)| = (j+1)!e^{-(j+2)\tau}\left|\left(\frac{T_1}{r_0+T_1e^{-\tau}y}\right)^{j+2} - \left(\frac{T_2}{r_0+T_2e^{-\tau}y}\right)^{j+2}\right| \\
        =(j+1)! e^{-(j+2)\tau}\left|\frac{r_0(T_1-T_2)}{(r_0+T_1e^{-\tau}y)(r_0+T_2e^{-\tau}y)}\sum_{\ell=0}^{j+1}\left(\frac{T_1}{r_0+T_1e^{-\tau}y}\right)^{j+1-\ell}\left(\frac{T_2}{r_0+T_2e^{-\tau}y}\right)^{\ell}\right|.
    \end{align*}
    We bound the sum as 
    \begin{equation*}
        \sum_{\ell=0}^{j+1}\left(\frac{T_1}{r_0+T_1e^{-\tau}y}\right)^{j+1-\ell}\left(\frac{T_2}{r_0+T_2e^{-\tau}y}\right)^{\ell} \lesssim_{\alpha_0,j}\left(\frac{T_1+T_2}{r_0} \right)^{j+1}\lesssim_{\alpha_0,j} \left(\frac{T_1+T_2}{r_0} \right)
    \end{equation*}
    and using the uniform bound on the denominator again, for all $y\in [-R_{\alpha_0},R_{\alpha_0}]$ one has 
    \begin{align*}
        |\partial_y^j \sigma_{T_1}(\tau,y)-\partial_y^j\sigma_{T_2}(\tau,y)| \lesssim_{j,\alpha_0} e^{-2\tau}\left(\frac{T_1+T_2}{r_0} \right)\left(\frac{|T_1-T_2|}{r_0}\right). 
    \end{align*}
    Summing over the $L^2(\mathcal{I}_{\alpha_0})$-norms for $0\leq j \leq k$
    \begin{equation*}
        \|\sigma_{T_1}(\tau,\cdot)-\sigma_{T_2}(\tau,\cdot)\|_{H^k} \lesssim_{\alpha_0,k}e^{-2\tau}\left(\frac{T_1+T_2}{r_0} \right)\frac{|T_1-T_2|}{r_0}
    \end{equation*}
    which implies \eqref{eq:lipschitz_param_E}.
\end{proof}
We now introduce the function spaces on which we carry out the fixed point arguments. Namely, solutions whose $\mathcal{H}^k$-norm decays exponentially in similarity time at the rate $\omega=\omega_0-\gamma>0$ given in Proposition \ref{prop:linearised_evolution}:
\begin{align*}
    \mathcal{X} ^k(-R_{\alpha_0},R_{\alpha_0}) &\coloneqq \{\mathbf{q}\in C\left([0,\infty);\mathcal{H}^k(-R_{\alpha_0},R_{\alpha_0}) \right):\|\mathbf{q}\|_{\mathcal{X}_{\alpha_0}^k} \coloneqq \sup_{\tau \geq 0}e^{\omega \tau}\|\mathbf{q}(\tau)\|_{\mathcal{H}_{\alpha_0}^k}<\infty\}.
\end{align*}
We also define the closed balls centred at $0$ of size $\varepsilon >0$ for the respective function spaces as  
    \begin{align*}
        B_{\varepsilon}^{\mathcal{H}^k}(-R_{\alpha_0},R_{\alpha_0}) &\coloneqq \{\mathbf{f}\in \mathcal{H}^k(-R_{\alpha_0},R_{\alpha_0}):\|\mathbf{f}\|_{\mathcal{H}_{\alpha_0}^k}\leq \varepsilon\} \\
        B_{\varepsilon}^{\mathcal{X}^k}(-R_{\alpha_0},R_{\alpha_0}) &\coloneqq \{\mathbf{q}\in \mathcal{X}^k(-R_{\alpha_0},R_{\alpha_0}): \|\mathbf{q}\|_{\mathcal{X}_{\alpha_0}^k}\leq \varepsilon\}.
    \end{align*}
We now wish to construct solutions to the nonlinear problem 
\begin{equation}\label{eq:sec_3_eq_radial}
    \partial_{\tau}\mathbf{q} = \mathbf{L}_{\alpha_0}\mathbf{q} + \mathbf{\mathcal{E}}_{n,T_0}(\tau,y) + \mathbf{N}(\mathbf{q})
\end{equation}
for some initial condition that will be determined later. By Duhamel's principle, solutions to \eqref{eq:sec_3_eq_radial} satisfy 
\begin{equation}
    \mathbf{q}_{\alpha_0,T_0}(\tau,y) = \mathbf{S}_{\alpha_0}(\tau)\mathbf{q}_{\alpha_0,T_0}(0,y) + \int_{0}^{\tau}\mathbf{S}_{\alpha_0}(\tau-\tau')\bigl[\mathbf{\mathcal{E}}_{n,T_0}(\tau',y) + \mathbf{N}(\mathbf{q}_{\alpha_0,T_0}(\tau',y))\bigr]d\tau'.
\end{equation}
Due to the presence of the unstable modes of the linearised operator $\mathbf{L}_{\alpha_0}$ as seen in Proposition \ref{prop:linearised_evolution}, we need to suppress such unstable trajectories. This leads us to a modified Duhamel formulation. Given a small but free `data' $\mathbf{q}_0\in \mathcal{H}^k$ we consider 
\begin{equation*}
    \mathbf{q}_{\alpha_0,T_0}(\tau,y) = \mathbf{S}_{\alpha_0}(\tau)(\mathbf{q}_0-\mathbf{C}_{\alpha_0,T_0}[\mathbf{q}_0,\mathbf{q}_{\alpha_0,T_0}]) + \int_{0}^{\tau}\mathbf{S}_{\alpha_0}(\tau-\tau')\bigl[\mathbf{\mathcal{E}}_{n,T_0}(\tau',y) + \mathbf{N}(\mathbf{q}_{\alpha_0,T_0}(\tau',y))\bigr]d\tau'
\end{equation*}
where the \textit{correction operator} $\mathbf{C}_{\alpha,T}:\mathcal{H}^k \times \mathcal{X}^k \to \mathcal{H}^k$ is given by 
\begin{align*}
    \mathbf{C}_{\alpha,T}[\mathbf{q}_0,\mathbf{q}] =  &\mathbf{P}_{\alpha}^{(\alpha_0)}\mathbf{q}_0 + \mathbf{P}_{0,\alpha}^{(\alpha_0)}\int_{0}^{\infty}\mathbf{\mathcal{E}}_{n,T}(\tau') + \mathbf{N}(\mathbf{q}(\tau'))d\tau' + \mathbf{L}_{\alpha}^{(\alpha_0)}\mathbf{P}_{0,\alpha}^{(\alpha_0)}\int_{0}^{\infty}(-\tau')\bigl[\mathbf{\mathcal{E}}_{n,T}(\tau') + \mathbf{N}(\mathbf{q}(\tau'))\bigr]d\tau' \\
    &+ \mathbf{P}_{1,\alpha}^{(\alpha_0)}\int_{0}^{\infty}e^{-\tau'}\bigl[\mathbf{\mathcal{E}}_{n,T}(\tau') + \mathbf{N}(\mathbf{q}(\tau'))\bigr]d\tau'.
\end{align*}
We will see that by modifying the Duhamel equation in this way, the right-hand side becomes contractive. 
\begin{proposition}\label{prop:existence_of_correction}
    Let $\alpha_0>0$ and $k \geq \lceil c_{\alpha_0} \rceil +1$. Then there exist constants $\varepsilon>0$, $C>1$ such that for all $\mathbf{q}_0 \in B_{\frac{\varepsilon}{C}}^{\mathcal{H}_{\alpha_0}^k}$, $\alpha \in [\alpha_0-\delta_0,\alpha_0+\delta_0]$, and $0<T<\min\{\frac{1}{2R_{\alpha_0}},\varepsilon\}r_0 \coloneqq \overline{T}_1$, there exists a unique solution $\mathbf{q}_{\alpha,T}\in B_{\varepsilon}^{\mathcal{X}_{\alpha_0}^k}$ satisfying 
    \begin{equation*}
    \mathbf{q}_{\alpha,T}(\tau,y) = \mathbf{S}_{\alpha}(\tau)(\mathbf{q}_0-\mathbf{C}_{\alpha,T}[\mathbf{q}_0,\mathbf{q}_{\alpha,T}]) + \int_{0}^{\tau}\mathbf{S}_{\alpha}(\tau-\tau')\bigl[\mathbf{\mathcal{E}}_{n,T}(\tau',y) + \mathbf{N}(\mathbf{q}_{\alpha,T}(\tau',y))\bigr]d\tau'
\end{equation*}
for all $\tau \geq 0$. In fact, $\mathbf{q}_{\alpha,T}$ satisfies the sharper bound 
\begin{equation}
    \|\mathbf{q}_{\alpha,T}\|_{\mathcal{X}_{\alpha_0}^k} \lesssim_{n,k,\omega_0,\gamma,\alpha_0}\|\mathbf{q}_0\|_{\mathcal{H}_{\alpha_0}^k} + \left(\frac{T}{r_0}\right)^2.
\end{equation}Moreover we have the following two properties:
\begin{itemize}
    \item The data-to-solution map $\mathbf{q}_0 \mapsto \mathbf{q}_{\alpha,T}$ is Lipschitz continuous.
    \item The map $(\alpha,T) \mapsto \mathbf{q}_{\alpha,T}$ is Lipschitz continuous; for $\alpha_1,\alpha_2 \in [\alpha_0-\delta_0,\alpha_0+\delta_0]$, $0<T_1,T_2<\min\{\frac{1}{2R_{\alpha_0}},\varepsilon\}r_0$ we have 
    \begin{equation}\label{eq:lipschitz_in_params}
        \|\mathbf{q}_{\alpha_1,T_1}-\mathbf{q}_{\alpha_2,T_2}\|_{\mathcal{X}_{\alpha_0}^k}\lesssim |\alpha_1-\alpha_2|\|\mathbf{q}_0\|_{\mathcal{H}_{\alpha_0}^k} + \frac{T_1+T_2}{r_0}\left(\frac{|T_1-T_2|}{r_0} +|\alpha_1-\alpha_2|\right)
    \end{equation}
\end{itemize}
\end{proposition}
\begin{proof}
Using the definition of the correction operator $\mathbf{C}_{\alpha,T}$, $\mathbf{P}_{\alpha}$ and $\widetilde{\mathbf{P}}_{\alpha}$, the equation we want to solve can be written as 
\begin{equation}
\begin{aligned}
\mathbf{q}_{\alpha,T}(\tau,y)
&= \mathbf{S}_{\alpha}(\tau)\Bigl(
    \mathbf{q}_0-\mathbf{P}_{\alpha}\mathbf{q}_0
    - \mathbf{P}_{0,\alpha}\int_{0}^{\infty}\bigl[\mathbf{\mathcal{E}}_{n,T}(\tau') + \mathbf{N}(\mathbf{q}(\tau'))\bigr]d\tau' - \mathbf{L}_{\alpha}\mathbf{P}_{0,\alpha}
      \int_{0}^{\infty}(-\tau')\bigl[\mathbf{\mathcal{E}}_{n,T}(\tau') + \mathbf{N}(\mathbf{q}(\tau'))\bigr]d\tau' \\
&
    - \mathbf{P}_{1,\alpha}
      \int_{0}^{\infty}e^{-\tau'}\bigl[\mathbf{\mathcal{E}}_{n,T}(\tau') + \mathbf{N}(\mathbf{q}(\tau'))\bigr]d\tau'\Bigr)
+ \int_{0}^{\tau}
  \mathbf{S}_{\alpha}(\tau-\tau')
  (\mathbf{P}_{\alpha}+\widetilde{\mathbf{P}}_{\alpha})
  \bigl[\mathbf{\mathcal{E}}_{n,T}(\tau') + \mathbf{N}(\mathbf{q}(\tau'))\bigr]d\tau'.
\end{aligned}
\end{equation}

Now, using the identities for $\mathbf{S}_{\alpha}\mathbf{P}_{\lambda,\alpha}$
given in Proposition \ref{prop:linearised_evolution}, and simplifying the resulting cancellations, we arrive at the more compact expression
\begin{align*}
\mathbf{q}_{\alpha,T}(\tau,y) &= \mathbf{S}_{\alpha}(\tau)\widetilde{\mathbf{P}}_{\alpha}\mathbf{q}_0 + \int_{0}^{\tau}\mathbf{S}_{\alpha}(\tau-\tau')\widetilde{\mathbf{P}}_{\alpha}\bigl[\mathbf{\mathcal{E}}_{n,T}(\tau') + \mathbf{N}(\mathbf{q}(\tau'))\bigr]d\tau' - \mathbf{P}_{0,\alpha}\int_{\tau}^{\infty}\mathbf{\mathcal{E}}_{n,T}(\tau') + \mathbf{N}(\mathbf{q}(\tau'))d\tau' \\
&- \mathbf{L}_{\alpha}\mathbf{P}_{0,\alpha}\int_{\tau}^{\infty}(\tau-\tau')\bigl[ \mathbf{\mathcal{E}}_{n,T}(\tau') + \mathbf{N}(\mathbf{q}(\tau'))\bigr]d\tau' -\mathbf{P}_{1,\alpha}\int_{\tau}^{\infty}e^{\tau-\tau'}\bigl[\mathbf{\mathcal{E}}_{n,T}(\tau') + \mathbf{N}(\mathbf{q}(\tau'))\bigr]d\tau'.
\end{align*}
If we denote as $\mathbf{\Phi}(\mathbf{q}_0,\mathbf{q})$ the fixed point operator, i.e. the right hand side, we claim that there exists a sufficiently small $\varepsilon>0$, and $C>1$ sufficiently large such that if $\mathbf{q}_0 \in B_{\frac{\varepsilon}{C}}^{\mathcal{H}^k}$, $\alpha \in [\alpha_0-\delta_0,\alpha_0+\delta_0]$, and $0<\frac{T}{r_0}<\min\{\frac{1}{2R_{\alpha_0}},\varepsilon\}$, then $\mathbf{\Phi}(\mathbf{q}_0,\cdot):B_{\varepsilon}^{\mathcal{X}^k}\to B_{\varepsilon}^{\mathcal{X}^k} $ is a contraction map. This will imply the existence and uniqueness of a fixed point by Banach's fixed point theorem.
Thus, we have 
\begin{align*}
    \mathbf{\Phi}_{\alpha,T}(\mathbf{q}_0,\mathbf{q})  &= \mathbf{S}_{\alpha}(\tau)\widetilde{\mathbf{P}}_{\alpha}\mathbf{q}_0 + \int_{0}^{\tau}\mathbf{S}_{\alpha}(\tau-\tau')\widetilde{\mathbf{P}}_{\alpha}\bigl[\mathbf{\mathcal{E}}_{n,T}(\tau') + \mathbf{N}(\mathbf{q}(\tau'))\bigr]d\tau' - \mathbf{P}_{0,\alpha}\int_{\tau}^{\infty}\mathbf{\mathcal{E}}_{n,T}(\tau') + \mathbf{N}(\mathbf{q}(\tau'))d\tau' \\
&- \mathbf{L}_{\alpha}\mathbf{P}_{0,\alpha}\int_{\tau}^{\infty}(\tau-\tau')\bigl[ \mathbf{\mathcal{E}}_{n,T}(\tau') + \mathbf{N}(\mathbf{q}(\tau'))\bigr]d\tau' -\mathbf{P}_{1,\alpha}\int_{\tau}^{\infty}e^{\tau-\tau'}\bigl[\mathbf{\mathcal{E}}_{n,T}(\tau') + \mathbf{N}(\mathbf{q}(\tau'))\bigr]d\tau'.
\end{align*}
Taking norms, using the estimates for $\mathcal{E}$, and $\mathbf{N}$ from Lemma \ref{lem:estimates_v2}, the semigroup estimates from Proposition \ref{prop:linearised_evolution}, and performing integration by parts we get 
\begin{align*}
    \|\mathbf{\Phi}_{\alpha,T}(\mathbf{q}_0,\mathbf{q}) (\tau)\|_{\mathcal{H}^k}&\lesssim_{n,k,\omega_0,\alpha_0} \left(\frac{\varepsilon}{C}e^{-\omega \tau} + \left(\frac{T}{r_0}\right)^2e^{-\omega\tau} + \varepsilon^2 e^{-2\omega\tau}\right) 
\end{align*}
for all $\tau\geq 0$. We now restrict to $T$ satisfying $\frac{T}{r_0}\leq \varepsilon$. Therefore, for $C>1$ sufficiently large and $\varepsilon>0$ sufficiently small, we can ensure that
\begin{equation*}
    \|\mathbf{\Phi}_{\alpha,T}(\mathbf{q}_0,\mathbf{q})(\tau)\|_{\mathcal{H}^k}\leq \varepsilon e^{-\omega \tau}
\end{equation*}
for all $\tau \geq 0$ which means $\mathbf{\Phi}_{\alpha,T}(\mathbf{q}_0,\mathbf{q}):B_{\varepsilon}^{\mathcal{X}^k}\to B_{\varepsilon}^{\mathcal{X}^k} $. We now show $\Phi_{\alpha,T}$ is a contraction. Indeed let $\mathbf{q},\mathbf{q}' \in B_{\varepsilon}^{\mathcal{X}^k}$, then by subtracting the iterate applied to these functions, we note that the terms with $\mathcal{E}_{n,T}$ cancel when we subtract, and the estimate follows:
\begin{align*}
    \|\mathbf{\Phi}_{\alpha,T}(\mathbf{q}_0,\mathbf{q}') - \mathbf{\Phi}_{\alpha,T}(\mathbf{q}_0,\mathbf{q})\|_{\mathcal{H}^k} \lesssim \varepsilon e^{-\omega \tau}\|\mathbf{q}-\mathbf{q}'\|_{\mathcal{X}^k}
\end{align*}
Possibly taking $\varepsilon>0$ even smaller, we can ensure that 
\begin{equation*}
    \|\mathbf{\Phi}_{\alpha,T}(\mathbf{q}_0,\mathbf{q})-\mathbf{\Phi}_{\alpha,T}(\mathbf{q}_0,\mathbf{q}')\|_{\mathcal{X}^k}\leq \frac{1}{2}\|\mathbf{q}-\mathbf{q}'\|_{\mathcal{X}^k}
\end{equation*}
Thus, with this $\varepsilon>0$, we have the existence and uniqueness of the fixed point. Let it be $\mathbf{q}$. We now derive the sharper bound. Indeed, we have $\mathbf{q}_{\alpha,T} = \mathbf{\Phi}_{\alpha,T}(\mathbf{q}_0,\mathbf{q}_{\alpha,T})$, then add and subtract the following term
\begin{equation*}
    \mathbf{q}_{\alpha,T} = \mathbf{\Phi}_{\alpha,T}(\mathbf{q}_0,\mathbf{q}_{\alpha,T}) - \mathbf{\Phi}_{\alpha,T}(\mathbf{q}_0,\mathbf{0}) + \mathbf{\Phi}_{\alpha,T}(\mathbf{q}_0,\mathbf{0}).
\end{equation*}
Since trivially $\mathbf{0}\in B_{\varepsilon}^{\mathcal{X}^k}$, then the triangle inequality and the contraction estimate gives
\begin{equation*}
    \|\mathbf{q}_{\alpha,T}\|_{\mathcal{X}^k} \leq \frac{1}{2}\|\mathbf{q}_{\alpha,T}\|_{\mathcal{X}^k}+ \|\mathbf{\Phi}_{\alpha,T}(\mathbf{q}_0,\mathbf{0})\|_{\mathcal{X}^k}.
\end{equation*}
It thus remains to bound $\|\mathbf{\Phi}_{\alpha,T}(\mathbf{q}_0,\mathbf{0})\|_{\mathcal{X}^k}$. Indeed, we obtain
\begin{equation*}
    \|\mathbf{\Phi}_{\alpha,T}(\mathbf{q}_0,\mathbf{0}) (\tau)\|_{\mathcal{H}^k}\lesssim_{n,k,\omega_0,\alpha_0} \left(\|\mathbf{q}_0\|_{\mathcal{H}^k} + \left(\frac{T}{r_0}\right)^2 \right) e^{-\omega\tau}
\end{equation*}
for all $\tau \geq 0$. This means that $\|\mathbf{\Phi}_{\alpha,T}(\mathbf{q}_0,\mathbf{0})\|_{\mathcal{X}^k}\lesssim_{n,k,\omega_0,\alpha_0} \|\mathbf{q}_0\|_{\mathcal{H}^k}+\left(\frac{T}{r_0}\right)^2$ which proves the estimate. 

We now prove that the solution depends Lipschitz continuously on initial data. Let $\mathbf{q}_0, \mathbf{q}_0' \in B_{\frac{\varepsilon}{C}}^{\mathcal{H}_{\alpha_0}^k}$, and for $\alpha \in [\alpha_0-\delta_0,\alpha_0+\delta_0]$, $0<T<\overline{T}_1$ the associated $\mathbf{\Phi}_{\alpha,T}$-fixed points given by $\mathbf{q}_{\alpha,T},\mathbf{q}_{\alpha,T}' \in B_{\varepsilon}^{\mathcal{X}_{\alpha_0}^k}$. Then, adding and subtracting $\mathbf{\Phi}_{\alpha,T}(\mathbf{q}_0,\mathbf{q}_{\alpha,T}')$, we have 
\begin{align*}
    \|\mathbf{q}_{\alpha,T}-\mathbf{q}_{\alpha,T}'\|_{\mathcal{X}_{\alpha_0}^k}&\leq \|\mathbf{\Phi}_{\alpha,T}(\mathbf{q}_0,\mathbf{q}_{\alpha,T})-\mathbf{\Phi}_{\alpha,T}(\mathbf{q}_0,\mathbf{q}_{\alpha,T}')\|_{\mathcal{X}_{\alpha_0}^k}+\|\mathbf{\Phi}_{\alpha,T}(\mathbf{q}_0,\mathbf{q}_{\alpha,T}')-\mathbf{\Phi}_{\alpha,T}(\mathbf{q}_0',\mathbf{q}_{\alpha,T}')\|_{\mathcal{X}_{\alpha_0}^k} \\
    &\leq \frac{1}{2}\|\mathbf{q}_{\alpha,T}-\mathbf{q}_{\alpha,T}'\|_{\mathcal{X}_{\alpha_0}^k}+ \|\mathbf{S}_{\alpha}(\tau)\widetilde{\mathbf{P}}_{\alpha}(\mathbf{q}_0-\mathbf{q}_0')\|_{\mathcal{X}_{\alpha_0}^k}.
\end{align*}
Now, by definition of the $\mathcal{X}^k$ norm, we have 
\begin{equation*}
    \|\mathbf{S}_{\alpha}(\tau)\widetilde{\mathbf{P}}_{\alpha}(\mathbf{q}_0-\mathbf{q}_0')\|_{\mathcal{X}_{\alpha_0}^k} \lesssim \|\widetilde{\mathbf{P}}_{\alpha}(\mathbf{q}_0-\mathbf{q}_0')\|_{\mathcal{H}_{\alpha_0}^k} \lesssim \|\mathbf{q}_0-\mathbf{q}_0'\|_{\mathcal{H}_{\alpha_0}^k}
\end{equation*}
which proves the estimate. 

Finally we prove Lipschitz dependence of the solution on the parameters for fixed initial $\mathbf{q}_0$. Given $\mathbf{q}_0 \in B_{\frac{\varepsilon}{C}}^{\mathcal{H}_{\alpha_0}^k}$, $\alpha_1,\alpha_2 \in [\alpha_0-\delta_0,\alpha_0+\delta_0]$, and $0<T_1,T_2<\overline{T}_1$ we obtain the fixed points $\mathbf{q}_{\alpha_1,T_1}=\mathbf{\Phi}_{\alpha_1,T_1}(\mathbf{q}_0,\mathbf{q}_{\alpha_1,T_1})$ and $\mathbf{q}_{\alpha_2,T_2}=\mathbf{\Phi}_{\alpha_2,T_2}(\mathbf{q}_0,\mathbf{q}_{\alpha_2,T_2})$. Adding and subtracting $\mathbf{\Phi}_{\alpha_1,T_1}(\mathbf{q}_0,\mathbf{q}_{\alpha_2,T_2})$ one has 
\begin{equation*}
    \|\mathbf{q}_{\alpha_1,T_1}-\mathbf{q}_{\alpha_2,T_2}\|_{\mathcal{X}_{\alpha_0}^k} \leq \frac{1}{2}\|\mathbf{q}_{\alpha_1,T_1}-\mathbf{q}_{\alpha_2,T_2}\|_{\mathcal{X}_{\alpha_0}^k} + \|\mathbf{\Phi}_{\alpha_2,T_2}(\mathbf{q}_0,\mathbf{q}_{\alpha_2,T_2})-\mathbf{\Phi}_{\alpha_1,T_1}(\mathbf{q}_0,\mathbf{q}_{\alpha_2,T_2})\|_{\mathcal{X}_{\alpha_0}^k}. 
\end{equation*}
Thus, it remains to consider the second term. Indeed, for simplicity we bound each term of the difference separately. Firstly, by Proposition \ref{prop:linearised_evolution}, for all $\tau \geq 0$
\begin{align*}
     \|(\mathbf{S}_{\alpha_1}(\tau)\widetilde{\mathbf{P}}_{\alpha_1}-\mathbf{S}_{\alpha_2}(\tau)\widetilde{\mathbf{P}}_{\alpha_2})(\mathbf{q}_0)\|_{\mathcal{H}_{\alpha_0}^k} \lesssim e^{-\omega \tau}|\alpha_1-\alpha_2|\|\mathbf{q}_0\|_{\mathcal{H}_{\alpha_0}^k}.
\end{align*}
Secondly, adding and subtracting alternate terms and using the Lipschitz dependence of \eqref{eq:lipschitz_param_E}
\begin{align*}
    &\left\|\int_{0}^{\tau}\mathbf{S}_{\alpha_1}(\tau-\tau')\widetilde{\mathbf{P}}_{\alpha_1}\left[\mathcal{E}_{n,T_1}(\tau',y)+\mathbf{N}(\mathbf{q}_{\alpha_2,T_2}(\tau'))\right]-\mathbf{S}_{\alpha_2}(\tau-\tau')\widetilde{\mathbf{P}}_{\alpha_2}\left[\mathcal{E}_{n,T_2}(\tau',y)+\mathbf{N}(\mathbf{q}_{\alpha_2,T_2}(\tau'))\right]d\tau'\right\|_{\mathcal{H}_{\alpha_0}^k} \\
    &\lesssim \int_{0}^{\tau}e^{-\omega(\tau-\tau')}\|\mathcal{E}_{n,T_1}(\tau',\cdot)-\mathcal{E}_{n,T_2}(\tau',\cdot)\|_{\mathcal{H}_{\alpha_0}^k} + e^{-\omega(\tau-\tau')}|\alpha_1-\alpha_2|\|\mathcal{E}_{n,T_2}(\tau',\cdot)\|_{\mathcal{H}_{\alpha_0}^k} \\
    &\qquad \qquad + \|\left[\mathbf{S}_{\alpha_1}(\tau-\tau')\widetilde{\mathbf{P}}_{\alpha_1}-\mathbf{S}_{\alpha_2}(\tau-\tau')\widetilde{\mathbf{P}}_{\alpha_2}\right]\mathbf{N}(\mathbf{q}_{\alpha_2,T_2})\|_{\mathcal{H}_{\alpha_0}^k}d\tau'\\
    &\lesssim e^{-\omega \tau}\left(\left(\frac{T_1+T_2}{r_0} \right)\frac{|T_1-T_2|}{r_0}+\left(\frac{T_2}{r_0}\right)^2|\alpha_1-\alpha_2|\right).
\end{align*}
The other terms are estimated in exactly the same way. 
\end{proof}
\subsection{Proof of Theorem \ref{Existence_of_sol_Main_Theorem}}
Let $n\geq 2$, $r_0>0$ and $\omega_0 \in (0,1)$. Take also $\alpha_0>0$, $k\geq \lceil c_{\alpha_0}\rceil +1$ and $0<\gamma<\omega_0$. From Proposition \ref{prop:existence_of_correction}, there exists a constant $0<\varepsilon<\frac{1}{2R_{\alpha_0}}$ so that for any $0<T_0<\varepsilon r_0$, we may take $\mathbf{q}_0 = \mathbf{0}$, and there exists a solution satisfying 
\begin{equation*}
    \mathbf{q}_{\alpha_0,T_0}(\tau,y) = \mathbf{S}_{\alpha_0}(\tau)(-\mathbf{C}_{\alpha_0,T_0}[\mathbf{q}_0,\mathbf{q}_{\alpha_0,T_0}]) + \int_{0}^{\tau}\mathbf{S}_{\alpha_0}(\tau-\tau')\bigl[\mathbf{\mathcal{E}}_{n,T_0}(\tau',y) + \mathbf{N}(\mathbf{q}_{\alpha_0,T_0}(\tau',y))\bigr]d\tau'
    \end{equation*}
    and 
    \begin{equation}\label{eq:bounds_in_proof}
\|\mathbf{q}_{\alpha_0,T_0}\|_{\mathcal{X}_{\alpha_0}^k}\lesssim_{n,k,\omega_0,\gamma,\alpha_0} \left(\frac{T_0}{r_0}\right)^2.
    \end{equation}
    This means 
    \begin{equation*}
        \mathbf{q}_{\alpha_0,T_0}(0,y) \in \rg(\mathbf{P}_{\alpha_0}^{(\alpha_0)}) = \text{span}\{\mathbf{g}_{0,\alpha_0},\mathbf{f}_{0,\alpha_0},\mathbf{f}_{1,\alpha_0}\} 
    \end{equation*}
    Since $\zeta_{\alpha_0,T_0}(r,t) = \mathbf{q}_{\alpha_0,T_0}^{(1)}(\tau,y)$, we recall the radial logarithmic transformation so that the solution to \eqref{Nd-equation} can be written as 
        \begin{equation*}
v_{\alpha_0,\kappa_0,T_0}^{(n)}(x,t) = -\alpha_0\log\left(1-\frac{t}{T_0}\right) -\alpha_0 \log\left(\sqrt{1+\alpha_0}+\frac{|x|-r_0}{T_0-t}\right) - \frac{(n-1)}{2}\log\left(\frac{|x|}{r_0}\right) + \zeta_{\alpha_0,T_0}(|x|,t) + \kappa_0.
\end{equation*}
    After passing to physical coordinates, the bounds \eqref{eq:bounds_in_proof} imply those in Theorem \ref{Existence_of_sol_Main_Theorem}.

    Moreover, for $n=3$, we have the closed-form family solution 
    \begin{equation*}
v_{\alpha_0,\kappa_0,T_0}^{(3)}(x,t) = -\alpha_0\log\left(1-\frac{t}{T_0}\right) -\alpha_0 \log\left(\sqrt{1+\alpha_0}+\frac{|x|-r_0}{T_0-t}\right) - \frac{(n-1)}{2}\log\left(\frac{|x|}{r_0}\right) + \kappa_0,
\end{equation*}
smooth in the backward light cone for $\frac{T_0}{r_0}<1$.
\section{Stability of the non-explicit radial blow-up family}\label{section_four}

We now turn to the stability of the family of solutions constructed in
Section \ref{section_three}. Although the analysis of Section
\ref{section_two} was carried out on the extended interval
\(\mathcal{I}_{\alpha_0}=(-R_{\alpha_0},R_{\alpha_0})\), in this section we return to
the light cone \(y\in[-1,1]\), for which the same spectral results remain
valid, see \cite{ghoul2025blow}.

Throughout this section we again fix arbitrarily 
\[
n\geq 2,\qquad r_0>0,\qquad \alpha_0>0,\qquad k\geq \ceil{c_{\alpha_0}}+1,\qquad
0<\omega_0<1,\qquad 0<\gamma<\omega_0.
\]

By Proposition \ref{prop:existence_of_correction}, there exists a constant
$\varepsilon>0$ such that, for every
\[
\alpha\in[\alpha_0-\delta_0,\alpha_0+\delta_0],\qquad
0<T<\min\Big\{\frac{1}{2R_{\alpha_0}},\varepsilon\Big\}r_0 \coloneqq \overline{T}_1,
\quad  \text{and} \quad \kappa\in\mathbb{R},
\]
equation \eqref{radial_eq_v} admits a solution (in similarity coordinates) of the form
\[
\mathcal{U}_{\alpha,\kappa,T}(\tau,y)
=
U_{\alpha,\kappa}(\tau,y) - \frac{n-1}{2}\log\left(1+\frac{T}{r_0}ye^{-\tau}\right)+\mathbf q_{\alpha,T}^{(1)}(\tau,y)
\]
with
\[
\|\mathbf q_{\alpha,T}\|_{\mathcal X^k(-R_{\alpha_0},R_{\alpha_0})}
\lesssim_{\omega_0,\gamma,\alpha_0,n,k,\theta}\left(\frac{T}{r_0}\right)^2
\]
Here, the $T$ in the similarity coordinates is the same $T$ appearing in the subscript of $\mathbf{q}_{\alpha,T}$. Moreover, we consider $\overline{T}_1$ as the first admissible $T$-range to ensure that the background solutions actually exist.

Furthermore, we note the canonical solution when $n=3$ given by 
\begin{equation*}
    U^{(3)}_{\alpha,\kappa,T}(\tau,y)=U_{\alpha,\kappa}(\tau,y) - \log\left(1+\frac{T}{r_0}ye^{-\tau}\right). 
\end{equation*}
\noindent Recall the equation
\begin{equation*}
    u_{tt} - u_{rr} - \frac{(n-1)}{r}u_r = (u_r)^2
\end{equation*}
    which reads in similarity coordinates (centred at parameter $T$) as 
\begin{equation}\label{similarity_variables_stability_sec}
    (y^2-1)U_{yy}+2yU_{y\tau} + 2y U_y + U_{\tau} + U_{\tau\tau} - \frac{(n-1)Te^{-\tau}}{yTe^{-\tau}+r_0}U_y = (U_y)^2.
\end{equation}
We now want to 
study the evolution of perturbations $V$ defined by
\[
  U(\tau,y) = \mathcal{U}_{\alpha,\kappa,T}(\tau,y) + V(\tau,y).
\]
We substitute the ansatz $U = \mathcal{U}_{\alpha,\kappa,T} + V$ into \eqref{similarity_variables_stability_sec}, so that $V$ satisfies 
\begin{equation*}
    (y^2-1)V_{yy} + 2yV_{y\tau} + 2yV_y + V_{\tau} + V_{\tau\tau} - \frac{(n-1)Te^{-\tau}}{yTe^{-\tau}+r_0}V_y - 2(\partial_y \mathcal{U}_{\alpha,\kappa,T})V_y = (V_y)^2
\end{equation*}
which, by definition of $\mathcal{U}_{\alpha,\kappa,T}$, is equivalent to
\begin{equation}
     (y^2-1)V_{yy} + 2yV_{y\tau} + 2yV_y + V_{\tau} + V_{\tau\tau} 
     - 2(\partial_y U_{\alpha,\kappa} + \partial_y\mathbf{q}_{\alpha,T}^{(1)})V_y = (V_y)^2.
     \label{eq_for_V}
\end{equation}
\begin{remark}
This is the second use of the radial logarithmic transformation: the explicit
logarithmic term cancels the radial drift also in the perturbation
equation. Consequently, the principal linear operator is again the
one-dimensional operator \(\mathbf{L}_\alpha\); the only additional linear term
comes from the non-explicit correction \(\mathbf{q}_{\alpha,T}\), and this term
decays in similarity time.
\end{remark}

Again the task is to construct a small and decaying $V$. Thus, to use our work from Section \ref{section_two}, we put this into first-order formulation, defining
\begin{align*}
    \mathbf{v} = \begin{pmatrix}
        V \\
        V_{\tau} + yV_y 
    \end{pmatrix} \coloneqq \begin{pmatrix}
        v_1 \\
        v_2
    \end{pmatrix}
\end{align*}
so that we have the vectorial evolution equation 
\begin{align*}
    \partial_{\tau}\begin{pmatrix}
        v_1 \\
        v_2 
    \end{pmatrix} = \mathbf{L}_{\alpha}\begin{pmatrix}
        v_1 \\
        v_2
    \end{pmatrix} +
    \begin{pmatrix}
       0 \\
       2(\partial_y \mathbf{q}_{\alpha,T}^{(1)})\partial_y v_1 
    \end{pmatrix} + 
    \begin{pmatrix}
        0 \\ 
        (\partial_y v_1)^2
    \end{pmatrix}
\end{align*}
and for compactness we denote as $\mathbf{H}_{\alpha,T}$ the small linear operator 
\begin{align*}
    \mathbf{H}_{\alpha,T}(\mathbf{v}) \coloneqq \begin{pmatrix}
        0 \\
        [2\partial_y \mathbf{q}_{\alpha,T}^{(1)}]\partial_y v_1 
    \end{pmatrix}
\end{align*}
which we will consider as small linear perturbation/forcing term. 
\begin{remark}
    If $\mathbf{q}_{\alpha,T}$ were stationary (i.e. exactly self-similar) a possible route would be to absorb $2(\partial_y \mathbf{q}_{\alpha,T}^{(1)})\partial_y v_1$ into a new $\mathbf{L}_{\alpha}$ and carry out a spectral analysis of a perturbed autonomous operator. This is not the case, however we can exploit the \textit{decay}  of $\mathbf{q}_{\alpha,T}$ and so we can consider the term arising from its linearisation as a perturbation to the original nonlinear evolution associated with $\mathbf{L}_{\alpha}$ and $\mathbf{N}(\cdot)$. 
\end{remark}
The Cauchy problem for $\mathbf{v}$ now reads as
\begin{align}\label{eq:abstract_cauchy_prob_v}
    \begin{cases}
\partial_{\tau}\mathbf{v} = \mathbf{L}_{\alpha}\mathbf{v} + \mathbf{H}_{\alpha,T}(\mathbf{v}) + \mathbf{N}(\mathbf{v}) \\
        \mathbf{v}\left(\tau =0, y\right) = \mathbf{v}_0(y).
    \end{cases}
\end{align}
As before, we introduce the correction term for the modified Duhamel formulation, defined in this case by 
\begin{align*}
    \mathbf{C}_{\alpha,T}[\mathbf{v}_0,\mathbf{v}] =& \mathbf{P}_{\alpha}\mathbf{v}_0 + \mathbf{P}_{0,\alpha}\int_{0}^{\infty}(\mathbf{H}_{\alpha,T}(\mathbf{v})+ \mathbf{N}(\mathbf{v}))\,d\tau' + \mathbf{L}_{\alpha}\mathbf{P}_{0,\alpha}\int_{0}^{\infty}(-\tau')(\mathbf{H}_{\alpha,T}(\mathbf{v})+ \mathbf{N}(\mathbf{v}))\,d\tau' \\
    +& \mathbf{P}_{1,\alpha}\int_{0}^{\infty}e^{-\tau'}(\mathbf{H}_{\alpha,T}(\mathbf{v})+ \mathbf{N}(\mathbf{v}))\,d\tau'
\end{align*}
and the modified Duhamel formulation
\begin{equation*}
    \mathbf{v}_{\alpha,T}(\tau,y) = \mathbf{S}_{\alpha}(\tau)(\mathbf{v}_0 - \mathbf{C}_{\alpha,T}[\mathbf{v}_0, \mathbf{v}_{\alpha,T}]) + \int_{0}^{\tau}\mathbf{S}_{\alpha}(\tau-\tau')(\mathbf{H}_{\alpha,T}(\mathbf{v}_{\alpha,T})+ \mathbf{N}(\mathbf{v}_{\alpha,T}))\,d\tau'.
\end{equation*}
In order to solve for $\mathbf{v}$ via a fixed point argument, we first state an estimate for $\mathbf{H}_{\alpha,T}(\mathbf{v})$.
\begin{lemma}\label{estimates_on_H}
    Let $k\geq 1$. For all $\alpha \in [\alpha_0-\delta_0,\alpha_0+\delta_0]$, $T\leq \overline{T}_1$ and $\mathbf{v} \in \mathcal{H}^k(-1,1)$
    \begin{equation*}
        \|\mathbf{H}_{\alpha,T}(\mathbf{v})\|_{\mathcal{H}^k} = 2\|\partial_y \mathbf{q}_{\alpha,T}^{(1)}\partial_y v_1\|_{H^k} \lesssim \left(\frac{T}{r_0}\right)^2 e^{-\omega \tau}\|\mathbf{v}\|_{\mathcal{H}^k}.
    \end{equation*}
    for all $\tau \geq 0$. 
\end{lemma}
\begin{proof}
    We simply recall the properties of the construction $\mathbf{q}_{\alpha,T}$ from Proposition \ref{prop:existence_of_correction}. For  $\alpha \in [\alpha_0-\delta_0,\alpha_0+\delta_0]$ and $0<T<\overline{T}_1$, we have the existence of the following functions satisfying the bound
    \begin{align*}
        \|\mathbf{q}_{\alpha,T}(\tau,\cdot)\|_{\mathcal{H}_{\alpha_0}^k} \lesssim \left(\frac{T}{r_0}\right)^2e^{-\omega \tau}.
    \end{align*}
    for all $\tau\geq 0$.
    Using the Banach algebra property of $H^k(-R_{\alpha_0},R_{\alpha_0})$ for $k>\frac{1}{2}$, we have 
    \begin{equation*}
        \|\mathbf{H}_{\alpha,T}(\mathbf{v})\|_{\mathcal{H}^k} = \|(2\partial_y \mathbf{q}_{\alpha,T}^{(1)})\partial_y v_1\|_{H^k} \lesssim_k \|\partial_y \mathbf{q}_{\alpha,T}^{(1)}\|_{H^k}\|\partial_y v_1\|_{H^k} \lesssim \|\mathbf{q}_{\alpha,T}\|_{\mathcal{H}^k}\|\mathbf{v}\|_{\mathcal{H}^k}.
    \end{equation*}
\end{proof}
\begin{proposition}\label{existence_of_v}
     There exist constants $0<\delta_1<\delta_0, \,M_1>1$ and $\overline{T}_2\leq \overline{T}_1$ with the following property. For every $0<\rho\leq \delta_1$, $\alpha \in [\alpha_0-\delta_0,\alpha_0+\delta_0]$, $0<T\leq\overline{T}_2$ and $\mathbf{v}_0\in \mathcal{H}^k(-1,1)$ with $\|\mathbf{v}_0\|_{\mathcal{H}^k}\leq \frac{\rho}{M_1}$, there is a unique solution $\mathbf{v}_{\alpha,T}\in B_{\rho}^{\mathcal{X}^k}$, satisfying  
    \begin{equation*}
        \mathbf{v}_{\alpha,T}(\tau_{T}) = \mathbf{S}_{\alpha}(\tau_{T})(\mathbf{v}_0 - \mathbf{C}_{\alpha,T}[\mathbf{v}_0,\mathbf{v}_{\alpha,T}]) + \int_{0}^{\tau_{T}}\mathbf{S}_{\alpha}(\tau_{T}-\tau')[\mathbf{H}_{\alpha,T}(\mathbf{v}_{\alpha,T}(\tau'))+\mathbf{N}(\mathbf{v}_{\alpha,T}(\tau'))]\,d\tau'
    \end{equation*}
    for all $\tau_{T} \geq 0$. 
\end{proposition}

\begin{proof}
As before, the fixed point operator can be written as \begin{equation}\label{compact_expression_fixed_point}
\begin{aligned}
\mathbf{\Phi}_{\alpha,T}(\mathbf{v}_0,\mathbf{v})(\tau)
&= \mathbf{S}_{\alpha}(\tau)\widetilde{\mathbf{P}}_{\alpha}\mathbf{v}_0
  + \int_{0}^{\tau}\mathbf{S}_{\alpha}(\tau-\tau')
    \widetilde{\mathbf{P}}_{\alpha}(\mathbf{H}(\mathbf{v})+ \mathbf{N}(\mathbf{v}))\,d\tau' 
- \mathbf{P}_{0,\alpha}\int_{\tau}^{\infty}(\mathbf{H}(\mathbf{v})+ \mathbf{N}(\mathbf{v}))\,d\tau' \\
&\quad - \mathbf{L}_{\alpha}\mathbf{P}_{0,\alpha}\int_{\tau}^{\infty}(\tau-\tau')
    (\mathbf{H}(\mathbf{v})+ \mathbf{N}(\mathbf{v}))\,d\tau' - \mathbf{P}_{1,\alpha}\int_{\tau}^{\infty}e^{\tau-\tau'}
    (\mathbf{H}(\mathbf{v})+ \mathbf{N}(\mathbf{v}))\,d\tau'.
\end{aligned}
\end{equation}
    Let $\mathbf{v}\in B_{\delta_1}^{\mathcal{X}^k}$ and $\mathbf{v}_0 \in B_{\frac{\delta_1}{M_1}}^{\mathcal{H}^k}$ for $0<\delta_1<\delta_0$ and $M_1>1$ to be determined. First consider that $0<T<\overline{T}_1$, estimating term by term (omitting subscripts), we have
   \begin{align*}
\|\mathbf{\Phi}_{\alpha,T}(\mathbf{v}_0,\mathbf{v})(\tau)\|_{\mathcal{H}^k}
\lesssim\;& e^{-\omega \tau}\|\mathbf{v}_0\|_{\mathcal{H}^k} \\
&+ \int_{0}^{\tau} e^{-\omega(\tau-\tau')}
   \big(\|\mathbf{H}(\mathbf{v})(\tau')\|_{\mathcal{H}^k}
   +\|\mathbf{N}(\mathbf{v})(\tau')\|_{\mathcal{H}^k}\big)\, d\tau' \\
&+ \int_{\tau}^{\infty}
   \big(\|\mathbf{H}(\mathbf{v})(\tau')\|_{\mathcal{H}^k}
   +\|\mathbf{N}(\mathbf{v})(\tau')\|_{\mathcal{H}^k}\big)\, d\tau' \\
&+ \int_{\tau}^{\infty}(\tau'-\tau)
   \big(\|\mathbf{H}(\mathbf{v})(\tau')\|_{\mathcal{H}^k}
   +\|\mathbf{N}(\mathbf{v})(\tau')\|_{\mathcal{H}^k}\big)\, d\tau' \\
&+ \int_{\tau}^{\infty} e^{\tau-\tau'}
   \big(\|\mathbf{H}(\mathbf{v})(\tau')\|_{\mathcal{H}^k}
   +\|\mathbf{N}(\mathbf{v})(\tau')\|_{\mathcal{H}^k}\big)\, d\tau'.
\end{align*}
Now we apply estimates from Lemma \ref{estimates_on_H}, and the fact that $\mathbf{v}\in B_{\delta_1}^{\mathcal{X}^k}$, i.e. $\|\mathbf{v}\|_{\mathcal{H}^k} \lesssim \delta_1 e^{-\omega \tau}$ for all $\tau\geq 0$
\begin{align*}
    \|\mathbf{\Phi}_{\alpha,T}(\mathbf{v}_0,\mathbf{v})(\tau)\|_{\mathcal{H}^k} \lesssim e^{-\omega \tau}\|\mathbf{v}_0\|_{\mathcal{H}^k} + \int_{0}^{\tau}e^{-\omega(\tau-\tau')}\left(\left(\frac{T}{r_0}\right)^2  \delta_1 e^{-2\omega \tau'}+\delta_1^2 e^{-2\omega \tau'}\right)\,d\tau' \\
    + \int_{\tau}^{\infty}\left(\frac{T}{r_0}\right)^2  \delta_1 e^{-2\omega \tau'} + \delta_1^2 e^{-2\omega \tau'}\,d\tau' + \int_{\tau}^{\infty}(\tau'-\tau)\left(\left(\frac{T}{r_0}\right)^2   \delta_1 e^{-2\omega \tau'} + \delta_1^2 e^{-2\omega \tau'}\right)\,d\tau' \\
    + \int_{\tau}^{\infty}e^{\tau-\tau'}\left(\left(\frac{T}{r_0}\right)^2   \delta_1 e^{-2\omega \tau'} + \delta_1^2 e^{-2\omega \tau'}\right)\,d\tau'.
\end{align*}
After computing the integrals, and bounding terms such as $e^{-2\omega \tau}\leq e^{-\omega \tau}$, we obtain 
\begin{align*}
     \|\mathbf{\Phi}_{\alpha,T}(\mathbf{v}_0,\mathbf{v})(\tau)\|_{\mathcal{H}^k} \lesssim e^{-\omega \tau}\frac{\delta_1}{M_1} + e^{-\omega \tau}\left(4\left(\frac{T}{r_0}\right)^2 \delta_1 + 4\delta_1^2\right) \leq \delta_1 e^{-\omega \tau}
\end{align*}
where the last inequality holds for $\delta_1$ small enough, $M_1$ large enough and possibly taking $0<T<\overline{T}_1$. Thus for $\mathbf{v}_0 \in B_{\frac{\delta_1}{M_1}}^{\mathcal{H}^k}$, we have for all $
|\alpha-\alpha_0|\leq \delta_0$, $\mathbf{\Phi}_{\alpha,T}(\mathbf{v}_0,\cdot):B_{\delta_1}^{\mathcal{X}^k} \to B_{\delta_1}^{\mathcal{X}^k}$. We now show that $\mathbf{\Phi}_{\alpha,T}(\mathbf{v}_0,\cdot)$ is a contraction. Firstly,
\begin{align*}
    \|\mathbf{\Phi}_{\alpha,T}(\mathbf{v}_0,\mathbf{v}) - \mathbf{\Phi}_{\alpha,T}(\mathbf{v}_0,\mathbf{v}')\|_{\mathcal{H}^k} \lesssim \int_{0}^{\tau}e^{-\omega(\tau-\tau')}(\|\mathbf{H}(\mathbf{v}-\mathbf{v}')\|_{\mathcal{H}^k} + \|\mathbf{N}(\mathbf{v})-\mathbf{N}(\mathbf{v}')\|_{\mathcal{H}^k})\,d\tau' \\
    + \int_{\tau}^{\infty}\|\mathbf{H}(\mathbf{v}-\mathbf{v}')\|_{\mathcal{H}^k} + \|\mathbf{N}(\mathbf{v})-\mathbf{N}(\mathbf{v}')\|_{\mathcal{H}^k}\,d\tau' + \int_{\tau}^{\infty}(\tau'-\tau)(\|\mathbf{H}(\mathbf{v}-\mathbf{v}')\|_{\mathcal{H}^k} + \|\mathbf{N}(\mathbf{v})-\mathbf{N}(\mathbf{v}')\|_{\mathcal{H}^k})\,d\tau' \\
    +\int_{\tau}^{\infty}e^{\tau-\tau'}(\|\mathbf{H}(\mathbf{v}-\mathbf{v}')\|_{\mathcal{H}^k} + \|\mathbf{N}(\mathbf{v})-\mathbf{N}(\mathbf{v}')\|_{\mathcal{H}^k})\,d\tau'.
\end{align*}
Now using Lemmas \ref{lem:estimates_v2}-\ref{estimates_on_H}, and $\|\mathbf{v}-\mathbf{v}'\|_{\mathcal{H}^k}\leq e^{-\omega \tau}\|\mathbf{v}-\mathbf{v}'\|_{\mathcal{X}^k}$ for all $\tau\geq 0$, we have 
\begin{align*}
\|\mathbf{\Phi}_{\alpha,T}(\mathbf{v}_0,\mathbf{v}) 
    - \mathbf{\Phi}_{\alpha,T}(\mathbf{v}_0,\mathbf{v}')\|_{\mathcal{H}^k}
&\lesssim
\|\mathbf{v}-\mathbf{v}'\|_{\mathcal{X}^k}
\Biggl\{
    \int_{0}^{\tau} e^{-\omega(\tau-\tau')}
    \left[\left(\frac{T}{r_0}\right)^2
          e^{-2\omega \tau'}
          + \delta_1 e^{-2\omega \tau'}\right]\, d\tau' \\
&\qquad\quad
    + \int_{\tau}^{\infty}
      \left[\left(\frac{T}{r_0}\right)^2
          e^{-2\omega \tau'}
          + \delta_1 e^{-2\omega \tau'}\right] d\tau' \\
&\qquad\quad
    + \int_{\tau}^{\infty}(\tau'-\tau)
     \left[\left(\frac{T}{r_0}\right)^2
          e^{-2\omega \tau'}
          + \delta_1 e^{-2\omega \tau'}\right] d\tau' \\
&\qquad\quad
    + \int_{\tau}^{\infty} e^{\tau-\tau'}
      \left[\left(\frac{T}{r_0}\right)^2
          e^{-2\omega \tau'}
          + \delta_1 e^{-2\omega \tau'}\right] d\tau'
\Biggr\} \\
    &\lesssim \|\mathbf{v}-\mathbf{v}'\|_{\mathcal{X}^k}
\Bigl\{4\left(\frac{T}{r_0}\right)^2 + \delta_1)e^{-\omega \tau}\Bigr\}.
\end{align*}
After taking $\delta_1>0$ small enough, and $T>0$ possibly smaller to obtain $\overline{T}_2$, we can ensure 
\begin{equation*}
    \|\mathbf{\Phi}_{\alpha,T}(\mathbf{v}_0,\mathbf{v}) 
    - \mathbf{\Phi}_{\alpha,T}(\mathbf{v}_0,\mathbf{v}')\|_{\mathcal{H}^k} \leq \frac{1}{2}e^{-\omega \tau}\|\mathbf{v}-\mathbf{v}'\|_{\mathcal{X}^k}
\end{equation*}
which yields the claim. Thus, by Banach's fixed point theorem we obtain a unique fixed point $\mathbf{v}_{\alpha,T} \in B_{\delta_1}^{\mathcal{X}^k}$. We remark that replacing $\delta_1>0$ by $\rho>0$ where $0<\rho\leq \delta_1$ only improves the above estimates and so the conclusion holds for the \textit{same} $T$.
\end{proof}
\subsection{Finite-dimensional reduction}
We recall our ``background'' family of blow-up solutions given by 
\begin{equation}\label{eq:background_sols_sec_4}
    \mathcal{U}_{\alpha_0,\kappa_0,T_0}(\tau_{T_0},y_{T_0}) = U_{\alpha_0,\kappa_0}(\tau_{T_0},y_{T_0}) - \frac{n-1}{2}\log\left(1+\frac{T_0}{r_0}y_{T_0}e^{-\tau_{T_0}}\right) + \mathbf{q}_{\alpha_0,T_0}^{(1)}(\tau_{T_0},y_{T_0})
\end{equation}
for all $T_0 \leq \overline{T}_1$ (and other parameters) as well as the special case for $n=3$ when 
\begin{equation*}
    \mathcal{U}^{(3)}_{\alpha_0,\kappa_0,T_0}(\tau_{T_0},y_{T_0}) = U_{\alpha_0,\kappa_0}(\tau_{T_0},y_{T_0}) - \log\left(1+\frac{T_0}{r_0}y_{T_0}e^{-\tau_{T_0}}\right).
\end{equation*}In the previous subsection, for possibly smaller $T\leq \overline{T}_2\leq \overline{T}_1$, we further proved the existence of family of solutions to \eqref{similarity_variables_stability_sec} of the form 
\begin{equation}\label{eq:form_of_solution}
    \mathbf{u}^{(1)}(\tau_{T},y_{T}) = U_{\alpha,\kappa}(\tau_{T},y_{T})+\mathbf{q}_{\alpha,T}^{(1)}(\tau_{T},y_{T}) + \mathbf{v}^{(1)}_{\alpha,T}(\tau_{T},y_{T}).
\end{equation}
We now want to find a solution of the form \eqref{eq:form_of_solution} for appropriate $(\alpha,\kappa,T)$, but which arises from a small initial perturbation to the family \eqref{eq:background_sols_sec_4} at parameters $(\alpha_0,\kappa_0,T_0)$. 

Before this, we make some important remarks. Observe that our construction relied on a sufficiently small blow-up time. Thus, to proceed with modulation of $T$, once a sufficiently small $\overline{T}$ is found with $\overline{T}\leq \overline{T}_2 \leq \overline{T}_1$, we will modulate inside the interval $(0,\overline{T})$. This will require a small interval around any $T_0 \in (0,\overline{T})$ which also lies in such an admissible range. I.e., we will require a radius $\eta_{\overline{T},T_0,\alpha_0}$ 
\begin{equation}\label{eq:condition_on_delta_1}
    \eta_{\overline T,T_0,\alpha_0} \leq \frac{1}{T_0}\min \left\{\frac{T_0}{2},\frac{\overline{T}-T_0}{2},T_0(R_{\alpha_0}-1),T_0\delta_0\right\}.
\end{equation}
This ensures three crucial properties. The first two quantities we take the minimum of ensure that 
\begin{equation}\label{eq:shrinking_delta_1}
    T\in [T_0-\eta_{\overline T,T_0,\alpha_0} T_0,T_0+\eta_{\overline T,T_0,\alpha_0} T_0]\subset \left(0,\overline{T}\right)
\end{equation}
and the final condition ensures that $T\leq T_0+\eta_{\overline T,T_0,\alpha_0} T_0 < T_0R_{\alpha_0}$, where this extra room was provided by Section \ref{section_two}, i.e.
\begin{equation*}
    \left(-\frac{T}{T_0},\frac{T}{T_0}\right) \subset (-R_{\alpha_0},R_{\alpha_0}) = \mathcal{I}_{\alpha_0}
\end{equation*}
holds for $T<T_0 R_{\alpha_0}$.
Finally, the last condition merely ensures that $\eta_{\overline T,T_0,\alpha_0}<\delta_0$ so that the $|\alpha-\alpha_0|<\delta_0$ are admissible per Section \ref{section_two} (recall $\delta_0 \coloneqq \min\{1,\frac{\alpha_0}{4}\}$).
\begin{definition}\label{def:initial_data_decomposition}
    Consider some $\overline{T}\leq \overline{T}_2$ to be determined later. Fix $0<T_0<\overline{T}$. Then for $\alpha\in [\alpha_0-\delta_0,\alpha_0+\delta_0]$, $\kappa,\kappa_0 \in \mathbb{R}$, and $T\in [T_0-\eta_{\overline T,T_0,\alpha_0} T_0,T_0+\eta_{\overline T,T_0,\alpha_0} T_0]$ we define the initial-data operator 
    \begin{equation*}
        \mathbf{U}_{\alpha,\kappa,T}:\mathcal{H}^k(\mathbb{R})\to \mathcal{H}^k(-1,1), \qquad \mathbf{f}\mapsto \mathbf{f}_{T} + \mathbf{f}_{0,T} - \mathbf{f}_{\alpha,\kappa,T}
    \end{equation*}
    where \begin{align*}
        \mathbf{f}(y) = \begin{pmatrix}
            f_1(y) \\
            f_2(y)
        \end{pmatrix}
    \end{align*} is the initial data perturbation and 
    \begin{align*}
        \mathbf{f}_{T}(y) = \begin{pmatrix}
            f_1(Ty) \\
            Tf_2(Ty)
        \end{pmatrix}, \quad \mathbf{f}_{0,T}(y) = \begin{pmatrix}
            \widetilde{U}_{\alpha_0,\kappa_0}(\frac{T}{T_0}y) -\frac{(n-1)}{2}\log(1+\frac{T}{r_0}y)+ \mathbf{q}_{\alpha_0,T_0}^{(1)}(0,\frac{T}{T_0}y) \\
            \frac{T}{T_0}\alpha_0 + (\frac{T}{T_0})^2 y \, \widetilde{U}_{\alpha_0,\kappa_0}'(\frac{T}{T_0}y) + \frac{T}{T_0}\mathbf{q}_{\alpha_0,T_0}^{(2)}(0,\frac{T}{T_0}y)
        \end{pmatrix} 
    \end{align*}
    \begin{align*}
        \mathbf{f}_{\alpha,\kappa,T}(y) = \begin{pmatrix}
            \widetilde{U}_{\alpha,\kappa}(y) - \frac{(n-1)}{2}\log(1+\frac{T}{r_0}y)+ \mathbf{q}_{\alpha,T}^{(1)}(0,y) \\
            \alpha + y\widetilde{U}'_{\alpha,\kappa}(y) + \mathbf{q}_{\alpha,T}^{(2)}(0,y)
        \end{pmatrix} 
    \end{align*}
\end{definition}
\begin{remark}
    Note that the contribution from $\log\left(\frac{r}{r_0}\right)$ in $\mathbf{f}_{0,T}$ and $\mathbf{f}_{\alpha,\kappa,T}$ cancels when they are subtracted. Moreover, there is no contribution in the second entry, i.e. the similarity time derivative for each of these functions.  
\end{remark}
\begin{lemma}(Taylor expansion of initial data).\label{initial_data_decomposition} 
Let $\overline{T}\leq \overline{T}_2$ and fix $0<T_0<\overline{T}$, $\kappa_0 \in \mathbb{R}$ and $T\in [T_0-\eta_{\overline T,T_0,\alpha_0} T_0,T_0+\eta_{\overline T,T_0,\alpha_0} T_0]$. Denote 
\begin{equation*}
    \lambda = \frac{T}{T_0}.
\end{equation*}
For any
\[
\alpha \in [\alpha_0-\eta_{\overline T,T_0,\alpha_0},\alpha_0+\eta_{\overline T,T_0,\alpha_0}],\qquad
\kappa \in [\kappa_0-\eta_{\overline T,T_0,\alpha_0},\kappa_0+\eta_{\overline T,T_0,\alpha_0}],\qquad
\lambda\in [1-\eta_{\overline T,T_0,\alpha_0},1+\eta_{\overline T,T_0,\alpha_0}],
\] 
and for any \(\mathbf f\in \mathcal H^k(\mathbb{R})\), we have
\[
\mathbf U_{\alpha,\kappa,T}(\mathbf f)
=
\mathbf f_T
+
\Big[(\kappa_0-\kappa)-\alpha(\lambda-1)\Big]\mathbf f_{0,\alpha}
+
(\lambda-1)\mathbf f_{1,\alpha}
+
(\alpha_0-\alpha)\mathbf g_{0,\alpha}
+
\mathbf h_{\alpha,T}
+
\mathbf r\!\left(\alpha,\lambda\right),
\]
where
\[
\mathbf h_{\alpha,T}(y)
:=
\begin{pmatrix}
\mathbf q_{\alpha_0,T_0}^{(1)}\!\left(0,\frac{T}{T_0}y\right)-\mathbf q_{\alpha,T}^{(1)}(0,y)\\[1mm]
\frac{T}{T_0}\mathbf q_{\alpha_0,T_0}^{(2)}\!\left(0,\frac{T}{T_0}y\right)-\mathbf q_{\alpha,T}^{(2)}(0,y)
\end{pmatrix},
\]
and the remainder satisfies
\[
\left\|\mathbf r\!\left(\alpha,\lambda\right)\right\|_{\mathcal H^k(-1,1)}
\lesssim_{\alpha_0,k}
|\alpha-\alpha_0|^2+\left|\lambda-1\right|^2.
\]
Moreover,
\begin{equation}\label{eq:initial_data_h_expression}
\|\mathbf h_{\alpha,T}\|_{\mathcal H^k(-1,1)}
\lesssim_{n,\theta,\omega_0,\gamma,\alpha_0,k}\frac{\overline{T}}{r_0}\left(|\alpha-\alpha_0|+\left|\lambda-1\right| \right)
\end{equation}
and finally, given $\mathbf{f}\in \mathcal{H}^k(\mathbb{R})$, for all $T<\frac{r_0}{2R_{\alpha_0}}$ we have the uniform bound
\begin{equation}\label{eq:uniform_bound_scaling}
    \|\mathbf{f}_{T}(\cdot)\|_{\mathcal{H}^k(-1,1)} \leq C_{\text{sc}}\|\mathbf{f}\|_{\mathcal{H}^k(\mathbb{R})}
\end{equation}
where $C_{sc}$ depends only on $\theta,r_0,\alpha_0$ and $k$. 
\end{lemma}
\begin{proof}
    For fixed $y\in [-1,1]$, Taylor's theorem applied to 
    \begin{equation*}
        \mathbf{G}:[\alpha_0-\delta_0,\alpha_0+\delta_0]\times [\kappa_0-\delta_0,\kappa_0+\delta_0] \times [1-\delta_0 ,1+\delta_0]\mapsto \mathbb{R}^2, \qquad (\alpha,\kappa,\lambda)\mapsto \mathbf{f}_{0,T}(y) - \mathbf{f}_{\alpha,\kappa,T}(y) - \mathbf{h}_{\alpha,T}(y)
    \end{equation*}
    around the point $(\alpha_0,\kappa_0,1)$ first yields  
    \begin{equation*}
       \mathbf{f}_{0,T}(y)-\mathbf{f}_{\alpha,\kappa,T}(y)-\mathbf{h}_{\alpha,T}(y) = [(\kappa_0-\kappa)+\alpha_0(1-\lambda)]\mathbf{f}_{0,\alpha_0} + (\lambda-1)\mathbf{f}_{1,\alpha_0} + (\alpha_0-\alpha)\mathbf{g}_{0,\alpha_0} + \mathbf{r}_0(\alpha,\lambda)
    \end{equation*}
    where 
    \begin{align*}
        \mathbf{r}_0(\alpha,\lambda) = (\lambda-1)^2\int_{0}^{1}(1-s)\partial_{\lambda \lambda}\mathbf{A}(1+s(\lambda-1),y)\,ds - (\alpha-\alpha_0)^2\int_{0}^{1}(1-s)\partial_{\alpha \alpha}\mathbf{B}(\alpha_0+s(\alpha-\alpha_0),y)\,ds
    \end{align*}
    and 
    \begin{align*}
        \mathbf{A}(\lambda,y) = \begin{pmatrix}
            -\alpha_0\log(\sqrt{1+\alpha_0}+\lambda y)+\kappa_0 \\
            \alpha_0 \lambda - \frac{\alpha_0 \lambda^2 y}{\sqrt{1+\alpha_0}+\lambda y}
        \end{pmatrix},\qquad \mathbf{B}(\alpha,y) = \begin{pmatrix}
            -\alpha \log(\sqrt{1+\alpha}+y) \\
            \alpha - \frac{\alpha y}{\sqrt{1+\alpha}+y}
        \end{pmatrix}. 
    \end{align*}
    Therefore, we can write 
    \begin{equation*}
         \mathbf{f}_{0,T}(y) - \mathbf{f}_{\alpha,\kappa,T}(y) - \mathbf{h}_{\alpha,T}(y) = \left[(\kappa_0-\kappa)-\alpha\left(\lambda-1\right)\right]\mathbf{f}_{0,\alpha} + \left(\lambda-1 \right)\mathbf{f}_{1,\alpha} + (\alpha_0-\alpha)\mathbf{g}_{0,\alpha} + \mathbf{r}\left(\alpha,\lambda\right)
    \end{equation*}
    where we added and subtracted such that the new remainder absorbs the left over terms:
    \begin{align*}
        \mathbf{r}(\alpha,\lambda) &= (\kappa_0-\kappa)(\mathbf{f}_{0,\alpha_0}-\mathbf{f}_{0,\alpha}) + (1-\lambda)(\alpha_0\mathbf{f}_{0,\alpha_0}-\alpha \mathbf{f}_{0,\alpha}) + (1-\lambda)(\mathbf{f}_{1,\alpha}-\mathbf{f}_{1,\alpha_0}) \\
        &+ (\alpha_0-\alpha)(\mathbf{g}_{0,\alpha_0}-\mathbf{g}_{0,\alpha}) + \mathbf{r}_0(\alpha,\lambda) \\
        &= (1-\lambda)(\alpha_0\mathbf{f}_{0,\alpha_0}-\alpha \mathbf{f}_{0,\alpha}) + (1-\lambda)(\mathbf{f}_{1,\alpha}-\mathbf{f}_{1,\alpha_0}) + (\alpha_0-\alpha)(\mathbf{g}_{0,\alpha_0}-\mathbf{g}_{0,\alpha}) + \mathbf{r}_0(\alpha,\lambda).
    \end{align*}
    Therefore we have the bound
    \begin{equation*}
        \left \|\mathbf{r}\left(\alpha,\lambda\right)\right\|_{\mathcal{H}^k} \leq C_r |\alpha-\alpha_0|^2 + \left|\lambda-1\right|^2
    \end{equation*}
    where $C_r$ only depends on $\alpha_0$ and the regularity $k$.  
    Thus, it remains to prove the bound \eqref{eq:initial_data_h_expression} for $\mathbf{h}_{\alpha,T}$.  By adding and subtracting, we begin by separating the scaled terms (with fixed parameters $\alpha_0,T_0$), and the terms where the parameter varies by writing $\mathbf{h}_{\alpha,T}=\mathbf{h}_{T}^s + \mathbf{h}_{\alpha,T}^p$ where
\begin{align}\label{eq:decomposing_h}
        \mathbf{h}_{T}^s(y) =\begin{pmatrix}
            \mathbf{q}_{\alpha_0,T_0}^{(1)}(0,\lambda y) - \mathbf{q}_{\alpha_0,T_0}^{(1)}(0,y) \\
            \lambda\mathbf{q}_{\alpha_0,T_0}^{(2)}(0,\lambda y) - \mathbf{q}_{\alpha_0,T_0}^{(2)}(0,y)
    \end{pmatrix} 
    \end{align}
    and 
    \begin{align}
        \mathbf{h}_{\alpha,T}^{p}(y) = \begin{pmatrix}\mathbf{q}_{\alpha_0,T_0}^{(1)}(0,y)-\mathbf{q}_{\alpha,T}^{(1)}(0,y) \\
\mathbf{q}_{\alpha_0,T_0}^{(2)}(0,y) - \mathbf{q}_{\alpha,T}^{(2)}(0,y)
        \end{pmatrix}.
    \end{align}
First we estimate the scaling part $\mathbf h_{T}^s$. Since the constructed initial datum
$\mathbf q_{\alpha_0,T_0}(0,\cdot)$ lies in the symmetry space, there exist constants
$c_1,c_2,c_3\in\mathbb R$ (dependent on $\alpha_0,T_0,\kappa_0)$ such that
\[
\mathbf q_{\alpha_0,T_0}(0,\cdot)
=
c_1\mathbf f_{0,\alpha_0}
+c_2\mathbf f_{1,\alpha_0}
+c_3\mathbf g_{0,\alpha_0},
\]
and by equivalence of norms on the finite-dimensional space $\rg(\mathbf{P}_{\alpha_0}^{(\alpha_0)})$,
\[
|c_1|+|c_2|+|c_3|
\lesssim_{\alpha_0,k}
\|\mathbf q_{\alpha_0,T_0}(0)\|_{\mathcal H^k(-1,1)}
\lesssim_{\omega_0,\gamma,\alpha_0,n,k} \left(\frac{T_0}{r_0}\right)^2 \leq \left(\frac{\overline{T}}{r_0}\right)^2.
\]
Now let $\Phi=(\Phi^{(1)},\Phi^{(2)})^\top$ be one of the basis functions
$\mathbf f_{0,\alpha_0},\mathbf f_{1,\alpha_0},\mathbf g_{0,\alpha_0}$.
Since $T<T_0R_{\alpha_0}$, we have $(-\lambda,\lambda)\subset \mathcal{I}_{\alpha_0}$, and hence $\Phi(\lambda \cdot)$ is smooth on $y\in [-1,1]$. For $0\le j\le k+1$,
\[
\partial_y^j\big(\Phi^{(1)}(\lambda y)-\Phi^{(1)}(y)\big)
=
\lambda^j \partial_y^j\Phi^{(1)}(\lambda y)-\partial_y^j\Phi^{(1)}(y)
=
\int_1^\lambda \partial_\mu\big(\mu^j \partial_y^j\Phi^{(1)}(\mu y)\big)\,d\mu,
\]
and therefore
\[
\|\Phi^{(1)}(\lambda\cdot)-\Phi^{(1)}(\cdot)\|_{H^{k+1}(-1,1)}
\lesssim_{\alpha_0,k}
|\lambda-1|.
\]
Similarly, for $0\le j\le k$,
\[
\partial_y^j\big(\lambda\Phi^{(2)}(\lambda y)-\Phi^{(2)}(y)\big)
=
\lambda^{j+1}\partial_y^j\Phi^{(2)}(\lambda y)-\partial_y^j\Phi^{(2)}(y)
=
\int_1^\lambda \partial_\mu\big(\mu^{j+1}\partial_y^j\Phi^{(2)}(\mu y)\big)\,d\mu,
\]
which yields
\[
\|\lambda\Phi^{(2)}(\lambda\cdot)-\Phi^{(2)}(\cdot)\|_{H^{k}(-1,1)}
\lesssim_{\alpha_0,k}
|\lambda-1|.
\]
Applying these bounds to each basis function and summing with the coefficients $c_i$, we obtain
\[
\|\mathbf h_{T}^s\|_{\mathcal H^k(-1,1)}
\lesssim_{\alpha_0,k}
\left(\frac{\overline{T}}{r_0}\right)^2 |\lambda-1|.
\]
Next, for the Lipschitz in parameters part, $\mathbf{h}_{\alpha,T}^p$, we apply the estimate \eqref{eq:lipschitz_in_params} from Proposition \ref{prop:existence_of_correction} with $\mathbf{q}_0=\mathbf{0}$ at $\tau=0$ and for $\alpha\in[\alpha_0-\delta_0,\alpha_0+\delta_0]$, $0<T<\overline{T}<\overline{T}_1$:
\begin{equation*}
    \|\mathbf{q}_{\alpha,T}(0,\cdot)-\mathbf{q}_{\alpha_0,T_0}(0,\cdot)\|_{\mathcal{H}_{\alpha_0}^k} \lesssim \frac{T+T_0}{r_0}\left(\frac{|T-T_0|}{r_0}+|\alpha-\alpha_0|\right) \lesssim \frac{\overline{T}}{r_0}\left(|1-\lambda|+|\alpha-\alpha_0|\right).
\end{equation*}
Now, to prove the scaling bound \eqref{eq:uniform_bound_scaling} take $\mathbf{f}\in H^{k+1}(\mathbb{R})\times H^k(\mathbb{R})$ for some $k\geq \lceil c_{\alpha_0} \rceil+1$ and $0<T<\frac{r_0}{R_{\alpha_0}}$. Firstly, using Sobolev embedding, 
\begin{equation*}
    \|f_1(T\cdot)\|_{L^2(-1,1)}\leq \sqrt{2}\|f_1(T\cdot)\|_{L^\infty(-1,1)}\leq \sqrt{2}\|f_1(\cdot)\|_{L^\infty(\mathbb{R})}\lesssim \|f_1(\cdot)\|_{H^1(\mathbb{R})}.
\end{equation*}
Next, for $1\leq j \leq k+1$, 
\begin{equation*}
    \|\partial_y^j f_1(T\cdot)\|_{L^2(-1,1)} = T^{j-\frac{1}{2}}\|\partial_y^j f_1(\cdot)\|_{L^2(-T,T)}\leq \left(\frac{r_0}{2R_{\alpha_0}}\right)^{j-\frac{1}{2}}\|\partial_y^j f_1(\cdot)\|_{L^2(\mathbb{R})}.
\end{equation*}
For the second component, we have 
\begin{equation*}
    \|Tf_2(T\cdot)\|_{L^2(-1,1)} = T^{\frac{1}{2}}\|f_2(\cdot)\|_{L^2(-T,T)}\leq \left( \frac{ r_0}{2R_{\alpha_0}}\right)^{\frac{1}{2}}\|f_2(\cdot)\|_{L^2(\mathbb{R})}.
\end{equation*}
Finally, for $1\leq j \leq k$,
\begin{equation*}
    \|\partial_y^j Tf_2(T\cdot)\|_{L^2(-1,1)} = T^{j+\frac{1}{2}}\|\partial_y^j f_2(\cdot)\|_{L^2(-T,T)} \leq \left(\frac{ r_0}{2R_{\alpha_0}}\right)^{j+\frac{1}{2}}\|\partial_y^j f_2(\cdot)\|_{L^2(\mathbb{R})}.
\end{equation*}
Therefore,
\begin{align*}
    \|\mathbf{f}_{T}(\cdot)\|_{\mathcal{H}^k(-1,1)} &\lesssim 2\|f_1(\cdot)\|_{H^1(\mathbb{R})} + \sum_{j=1}^{k+1}\left( \frac{r_0}{2R_{\alpha_0}}\right)^{j-\frac{1}{2}}\|\partial_y^j f_1(\cdot)\|_{L^2(\mathbb{R})}+ \left(\frac{ r_0}{2R_{\alpha_0}}\right)^{\frac{1}{2}}\|f_2(\cdot)\|_{L^2(\mathbb{R})}\\
    &+ \sum_{j=1}^{k}\left(\frac{r_0}{2R_{\alpha_0}} \right)^{j+\frac{1}{2}}\|\partial_y^j f_2(\cdot)\|_{L^2(\mathbb{R})} \leq C_{sc}(r_0,\alpha_0,k)\|\mathbf{f}(\cdot)\|_{\mathcal{H}^k(\mathbb{R})}.
\end{align*}
\end{proof}
\begin{proposition}\label{prop:finite_dim_reduction}
Let \(0<\delta_1<\delta_0\), \(M_1>1\), and
\(\overline T_2>0\) be the constants from
Proposition~\ref{existence_of_v}. 
Then there exist \(M_2>1\) and \(0<\overline T\le \overline T_2\) such that
the following holds. For every \(T_0\in(0,\overline T)\), there exists
\(\delta_2=\delta_2(T_0)>0\) such that, for $\mathbf{f}\in\mathcal{H}^k(\mathbb{R})$ satisfying 
$\|\mathbf f\|_{\mathcal H^k(\mathbb R)}
\le
\frac{\delta_2}{M_2^2}$,
there exist parameters
\[
\alpha^\star\in
\left[
\alpha_0-\frac{\delta_2}{M_2},
\alpha_0+\frac{\delta_2}{M_2}
\right],
\qquad
\kappa^\star\in
\left[
\kappa_0-\frac{\delta_2}{M_2},
\kappa_0+\frac{\delta_2}{M_2}
\right],
\]
and
\[
T^\star\in
\left[
T_0\left(1-\frac{\delta_2}{M_2}\right),
T_0\left(1+\frac{\delta_2}{M_2}\right)
\right]
\subset (0,\overline T),
\]
with
\[
\left(-\frac{T^\star}{T_0},\frac{T^{\star}}{T_0}\right)\subset \mathcal{I}_{\alpha_0},
\]
and a unique 
\(\mathbf v_{\alpha^\star,\kappa^\star,T^\star} \in C([0,\infty);\mathcal{H}^k(-1,1))\) with $\|\mathbf{v}_{\alpha^\star,\kappa^\star,T^\star}\|_{\mathcal{X}^k}\leq \delta_2$ satisfying 
\begin{equation}\label{eq:final_prop_mild_sol}
    \mathbf v_{\alpha^\star,\kappa^\star,T^\star}(\tau_{T^\star}) = \mathbf{S}_{\alpha^\star}(\tau_{T^\star})\mathbf{U}_{\alpha^\star,\kappa^\star,T^\star}(\mathbf{f}) + \int_{0}^{\tau_{T^\star}}\mathbf{S}_{\alpha^\star}(\tau_{T^\star}-\tau')\left[\mathbf{H}_{\alpha^{\star},T^{\star}}(\mathbf{v}_{\alpha^\star,\kappa^\star,T^\star}(\tau')) + \mathbf{N}(\mathbf{v}_{\alpha^\star,\kappa^\star,T^\star}(\tau'))\right]d\tau'
\end{equation}
\end{proposition}
\begin{proof}
We use Lemma~\ref{initial_data_decomposition} in its uniform form. Namely,
for every \(0<\Theta\le \overline T_2\), every \(0<T_0<\Theta\), and every
admissible \((\alpha,\kappa,T)\), with $\lambda:=\frac{T}{T_0}$, we have
\begin{equation}\label{eq:uniform_initial_decomposition_for_prop45}
\begin{aligned}
\mathbf U_{\alpha,\kappa,T}(\mathbf f)
&=
\mathbf f_T
+
\big[(\kappa_0-\kappa)-\alpha(\lambda-1)\big]\mathbf f_{0,\alpha}
+
(\lambda-1)\mathbf f_{1,\alpha}
\\
&+
(\alpha_0-\alpha)\mathbf g_{0,\alpha}
+
\mathbf h_{\alpha,T}
+
\mathbf r(\alpha,\lambda),
\end{aligned}
\end{equation}
where the constants \(C_{\mathrm{sc}},C_h,C_r>0\) are independent of
\(\Theta,T_0,T,\alpha,\kappa\), and
\begin{equation}\label{eq:scaling_bound_for_prop45}
\|\mathbf f_T\|_{\mathcal H^k(-1,1)}
\le
C_{\mathrm{sc}}\|\mathbf f\|_{\mathcal H^k(\mathbb R)},
\end{equation}
\begin{equation}\label{eq:h_bound_for_prop45}
\|\mathbf h_{\alpha,T}\|_{\mathcal H^k(-1,1)}
\le
C_h\frac{\Theta}{r_0}
\bigl(
|\alpha-\alpha_0|+|\lambda-1|
\bigr),
\end{equation}
and
\begin{equation}\label{eq:r_bound_for_prop45}
\|\mathbf r(\alpha,\lambda)\|_{\mathcal H^k(-1,1)}
\le
C_r
\bigl(
|\alpha-\alpha_0|^2+|\lambda-1|^2
\bigr).
\end{equation}
Since the symmetry modes depend smoothly on \(\alpha\), define
\[
C_{\mathrm{sym}}
:=
\sup_{|\alpha-\alpha_0|\le \delta_0}
\left(
\|\mathbf f_{0,\alpha}\|_{\mathcal H^k}
+
\|\mathbf f_{1,\alpha}\|_{\mathcal H^k}
+
\|\mathbf g_{0,\alpha}\|_{\mathcal H^k}
\right)
<\infty.
\]
Next, let
\[
\{\mathbf g_\alpha^1,\mathbf g_\alpha^2,\mathbf g_\alpha^3\}
\]
denote the \(\mathcal H^k\)-dual basis to the ordered basis
\[
\{\mathbf g_{0,\alpha},\mathbf f_{0,\alpha},\mathbf f_{1,\alpha}\}
\]
of \(\operatorname{rg}(\mathbf P_\alpha)\), i.e.
\[
\langle \mathbf g_{0,\alpha},\mathbf g_\alpha^1\rangle_{\mathcal H^k}=1,
\qquad
\langle \mathbf f_{0,\alpha},\mathbf g_\alpha^2\rangle_{\mathcal H^k}=1,
\qquad
\langle \mathbf f_{1,\alpha},\mathbf g_\alpha^3\rangle_{\mathcal H^k}=1,
\]
and all other pairings between the ordered basis and the dual basis vanish.

We now quantitatively record the finite-dimensional bounds that will be used below. By
continuity in \(\alpha\) and compactness of
\(\{\alpha:|\alpha-\alpha_0|\le\delta_0\}\), the following constants are finite:
\[
B_P
:=
\sup_{|\alpha-\alpha_0|\le\delta_0}
\max_{1\le n\le3}
\|\mathbf g_\alpha^n\|_{\mathcal H^k}\,
\|\mathbf P_\alpha\|_{\mathcal{L}(\mathcal{H}^k)},
\]
\[
B_0
:=
\sup_{|\alpha-\alpha_0|\le\delta_0}
\max_{1\le n\le3}
\|\mathbf g_\alpha^n\|_{\mathcal H^k}\,
\|\mathbf P_{0,\alpha}\|_{\mathcal{L}(\mathcal{H}^k)},
\]
\[
B_L
:=
\sup_{|\alpha-\alpha_0|\le\delta_0}
\max_{1\le n\le3}
\|\mathbf g_\alpha^n\|_{\mathcal H^k}\,
\|\mathbf L_\alpha\mathbf P_{0,\alpha}\|_{\mathcal{L}(\mathcal{H}^k)},
\]
and
\[
B_1
:=
\sup_{|\alpha-\alpha_0|\le\delta_0}
\max_{1\le n\le3}
\|\mathbf g_\alpha^n\|_{\mathcal H^k}\,
\|\mathbf P_{1,\alpha}\|_{\mathcal{L}(\mathcal{H}^k)}.
\]
Therefore, for every \(\mathbf{w}\in\mathcal H^k(-1,1)\),
\begin{equation}\label{eq:finite_dimensional_bound_projected_data}
\left|
\left\langle
\mathbf P_\alpha\mathbf{w},\mathbf g_\alpha^n
\right\rangle_{\mathcal H^k}
\right|
\le
B_P\|\mathbf{w}\|_{\mathcal H^k}.
\end{equation}
Moreover, for every Bochner-integrable
\(\mathbf G:[0,\infty)\to\mathcal H^k(-1,1)\),
\begin{equation}\label{eq:finite_dimensional_bound_tail_terms}
\begin{aligned}
&\left|
\left\langle
\mathbf P_{0,\alpha}
\int_0^\infty \mathbf G(\tau)\,d\tau,
\mathbf g_\alpha^n
\right\rangle_{\mathcal H^k}
\right| +
\left|
\left\langle
\mathbf L_\alpha\mathbf P_{0,\alpha}
\int_0^\infty (-\tau)\mathbf G(\tau)\,d\tau,
\mathbf g_\alpha^n
\right\rangle_{\mathcal H^k}
\right|
\\
&+
\left|
\left\langle
\mathbf P_{1,\alpha}
\int_0^\infty e^{-\tau}\mathbf G(\tau)\,d\tau,
\mathbf g_\alpha^n
\right\rangle_{\mathcal H^k}
\right|
\\
&\le
B_0\int_0^\infty
\|\mathbf G(\tau)\|_{\mathcal H^k}\,d\tau
+
B_L\int_0^\infty
\tau\|\mathbf G(\tau)\|_{\mathcal H^k}\,d\tau
+
B_1\int_0^\infty
e^{-\tau}\|\mathbf G(\tau)\|_{\mathcal H^k}\,d\tau.
\end{aligned}
\end{equation}
Let \(C_H,C_N>0\) be constants such that, for
\(|\alpha-\alpha_0|\le\delta_0\), \(0<T\le\overline T_2\), and
\(\mathbf v\in B_{\delta_1}^{\mathcal X^k}\),
\begin{equation}\label{eq:H_bound_for_prop45}
\|\mathbf H_{\alpha,T}(\mathbf v)(\tau)\|_{\mathcal H^k}
\le
C_H\frac{T}{r_0}e^{-\omega\tau}\|\mathbf v(\tau)\|_{\mathcal H^k},
\end{equation}
and
\begin{equation}\label{eq:N_bound_for_prop45}
\|\mathbf N(\mathbf v)(\tau)\|_{\mathcal H^k}
\le
C_N\|\mathbf v(\tau)\|_{\mathcal H^k}^2.
\end{equation}
Define also the constant 
\[
J_{\mathrm{tail}}
:=
\frac{B_0}{2\omega}
+
\frac{B_L}{(2\omega)^2}
+
\frac{B_1}{1+2\omega},
\]
which is precisely the constant obtained by inserting \(e^{-2\omega\tau}\) as $\mathbf{G}$ into
the three integrals in \eqref{eq:finite_dimensional_bound_tail_terms}.
Next let 
\[
K_{\mathrm{init}}
:=
C_{\mathrm{sc}}
+
C_{\mathrm{sym}}(\alpha_0+4)
+
2C_h
+
2C_r,
\]
and
\[
K_f:=B_PC_{\mathrm{sc}},
\qquad
K_h:=2B_PC_h,
\qquad
K_r:=2B_PC_r,
\]
\[
K_H:=J_{\mathrm{tail}}C_H,
\qquad
K_N:=J_{\mathrm{tail}}C_N.
\]
We now choose the parameters. First choose \(M_2>1\) so large that
\begin{equation}\label{eq:M2_choice_prop45}
\frac{K_{\mathrm{init}}}{M_2}
\le
\frac1{M_1},
\qquad
\frac{K_f}{M_2}
\le
\frac1{12(\alpha_0+2)}.
\end{equation}
Next choose the final time ceiling \(0<\overline T\le\overline T_2\) so small
that
\begin{equation}\label{eq:Tbar_choice_prop45}
\frac{\overline{T}}{r_0}<1
\qquad \text{and} \qquad
K_h\frac{\overline{T}}{r_0}+K_HM_2\frac{\overline{T}}{r_0}
\le
\frac1{12(\alpha_0+2)}.
\end{equation}
Now fix \(T_0\in(0,\overline T)\). Define
\[
\eta_{\overline T,T_0,\alpha_0}
:=
\frac12
\min
\left\{
1,\,
\delta_0,\,
R_{\alpha_0}-1,\,
\frac{\overline T-T_0}{T_0}
\right\}.
\]
Choose \(\delta_2=\delta_2(T_0)>0\) so small that
\begin{equation}\label{eq:delta2_basic_choice_prop45}
\delta_2\le\delta_1,
\qquad
\delta_2\le1,
\qquad
\frac{\delta_2}{M_2}
\le
\eta_{\overline T,T_0,\alpha_0},
\end{equation}
and
\begin{equation}\label{eq:delta2_nonlinear_choice_prop45}
K_r\frac{\delta_2}{M_2}
+
K_NM_2\delta_2
\le
\frac1{12(\alpha_0+2)}.
\end{equation}
Define the modulation box
\[
\mathcal D
:=
[\alpha_0-\frac{\delta_2}{M_2},\alpha_0+\frac{\delta_2}{M_2}]
\times
[\kappa_0-\frac{\delta_2}{M_2},\kappa_0+\frac{\delta_2}{M_2}]
\times
[T_0(1-\frac{\delta_2}{M_2}),T_0(1+\frac{\delta_2}{M_2})].
\]
We first check that the whole box is admissible. Since
\(\frac{\delta_2}{M_2}\le \eta_{\overline T,T_0,\alpha_0}\), we have \(\frac{\delta_2}{M_2}\le\frac{1}{2}\). Hence
\(T>0\) on \(\mathcal D\). Moreover,
\[
T\le T_0(1+\frac{\delta_2}{M_2})
\le
T_0+\frac{\overline T-T_0}{2}
=
\frac{T_0+\overline T}{2}
<
\overline T.
\]
Thus \(T\in(0,\overline T)\subset(0,\overline T_2)\) on \(\mathcal D\).
Furthermore,
\[
\lambda:=\frac{T}{T_0}\in[1-\frac{\delta_2}{M_2},1+\frac{\delta_2}{M_2}],
\]
and since
\[
\frac{\delta_2}{M_2}\le \frac{R_{\alpha_0}-1}{2},
\]
we have
\[
1+\frac{\delta_2}{M_2}<R_{\alpha_0}.
\]
Consequently,
\[
(-\lambda,\lambda)\subset (-R_{\alpha_0},R_{\alpha_0})=\mathcal{I}_{\alpha_0}.
\]
Finally,
\[
\frac{\delta_2}{M_2}\le \frac{\delta_0}{2},
\]
so \(|\alpha-\alpha_0|<\delta_0\) on \(\mathcal D\). Therefore the uniform
Lemma~\ref{initial_data_decomposition} applies throughout \(\mathcal D\) with $\Theta=\overline T$.
Let \((\alpha,\kappa,T)\in\mathcal D\), and assume
\[
\|\mathbf f\|_{\mathcal H^k(\mathbb R)}
\le
\frac{\delta_2}{M_2^2}.
\]
Indeed, using Lemma \ref{initial_data_decomposition} with
\(\Theta=\overline T\), the choice of $K_{\mathrm{init}}$, \eqref{eq:M2_choice_prop45} and the fact that $\frac{\overline{T}}{r_0},\frac{\delta_2}{M_2}<1$ we have 
\begin{align*}
\|\mathbf U_{\alpha,\kappa,T}(\mathbf f)\|_{\mathcal H^k}
&\le
C_{\mathrm{sc}}\frac{\delta_2}{M_2^2}
+
C_{\mathrm{sym}}
\left(
|\kappa_0-\kappa|
+
|\alpha|\,|\lambda-1|
+
|\lambda-1|
+
|\alpha_0-\alpha|
\right)
\\
&
+
C_h\frac{\overline{T}}{r_0}
\left(
|\alpha-\alpha_0|+|\lambda-1|
\right)
+
C_r
\left(
|\alpha-\alpha_0|^2+|\lambda-1|^2
\right) \\
&\leq C_{\mathrm{sc}}\frac{\delta_2}{M_2^2} + C_{\mathrm{sym}}(\alpha_0+4)\frac{\delta_2}{M_2} + 2C_{h}\frac{\overline{T}}{r_0}\frac{\delta_2}{M_2}+ 2C_r\left(\frac{\delta_2}{M_2}\right)^2 \\
&\leq K_{\mathrm{init}}\frac{\delta_2}{M_2} \\
&\leq \frac{\delta_2}{M_1}
\end{align*}
Therefore Proposition~\ref{existence_of_v}, applied with
$\rho=\delta_2$, and 
$\mathbf v_0=\mathbf U_{\alpha,\kappa,T}(\mathbf f)$
yields a unique fixed point $\mathbf v_{\alpha,\kappa,T}\in B_{\delta_2}^{\mathcal X^k}$. The goal now is to find parameters $(\alpha^\star,\kappa^\star,T^\star)$ such that the correction vanishes. Since the correction belongs to the finite-dimensional symmetry space, it suffices to prove the existence of parameters such that the linear functional 
\begin{equation*}
    \ell:\rg(\mathbf{P}_{\alpha})\to \mathbb{R}, \qquad \mathbf{g}\mapsto \left\langle \mathbf{C}_{\alpha,T}(\mathbf{U}_{\alpha,\kappa,T}(\mathbf{f}),\mathbf{v}_{\alpha,\kappa,T})\, , \, \mathbf{g}\right \rangle_{\mathcal{H}^k}
\end{equation*}
is zero on a basis of $\rg(\mathbf{P}_{\alpha})$, such as the dual basis. For this we use a Brouwer fixed point argument. 
We shall first use that the map
\[
(\alpha,\kappa,T)\mapsto \mathbf v_{\alpha,\kappa,T}
\]
is continuous from \(\mathcal D\) into \(\mathcal X^k\). This follows from
the same contraction argument as in Proposition~\ref{existence_of_v}. Indeed,
the contraction constant is uniform on \(\mathcal D\). If
\(p_j=(\alpha_j,\kappa_j,T_j)\to p=(\alpha,\kappa,T)\), and if $\Phi_p$
denotes the fixed point map from Proposition~\ref{existence_of_v} with datum
\(\mathbf U_p(\mathbf f)\), then,
\[
\frac{1}{2}\|\mathbf v_{p_j}-\mathbf v_p\|_{\mathcal X^k}
\le
\|\Phi_{p_j}(\mathbf v_p)-\Phi_p(\mathbf v_p)\|_{\mathcal X^k}.
\]
The right-hand side tends to zero by the continuity of
\(\mathbf U_{\alpha,\kappa,T}(\mathbf f)\), \(\mathbf H_{\alpha,T}\), the
semigroup, and the finite-dimensional projections with respect to the
parameters. Hence
\[
\mathbf v_{p_j}\to\mathbf v_p
\quad\text{in }\mathcal X^k.
\]
For brevity, write
\[
\mathbf G_{\alpha,T}(\mathbf v)
:=
\mathbf H_{\alpha,T}(\mathbf v)+\mathbf N(\mathbf v).
\]
For \(n=1,2,3\), define \(F_n:\mathcal D\to\mathbb R\) by
\begin{align}
F_n(\alpha,\kappa,T)
:={}&
\left\langle
\mathbf P_\alpha\mathbf f_T,\mathbf g_\alpha^n
\right\rangle_{\mathcal H^k}
+
\left\langle
\mathbf P_\alpha\mathbf h_{\alpha,T},\mathbf g_\alpha^n
\right\rangle_{\mathcal H^k}
+
\left\langle
\mathbf P_\alpha\mathbf r(\alpha,T/T_0),\mathbf g_\alpha^n
\right\rangle_{\mathcal H^k}
\nonumber
\\
&+
\left\langle
\mathbf P_{0,\alpha}
\int_0^\infty
\mathbf G_{\alpha,T}(\mathbf v_{\alpha,\kappa,T})(\tau')\,d\tau',
\mathbf g_\alpha^n
\right\rangle_{\mathcal H^k}
\nonumber
\\
&+
\left\langle
\mathbf L_\alpha\mathbf P_{0,\alpha}
\int_0^\infty
(-\tau')
\mathbf G_{\alpha,T}(\mathbf v_{\alpha,\kappa,T})(\tau')\,d\tau',
\mathbf g_\alpha^n
\right\rangle_{\mathcal H^k}
\nonumber
\\
&+
\left\langle
\mathbf P_{1,\alpha}
\int_0^\infty
e^{-\tau'}
\mathbf G_{\alpha,T}(\mathbf v_{\alpha,\kappa,T})(\tau')\,d\tau',
\mathbf g_\alpha^n
\right\rangle_{\mathcal H^k}.
\label{eq:definition_of_Fn_prop45}
\end{align}
By the continuity just established, each \(F_n\) is continuous on \(\mathcal D\).
Define the continuous map $\mathcal F:\mathcal D\to\mathbb R^3$ by
\begin{equation}\label{eq:Brouwer_map_prop45}
\mathcal F(\alpha,\kappa,T)
:=
\left(
\alpha_0+F_1(\alpha,\kappa,T),
\,
\kappa_0+F_2(\alpha,\kappa,T)+\alpha F_3(\alpha,\kappa,T),
\,
T_0-T_0F_3(\alpha,\kappa,T)
\right).
\end{equation}
We now prove that
\[
\mathcal F(\mathcal D)\subset\mathcal D.
\]
Indeed, let \((\alpha,\kappa,T)\in\mathcal D\). Since \(T<\overline T\) and $\|\mathbf v_{\alpha,\kappa,T}(\tau)\|_{\mathcal H^k}
\le
\delta_2e^{-\omega\tau}$, estimates
\eqref{eq:H_bound_for_prop45} and \eqref{eq:N_bound_for_prop45} give
\[
\|\mathbf G_{\alpha,T}(\mathbf v_{\alpha,\kappa,T})(\tau)\|_{\mathcal H^k}
\le
C_H\frac{\overline{T}}{r_0}\delta_2 e^{-2\omega\tau}
+
C_N\delta_2^2 e^{-2\omega\tau}.
\]
Inserting this into
\eqref{eq:finite_dimensional_bound_tail_terms}, the last three terms of $F_n$ are estimated as
\begin{equation}\label{eq:tail_bound_prop45}
\begin{aligned}
&\left|
\left\langle
\mathbf P_{0,\alpha}
\int_0^\infty
\mathbf G_{\alpha,T}(\mathbf v_{\alpha,\kappa,T})(\tau')\,d\tau',
\mathbf g_\alpha^n
\right\rangle_{\mathcal H^k}
\right|
+
\left|
\left\langle
\mathbf L_\alpha\mathbf P_{0,\alpha}
\int_0^\infty
(-\tau')
\mathbf G_{\alpha,T}(\mathbf v_{\alpha,\kappa,T})(\tau')\,d\tau',
\mathbf g_\alpha^n
\right\rangle_{\mathcal H^k}
\right|
\\
&\quad+
\left|
\left\langle
\mathbf P_{1,\alpha}
\int_0^\infty
e^{-\tau'}
\mathbf G_{\alpha,T}(\mathbf v_{\alpha,\kappa,T})(\tau')\,d\tau',
\mathbf g_\alpha^n
\right\rangle_{\mathcal H^k}
\right|
\\
&\le
K_H\frac{\overline{T}}{r_0}\delta_2
+
K_N\delta_2^2.
\end{aligned}
\end{equation}
The first three terms in \(F_n\) are estimated using
\eqref{eq:finite_dimensional_bound_projected_data},
\eqref{eq:scaling_bound_for_prop45},
\eqref{eq:h_bound_for_prop45}, and \eqref{eq:r_bound_for_prop45}. Namely,
\[
\left|
\left\langle
\mathbf P_\alpha\mathbf f_T,\mathbf g_\alpha^n
\right\rangle_{\mathcal H^k}
\right|
\le
K_f\frac{\delta_2}{M_2^2},
\]
\[
\left|
\left\langle
\mathbf P_\alpha\mathbf h_{\alpha,T},\mathbf g_\alpha^n
\right\rangle_{\mathcal H^k}
\right|
\leq
K_h\frac{\overline T}{r_0}\frac{\delta_2}{M_2},
\]
and
\[
\left|
\left\langle
\mathbf P_\alpha\mathbf r\left(\alpha,\frac{T}{T_0}\right),\mathbf g_\alpha^n
\right\rangle_{\mathcal H^k}
\right|
\le
K_r\frac{\delta_2^2}{M_2^2}.
\]
Combining with \eqref{eq:tail_bound_prop45}, we obtain
for \(n=1,2,3\),
\begin{align}
|F_n(\alpha,\kappa,T)|
&\le
K_f\frac{\delta_2}{M_2^2}
+
K_h\frac{\overline{T}}{r_0}\frac{\delta_2}{M_2}
+
K_r\frac{\delta_2^2}{M_2^2}
+
K_H\frac{\overline{T}}{r_0}\,\delta_2
+
K_N\delta_2^2
\nonumber
\\
&=
\frac{\delta_2}{M_2}
\left(
\frac{K_f}{M_2}
+
K_h\frac{\overline{T}}{r_0}
+
K_r\frac{\delta_2}{M_2}
+
K_HM_2\frac{\overline{T}}{r_0}
+
K_NM_2\delta_2
\right).
\label{eq:Fn_bound_before_choices_prop45}
\end{align}
Using \eqref{eq:M2_choice_prop45}, \eqref{eq:Tbar_choice_prop45}, and
\eqref{eq:delta2_nonlinear_choice_prop45}, 
we obtain
\begin{equation*}
    |F_n(\alpha,\kappa,T)|\leq \frac{\delta_2}{4(2+\alpha_0)M_2} \qquad n=1,2,3.
\end{equation*}
Therefore, for each component of $\mathcal{F}$,
\[
|\alpha_0+F_1(\alpha,\kappa,T)-\alpha_0|
\le
\frac{\delta_2}{4(2+\alpha_0)M_2}
\le \frac{\delta_2}{M_2}.
\]
Second,
\[
|T_0-T_0F_3(\alpha,\kappa,T)-T_0|
=
T_0|F_3(\alpha,\kappa,T)|
\le
T_0\frac{\delta_2}{4(\alpha_0+2)M_2}
\le
T_0\frac{\delta_2}{M_2}.
\]
Third, using \(|\alpha|\leq \alpha_0 +1\),
\begin{align*}
&|\kappa_0+F_2(\alpha,\kappa,T)+\alpha F_3(\alpha,\kappa,T)-\kappa_0|
\le
|F_2(\alpha,\kappa,T)|
+
|\alpha|\,|F_3(\alpha,\kappa,T)|
\\
&
\le
(2+\alpha_0)\frac{\delta_2}{4(2+\alpha_0)M_2}
\le \frac{\delta_2}{M_2}.
\end{align*}
Therefore
$
\mathcal F(\mathcal D)\subset\mathcal D
$
and so Brouwer's fixed point theorem gives
\[
(\alpha^\star,\kappa^\star,T^\star)\in\mathcal D
\]
such that
\[
\mathcal F(\alpha^\star,\kappa^\star,T^\star)
=
(\alpha^\star,\kappa^\star,T^\star).
\]
The fixed point identities are then
\begin{equation}\label{eq:fixed_point_identities_prop45}
(\alpha_0-\alpha^\star)
+
F_1(\alpha^\star,\kappa^\star,T^\star)
=0,
\end{equation}
\begin{equation}\label{eq:fixed_point_identity_T_prop45}
\left(\frac{T^{\star}}{T_0}-1\right)
+
F_3(\alpha^\star,\kappa^\star,T^\star)
=0,
\end{equation}
and
\begin{equation}\label{eq:fixed_point_identity_kappa_prop45}
(\kappa_0-\kappa^\star)
-
\alpha^\star\left(\frac{T^{\star}}{T_0}-1\right)
+
F_2(\alpha^\star,\kappa^\star,T^\star)
=0.
\end{equation}

By the initial-data decomposition
\eqref{eq:uniform_initial_decomposition_for_prop45}, the definition of the
correction operator, and the definition of \(F_n\), the identities
\eqref{eq:fixed_point_identities_prop45},
\eqref{eq:fixed_point_identity_T_prop45}, and
\eqref{eq:fixed_point_identity_kappa_prop45} are precisely
\[
\left\langle
\mathbf C_{\alpha^\star,T^\star}
\left[
\mathbf U_{\alpha^\star,\kappa^\star,T^\star}(\mathbf f),
\mathbf v_{\alpha^\star,\kappa^\star,T^\star}
\right],
\mathbf g_{\alpha^\star}^n
\right\rangle_{\mathcal H^k}
=0,
\qquad n=1,2,3.
\]
which concludes the proof. 
\end{proof}
\begin{remark}
    In the case of $n=3$, for which we have the explicit solution $v^{(3)}_{\alpha,\kappa,T}$, then $\mathbf{H}_{\alpha,T} = \mathbf{h}_{\alpha,T}\equiv 0$. The proof then simplifies drastically and we may take the final time ceiling as $\overline{T}=r_0$. 
\end{remark}
\subsection{Proof of Theorem \ref{thm:stability}}
\begin{proof}
    Fix $n\geq 2$, $r_0>0$ and $\omega_0 \in (0,1)$. Let $\alpha_0>0$, $k\geq \lceil c_{\alpha_0} \rceil +1$, $\gamma \in (0,\omega_0)$ so that $\omega = \omega_0 - \gamma$, and $\kappa_0 \in \mathbb{R}$. By Proposition \ref{prop:finite_dim_reduction}, there exist $M_2>1$ and $0<\overline{T}<\overline{T}_2$, where $\overline{T}_2$ is the constant from Proposition \ref{existence_of_v} such that, after taking $T_0 \in (0,\overline{T})$, there exist $\delta_2=\delta_2(T_0)>0$, such that if we take $(f,g) \in H_{\text{rad}}^{k+1}(\mathbb{R}^n)\times H_{\text{rad}}^k(\mathbb{R}^n)$ satisfying $\|(f,g)\|_{H^{k+1}(\mathbb{R}^n)\times H^k(\mathbb{R}^n)}\leq \frac{\delta_2}{M_2^2}$, then there exists a unique mild solution $\mathbf{v}_{\alpha^{\star},\kappa^{\star},T^{\star}}=(\mathbf{v}_{\alpha^{\star},\kappa^{\star},T^{\star}}^{(1)},\mathbf{v}_{\alpha^{\star},\kappa^{\star},T^{\star}}^{(2)})$ to \eqref{eq:final_prop_mild_sol} which is also a classical solution to \eqref{eq:abstract_cauchy_prob_v}. If we let 
    \begin{align*}
        v_{\alpha^{\star},\kappa^{\star},T^{\star}}(r,t) =& U_{\alpha^{\star},\kappa^{\star}}\left(-\log\left(1-\frac{t}{T^{\star}}\right),\frac{r-r_0}{T^{\star}-t} \right) - \frac{n-1}{2}\log\left(\frac{r}{r_0}\right) + \zeta_{\alpha^{\star},T^{\star}}\left(-\log\left(1-\frac{t}{T^{\star}}\right),\frac{r-r_0}{T^{\star}-t} \right) \\
        &+ \mathbf{v}_{\alpha^{\star},\kappa^{\star},T^{\star}}^{(1)}\left(-\log\left(1-\frac{t}{T^{\star}}\right),\frac{r-r_0}{T^{\star}-t} \right),
    \end{align*}
    then $v_{\alpha^{\star},\kappa^{\star},T^{\star}}$ gives the unique solution to the Cauchy problem for \eqref{Nd-equation} in $\mathcal{A}_{1}(r_0,T^{\star})$ with initial data 
    \begin{align*}
        v_{\alpha^{\star},\kappa^{\star},T^{\star}}(x,0) = v_{\alpha_0,\kappa_0,T_0}(x,0) + f(x), \qquad \partial_t v_{\alpha^{\star},\kappa^{\star},T^{\star}}(x,0) = \partial_t v_{\alpha_0,\kappa_0,T_0}(x,0) + g(x) \qquad (x\in A_{T^{\star}}(r_0)).
    \end{align*}
    Moreover, from Proposition \ref{prop:finite_dim_reduction} we infer the bounds for all $0<t<T^{\star}$
\[
(T^\star-t)^{-\frac12+j}
\|\partial_r^j\eta(\cdot,t)\|_{L^2(B_{T^\star-t}(r_0))}
\lesssim \delta_2 \left(1-\frac{t}{T^{\star}}\right)^\omega,
\qquad j=0,1,\dots,k+1,
\]
and
\[
(T^\star-t)^{\frac12+j}
\|\partial_r^j\partial_t\eta(\cdot,t)\|_{L^2(B_{T^\star-t}(r_0))}
\lesssim \delta_2 \left(1-\frac{t}{T^{\star}}\right)^\omega,
\qquad j=0,1,\dots,k. 
\]
\end{proof}
\appendix 
\section{Non-existence of generalised eigenfunctions}
\begin{lemma}\label{Appendix_lemma_eigenfunctions}
    Given $\alpha>0$, $k\geq \lceil c_{\alpha}\rceil+1$, there are no solutions $\mathbf{v}\in \mathcal{H}^k(-1,1)$ to $\mathbf{L}_{\alpha}\mathbf{v}=\mathbf{g}_{0,\alpha}$.
\end{lemma}
\begin{proof}
    Since $c_{\alpha}>\sqrt{1+\alpha}$, we actually prove that for $k\geq \lceil \sqrt{1+\alpha} +\frac{1}{2}\rceil $ there are no solutions to the equation, which settles our claim. $\mathbf{L}_{\alpha}\mathbf{v}= \mathbf{g}_{0,\alpha}$ is the following ODE system
\begin{align}\label{No_more_generalised_eigenfunctions}
    \begin{cases}
        -y\partial_y v_1 + v_2 = g_{0,\alpha,1} \\
        \partial_{yy} v_1 -\frac{2\alpha}{\sqrt{1+\alpha}+y}\partial_y v_1 -v_2 -y\partial_y v_2 = g_{0,\alpha,2},
    \end{cases}
    \end{align}
    and we suppose for contradiction that there exists a solution $(v_1,v_2)^{\top}\in \mathcal{H}^k(-1,1)$ to \eqref{No_more_generalised_eigenfunctions} for $k\geq \lceil \sqrt{1+\alpha} +\frac{1}{2}\rceil$. By \eqref{No_more_generalised_eigenfunctions} $v_1$ solves
    \begin{equation*}
        (1-y^2)\partial_{yy}v_1 - \left(\frac{2\alpha}{\sqrt{1+\alpha}+y}+2y\right)\partial_y v_1 = y\partial_y g_{0,\alpha,1} + g_{0,\alpha,2} + g_{0,\alpha,1} \coloneqq -g(y).
    \end{equation*}
    Computing the integrating factor $\mu(y)$ we have
    \begin{equation}\label{eq:solution_with_int_fact}
        \partial_y v_1(y) = c\mu(y)^{-1} + \mu(y)^{-1}\int_{y}^{1}\frac{\mu(z)g(z)}{1-z^2}\,dz
    \end{equation}
    where \begin{equation*}
        \mu(y) \coloneqq (\sqrt{1+\alpha}+y)^2\left(\frac{1-y}{1+y}\right)^{\sqrt{1+\alpha}}.
    \end{equation*}
    Since $\|\mu(y)^{-1}\|_{L^2(r,1)}^2 \sim (1-r)^{-2\sqrt{1+\alpha}+1}$ for $r\sim 1^{-}$, and noting that the integral term is bounded near $y=1$ (which can be shown via l'Hôpital's rule), this implies $c=0$ since we supposed $\partial_y v_1 \in L^2(-1,1)$.

    In the rest of the analysis, we now denote $p\coloneqq \sqrt{1+\alpha}>1$. We wish to investigate the regularity of the solution given by \eqref{eq:solution_with_int_fact} (with $c=0$) at $y=-1$, and thus put terms which are regular at $y=-1$ to the left-hand side. Indeed, by definition of $g(z)$, \eqref{eq:solution_with_int_fact} becomes
    \begin{align}\label{Terms_one_side}
        (p+y)^2(1-y)^p \partial_y v_1(y) = (1+y)^p\int_{y}^{1}\frac{G(z)}{(1+z)^{1+p}}\,dz
    \end{align}
    where 
    \begin{equation}
        \label{eq_for_G}
        G(z) \coloneqq (1-z)^{p-1}\left(-z^2 + \left(\frac{p^2-1}{2p}\right)z + \frac{p^2+1}{2}-(p+z)^2\log(p+z)\right).
    \end{equation}
To prove our claim we will show that \eqref{Terms_one_side} takes the form
\[
(p+y)^2(1-y)^p \partial_y v_1(y)=h(y)+C_p\,\psi_p(y)
\]
with $h$ smooth near $y=-1$, the following two cases for $\psi_p$
\[
\psi_p(y):=
\begin{cases}
(1+y)^p, & p\notin \mathbb N,\\[1mm]
(1+y)^p\log(1+y), & p\in\mathbb N,
\end{cases}
\]
and crucially the coefficient \(C_p\neq0\).
This will yield our claim. Indeed, letting $p=\lfloor p \rfloor + \sigma \coloneqq m +\sigma$, then for the case $p\in \mathbb{N}$, $\partial_y^{p+1}((1+y)^p\log(1+y))\sim (1+y)^{-1}$ which is not in $L^2(-1,-\frac{1}{2})$, thus $\partial_y v_1 \notin H^{p+1}(-1,-\frac{1}{2})$. 
For $p\notin \mathbb{N}$ and $\sigma \in (0,\frac{1}{2}]$ then $\partial_y^{m+1}(1+y)^p \sim (1+y)^{\sigma -1}$ which is not in $L^2(-1,-\frac{1}{2})$, so $\partial_y v_1 \notin H^{m+1}(-1,-\frac{1}{2})$. Then for $\sigma \in (\frac{1}{2},1)$, $\partial_y^{m+2}(1+y)^p \sim (1+y)^{\sigma -2}$ which is not in $L^2(-1,-\frac{1}{2})$ so $\partial_y v_1 \notin H^{m+2}(-1,-\frac{1}{2})$. All cases imply that $\partial_y v_1 \notin H^{\lceil p +\frac{1}{2}\rceil}(-1,1)$. 
\\
\\
Firstly, note that $G$ is $C^\infty$ near $z=-1$. Thus for any $N\in\mathbb N$ we Taylor expand around $z=-1$
\begin{equation}\label{eq:TaylorG}
G(z)=\sum_{n=0}^N \frac{1}{n!}G^{(n)}(-1)(1+z)^n + (1+z)^{N+1}R_N(z)
\end{equation}
with $R_N$ bounded on $[-1,1]$ and $C^\infty$ near $z=-1$. We now consider the following cases, $p \in \mathbb{N}\setminus \{1\}$ and $p\notin \mathbb{N}$.

Case 1: $p\in\mathbb N\setminus\{1\}$.
We take a Taylor expansion  of $G$ up to order $p$, i.e. $N=p$ in \eqref{eq:TaylorG}. Then inserting the expansion into \eqref{Terms_one_side} we compute for $n\in \{0,\dots, p-1\}$,
\[
\int_y^1 (1+z)^{n-p-1}\,dz=\frac{2^{n-p}-(1+y)^{n-p}}{n-p},
\]
while for $n=p$ one has
\[
\int_y^1 \frac{dz}{1+z}=\log\left(\frac{2}{1+y}\right).
\]
Thus
\begin{equation}\label{eq:wm-log}
(p+y)^2(1-y)^p \partial_y v_1(y)=C_p(1+y)^p\log(1+y)+ H(y)
\end{equation}
where $H$ is smooth at $y=-1$ (in fact $C^\infty([-1,1])$) and 
\begin{equation}\label{eq:Cm}
C_p = -\frac{G^{(p)}(-1)}{p!}=\frac{(-1)^{p}(p+1)^{p+1}}{2p(p-1)^{p-1}}.
\end{equation}
We therefore see that $C_p\neq 0$. 

Case 2: \(p\notin\mathbb N\). Let \(m:=\lfloor p\rfloor\), we translate $y=-1$ to $0$ via the change of variables $t\coloneqq 1+y$ and define
\[
W(t):=(p+t-1)^2(2-t)^p\,\partial_y v_1(t-1)
\]
so that we now consider $t\to 0$. Then \eqref{Terms_one_side} becomes
\begin{equation}\label{integral_form_of_W}
W(t)=t^p\int_t^2 \frac{G(s-1)}{s^{p+1}}\,ds,
\end{equation}
which satisfies the ODE 
\begin{equation}\label{ode_w_appendix}
    tW'(t)-pW(t)=-G(t-1)
\end{equation}
or equivalently with the explicit expression for $G$  
\[
tW'(t)-pW(t)
=-(2-t)^{p-1}\Bigl(
-(t-1)^2+\frac{p^2-1}{2p}(t-1)+\frac{p^2+1}{2}-(p+t-1)^2\log(p+t-1)
\Bigr).
\]
For convenience, we eliminate the logarithmic term by defining a new function $b(t)$ via 
\begin{equation}\label{eq:def_of_small_b}
    W(t)=(2-t)^p\left[\left(-t-\frac{(p-1)^2}{2p}\right)\log(p+t-1)+b(t)\right]
\end{equation}
so that \eqref{ode_w_appendix} simplifies to
\begin{equation}\label{eq:for_little_b}
    t(2-t)b'(t)-2pb(t)
=
-\frac{(p-1)(p+1)(p^2+2pt-2p+1)}{2p(p+t-1)}.
\end{equation}
We make a further change of variables 
\[
x:=\frac{t}{2-t}\in [0,\infty),
\qquad\qquad
t=\frac{2x}{1+x} \in [0,2]
\]
and denote $B(x):=b(t(x))$. The equation \eqref{eq:for_little_b} then becomes
\[
xB'(x)-pB(x)
=
-\frac{(p-1)(p+1)^2}{4p}
+
\frac{(p-1)(p+1)}{2p}\frac{1}{1+\frac{p+1}{p-1}x}.
\]

Firstly, observe from \eqref{integral_form_of_W} that \(W(t)\to0\) as \(t\to2\), which implies that $(2-t)^p b(t) \to 0$ as $t\to 2$, which further implies $\frac{B(x)}{x^p}\to 0$ as \(x\to\infty\). Together with the fact that $Cx^p$ solves the homogeneous equation, variation of constants gives
\[
B(x)
=
\frac{(p-1)(p+1)^2}{4p^2}
-\frac{(p-1)(p+1)}{2p}\,
x^p\int_x^\infty \frac{\xi^{-p-1}}{1+\frac{p+1}{p-1}\xi}\,d\xi.
\]
Substituting this into \eqref{eq:def_of_small_b}, we obtain
\[
W(t)
=
\text{\rm(a \(C^\infty\)-term near \(t=0\))}
-\frac{(p-1)(p+1)}{2p}
t^p\int_{x(t)}^\infty \frac{\xi^{-p-1}}{1+\frac{p+1}{p-1}\xi}\,d\xi
\]
and in the $y$ variable, noting that $x=\frac{1+y}{1-y}$
\[
(p+y)^2(1-y)^p\partial_y v_1(y)
=
\text{\rm(a \(C^\infty\)-term near \(y=-1\))}
-\frac{(p-1)(p+1)}{2p}(1-y)^p
x^p\int_x^\infty \frac{\xi^{-p-1}}{1+\frac{p+1}{p-1}\xi}\,d\xi
\]
As $y\to -1$, $x\to 0$, so it remains to analyse the small-$x$ asymptotic of 
\[
x^p\int_x^\infty \frac{\xi^{-p-1}}{1+\frac{p+1}{p-1}\xi}\,d\xi
=
\Bigl(\frac{p+1}{p-1}x\Bigr)^p
\int_{\frac{p+1}{p-1}x}^\infty \frac{u^{-p-1}}{1+u}\,du.
\]
Since $p\notin \mathbb{N}$, we use the identities $\frac1{1+u}=\sum_{j=0}^{m}(-1)^j u^j+(-1)^{m+1}\frac{u^{m+1}}{1+u}$ and $\int_0^\infty \frac{u^{m-p}}{1+u}\,du
=
\frac{(-1)^m\pi}{\sin(\pi p)}$ to obtain
\[
\Bigl(\frac{p+1}{p-1}x\Bigr)^p
\int_{\frac{p+1}{p-1}x}^\infty \frac{u^{-p-1}}{1+u}\,du
=
\sum_{j=0}^{m}\frac{(-1)^j}{p-j}\Bigl(\frac{p+1}{p-1}x\Bigr)^j
-\frac{\pi}{\sin(\pi p)}\Bigl(\frac{p+1}{p-1}x\Bigr)^p
+\mathcal{O}(x^{m+1}).
\]
Hence
\[
(p+y)^2(1-y)^p\partial_y v_1(y)
=
\text{\rm(a \(C^\infty\)-term near \(y=-1\))}
+
\frac{\pi (p+1)^{p+1}}{2p(p-1)^{p-1}\sin(\pi p)}(1+y)^p
+
\mathcal{O}\bigl((1+y)^{m+1}\bigr).
\]
Since \(p\notin\mathbb N\), we have \(\sin(\pi p)\neq0\), so the coefficient of \((1+y)^p\) is nonzero. 
\end{proof}
\bibliographystyle{plain}
\bibliography{Radial_symmetry_refs}
\end{document}